\documentclass[10pt, reqno]{amsart}
\usepackage{graphicx} 

\usepackage{amsmath}
\usepackage{amssymb}
\usepackage{tikz-cd}
\usepackage{amsthm}
\usepackage{mathrsfs}
\usepackage{fullpage}
\usepackage{float}
\usepackage{wasysym}
\usepackage{keytheorems}
\usepackage[capitalise]{cleveref}

\newcommand{\R}{\mathbb{R}}

\newcommand{\Z}{\mathbb{Z}}
\newcommand{\N}{\mathbb{N}}

\newcommand{\h}{\mathbb{H}}

\newcommand{\Ha}{\mathbb{H}}

\DeclareMathOperator{\HNN}{HNN}
\DeclareMathOperator{\up}{up}
\DeclareMathOperator{\Stab}{Stab}

\DeclareMathOperator{\Nbhd}{Nbhd}

\DeclareMathOperator{\Aut}{Aut}
\DeclareMathOperator{\Axis}{Axis}

\DeclareMathOperator{\Isom}{Isom}
\DeclareMathOperator{\BS}{BS}
\DeclareMathOperator{\id}{id}
\DeclareMathOperator{\ev}{ev}
\DeclareMathOperator{\tr}{tr}
\DeclareMathOperator{\SL}{SL}

\DeclareMathOperator{\CAT}{CAT}
\DeclareMathOperator{\Cay}{Cay}

\newcommand{\inv}{^{-1}}

\usepackage[top=1.2in,bottom=1.2in,left=1.2in,right=1.2in]{geometry}

\newtheorem{thm}{Theorem}[subsection]
\renewcommand{\thethm}{%
	\ifnum\value{subsection}>0
	\thesubsection
	\else
	\thesection
	\fi
	.\arabic{thm}%
}
\newtheorem{lemma}[thm]{Lemma}
\newtheorem{cor}[thm]{Corollary}
\newtheorem{prop}[thm]{Proposition}
\newtheorem{question}{Question}

\theoremstyle{definition}
\newtheorem{Def}[thm]{Definition}

\newkeytheorem{Theorem}
\newkeytheorem{Corollary}

\title{Groups Acting on Horocyclic Products}
\author{Noah Caplinger and Daniel Levitin}
\date{\today}

\usepackage{subfiles} 

\begin{document}

\begin{abstract}
    Horocyclic products are a well-studied class of metric spaces that provide models for various solvable Lie groups, Baumslag-Solitar groups, and Lamplighter groups. Let $G$ act geometrically on a horocyclic product $X \bowtie Y$ of $\CAT(-\kappa)$ spaces $X,Y$. We show that every such group is either an ascending HNN extension of a finitely-generated virtually nilpotent group, or else is not finitely presented, depending on the connectivity of the visual boundary of $X\bowtie Y$.
\end{abstract}

\maketitle

\vspace{-1cm}

\section{Introduction}





The \textit{horocyclic product} $X\bowtie Y$ of two Gromov-hyperbolic spaces $X$ and $Y$ is the subspace $$X\bowtie Y = \{(x,y) \in X\times Y \mid h_X(x)=-h_Y(y)\},$$ where $h_X$ and $h_Y$ are \textit{height functions} (the negative of Busemann functions), equipped with one of a family of \textit{admissible metrics} arising from a choice of norm on $\R^2$ (See Definitions \ref{def:horocyclic_prod} and \ref{def:MetricOnHorocyclicProduct}). Horocyclic products appear naturally in geometric group theory as sources of \textit{geometric} (properly discontinuous and cocompact) actions of well-known groups. Let $T_n$ denote the regular $(n+1)$-valent tree and $\h^2$ the hyperbolic plane. There are geometric actions \[(\Z/n\Z)\wr \Z\curvearrowright T_n\bowtie T_n, \quad BS(1,n)\curvearrowright \h^2\bowtie T_n, \quad \text{and} \quad \Z^2 \rtimes_A\Z\curvearrowright \h^2\bowtie\h^2, \]
of the lamplighter group $(\Z/n\Z)\wr\Z$, the solvable Baumslag-Solitar group $BS(1,n)=\langle a, t\,|\, tat^{-1}=a^n\rangle$, and $\Z^2\rtimes_A \Z$, for $A \in \SL_2(\Z)$ with $|\tr(A)| > 2$. The group $\Z^2\rtimes_A\Z$ acts isometrically on the Lie group Sol, which carries a left-invariant Riemannian metric with which it is isometric to (some choice of admissible metric on) $\h^2\bowtie \h^2$.  

These three groups display well-known algebraic properties: the lamplighter groups are the prototypical finitely-generated groups that are not finitely-presented, and both $\BS(1,n)$ and $\Z^2 \rtimes_A \Z$ are ascending HNN extensions of abelian groups, the latter being a semidirect product. More generally, there is a class of fundamental groups of mapping of tori of certain Anosov (or expanding) maps on nilmanifolds---which are ascending HNN extensions of nilpotent groups---that act geometrically on horocyclic products---see Subsection \ref{subsec:Construct_millefeuille_example} for details. 

Our main result says that these algebraic properties of groups acting on horocyclic products can be read off from the \textit{visual boundary} of $X\bowtie Y$. In \cite{Ferragut:Visual_Boundary}, Ferragut showed that $X\bowtie Y$ has a visual boundary $\partial (X\bowtie Y)$ independent of choice of metric, and that $\partial (X\bowtie Y)$ is a disjoint union between the \textit{upper boundary} $\partial^u Y=\partial Y\setminus \{\infty_Y\}$, and the \textit{lower boundary} $\partial_lX=\partial X\setminus\{\infty_X\}$, where the boundary points $\infty_Y$ and $\infty_X$ are determined by the Busemann functions defining $X\bowtie Y$. See Section \ref{subsec:FerragutVisualBoundaries} for precise definitions and discussion of Ferragut's theorems.

\begin{Theorem} [manual-num=A, label=thm:boundary_implies_algebra]
    Let $X$ and $Y$ be proper, geodesically-complete $\CAT(-\kappa)$ spaces equipped with height functions based at non-isolated points at infinity. Let $G$ be a finitely-generated group acting geometrically on $X\bowtie Y$ with metric $d_{\bowtie}$ arising from a monotone norm.

    \begin{enumerate}
        \item If neither $\partial_l X$ nor $\partial^u Y$ is connected, then $G$ is not finitely presented.

        \item If exactly one of $\partial_l X$ and $\partial^uY$ is connected, then there is a virtually nilpotent, finitely-generated group $H$ and an injection $f:H\to H$ satisfying $1 < [H:f(H)] < \infty$ so that
        \[G\cong \HNN(H,f) = \langle H, t \mid tht\inv = f(h) \rangle.\]

        \item If both $\partial_l X$ and $\partial^u Y$ are connected, then there is a virtually nilpotent, finitely-generated group $H$ and a finite-index subgroup $K \subset G$ so that $K \cong H\rtimes \Z$.
    \end{enumerate}
\end{Theorem}

One might hope for a stronger conclusion to Part 1, perhaps that $G$ is somehow related to a lamplighter group. This would appear to be difficult, even in the $T_n\bowtie T_n$ case: Cornulier-Fisher-Kashyap \cite{CFK} characterized cocompact, closed subgroups of $\Isom(T_n\bowtie T_n)$, and in particular found lattices not virtually isomorphic to lamplighter groups. 

We actually prove a stronger (and much more verbose) result for Parts 2 and 3 of \cref{thm:boundary_implies_algebra} which essentially says that $H$ is nearly a uniform lattice in a nilpotent Lie group $N_1\times N_2$, and the map $f$ is nearly an Anosov map which expands $N_1$ and contracts $N_2$---see Corollaries \ref{cor:thmA_part2_STRONG} and \ref{cor:thmA_part3_STRONG}. As we show in Subsection \ref{subsec:Construct_millefeuille_example}, HNN extensions of such Anosov maps of lattices $\Gamma \subset N_1\times N_2$ act on horocyclic products, so this result is essentially sharp.

We obtain \cref{thm:boundary_implies_algebra} by upgrading the action of $G$ on $X\bowtie Y$ to a new action of a finite-index subgroup of $G$ on a horocyclic product of \textit{millefeuille spaces}, which we define below. We will later describe how this action on $Z\bowtie W$ allows us to read off algebraic properties of the group $G$.

\begin{Def}[Heintze group]
    A \textit{Heintze group} is a group of the form \[N\rtimes_\alpha \R,\] where $N$ is a simply connected nilpotent Lie group, and $\alpha:\R \to \Aut(N)$ a one parameter family of expanding maps. This means the eigenvalues of $\alpha(t):T_1N \to T_1N$ all lie outside the unit circle for $t>0$.
\end{Def}



Heintze groups will be implicitly endowed with some left-invariant Riemannian metric. Heintze \cite{Heintze} showed that there is such a metric with negative sectional curvature, and that all negatively curved homogeneous manifolds are isometric to some metric on a Heintze group. For example, $\Ha^n$ is isometric to a choice of metric on $\R^{n-1} \rtimes_{e^t} \R$.

\begin{Def}[Millefeuille space]
    Let $X = N\rtimes \R$ be a Heintze group endowed with a left-invariant Riemannian metric of negative curvature, and the height function $h_X(n,t) = t$, and let $k \geq 1$ an integer. The \textit{millefeuille space} $X[k]$ is $$X[k] = \{(x,v) \in X \times T_k \mid h_X(x) = h_{T_k}(v)\},$$ where $T_k$ denotes the $(k+1)$-regular tree. 
\end{Def}

For every left-invariant Riemannian distance on $X$ with negative curvature, there is an associated metric on $X[k]$ which is $\CAT(-\kappa)$. All millefeuille spaces will have such a metric in the sequel. Note also that our definition is somewhat more restrictive than the original definition due to \cite{CCMT}, where the space $X$ is allowed to be any $\CAT(-\kappa)$ space. 

\begin{figure}
    \centering\includegraphics[scale=0.6]{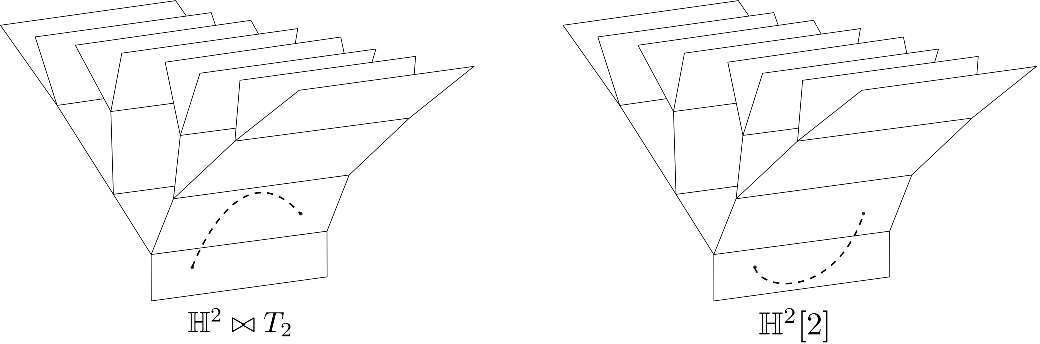}
    \caption{\small The horocyclic product $\h^2\bowtie T_2$ (left) and the millefeuille space $\h^2[2]$ (right) with dotted geodesics. Both are trees of hyperbolic planes, but the $\h^2$-leaves are embedded with opposite heights.\vspace{-.1cm}}
    \label{figure:Difference}
\end{figure}

Note that millefeuille spaces are not horocyclic products---they are defined by the equation $h_X = h_Y$, rather than $h_X = -h_Y$. The difference is illustrated in Figure \ref{figure:Difference}. Note also that in the cases where either $X$ or $T_k$ is a line, the millefeuille space $X[k]$ is (respectively) a tree or a Heintze group. Such millefeuille spaces will be termed \textit{degenerate}.

\begin{Theorem}[manual-num=B, label=thm:Upgrade_to_Millefeuille]
    Let $X\bowtie Y$ be a horocyclic product of proper, geodesically-complete $\CAT(-\kappa)$ spaces equipped with height functions about non-isolated points at infinity, and let $G$ be a finitely-generated group acting geometrically on it. Then there are (possibly degenerate) millefeuille spaces $Z,W$, quasi-isometric to $X$ and $Y$ respectively, so that $Z\bowtie W$ admits a geometric action of a finite-index subgroup of $G$. 
\end{Theorem}

The millefeuille spaces $Z$ and $W$ in Theorem \ref{thm:Upgrade_to_Millefeuille} arise from the structure theory of amenable hyperbolic groups due to Caprace, Cornulier, Monod, and Tessera \cite{CCMT}. A consequence of this theory is that if the stabilizer $\Isom(X)_{\infty_X}\subset\Isom(X)$ of the point $\infty_X$ acts cocompactly on $X$, then there is a millefeuille space $Z$ quasi-isometric to $X$ and a geometric action of $\Isom(X)_{\infty_X}$ on $Z$. If the geometric action $G \curvearrowright X\bowtie Y$ gives a cocompact action $G \curvearrowright X$ by $\Isom(X)_{\infty_X}$, we could apply this fact to produce a millefeuille space $Z$. One therefore might hope that $\Isom(X\bowtie Y)$ consists of maps $f$ satisfying:

\begin{enumerate}
\item [1.]$f$ is a product map of isometries $f_X$ and $f_Y$ of $X$ and $Y$.
\item [2.]$f_X$ and $f_Y$ fix $\infty_X$ and $\infty_Y$ respectively.
\item [3.]$f_X$ and $f_Y$ have opposite height changes (i.e. $f_X$ increases heights in $X$ by the same amount that $f_Y$ decreases heights in $Y$).
\end{enumerate}

\noindent We will denote this group of isometries by $\Isom(X)_\infty \times_{-h} \Isom(Y)_\infty$ (the fibered product along opposite height functions). Conditions 1-3 together imply that $f$ stabilizes the locus $X\bowtie Y\subset X\times Y$ and acts isometrically for any choice of admissible metric. It turns out that the guess $\Isom(X\bowtie Y)=\Isom(X)_\infty \times_{-h} \Isom(Y)_\infty$ is overly optimistic for two reasons.

The first reason is simple. If $X\cong Y$, then the map $(x, y)\mapsto(y,x)$ is an isometry for many choices of metric. So we must restrict our focus to a finite-index subgroup of $\Isom(X\bowtie Y)$, and thus to the action of a finite-index subgroup $K\subset G$, in order to have any chance to find factor maps $f_X$ and $f_Y$.

The second issue with the guess is technical. For a finite-index subgroup of $\Isom(X\bowtie Y)$, we can indeed extract factor maps $f_X$ and $f_Y$ so that $f=(f_X, f_Y)$, but $f_X$ and $f_Y$ may not be isometries of $X$ and $Y$. However, $f_X$ and $f_Y$ do turn out to be isometries of \textit{different metrics} $d'_X$ and $d'_Y$ (see Definition \ref{def:d'_initial}). These new metric spaces $X'=(X, d'_X)$ and $Y'=(Y, d'_Y)$  are quasi-isometric to $X$ and $Y$, and have isometry groups that contain those of $X$ and $Y$, but may in principle be larger. We prove that a finite-index subgroup of $\Isom(X\bowtie Y)$, must satisfy the weaker condition 

\begin{enumerate}
\item[1$'$.]$f$ is a product map of isometries $f_X$ and $f_Y$ of $X'$ and $Y'$.
\end{enumerate}

\noindent as well as Conditions 2 and 3, above.

\begin{Theorem}[manual-num=C, label=thm:IsomGroupOfHorocyclicProduct]
    Let $X\bowtie Y$ satisfy the hypotheses of \cref{thm:boundary_implies_algebra}. If $(x,y)\mapsto(y,x)$ is an isometry of $X\bowtie Y$, then \[ \Isom(X)_\infty \times_{-h} \Isom(Y)_\infty \subset \Isom(X\bowtie Y)\subset \big{(}\Isom(X')_{\infty} \times_{-h} \Isom(Y')_{\infty}\big{)}\rtimes\Z/2\Z, \]\noindent where a generator of the $\Z/2\Z$-factor can be taken to be the map $(x,y)\mapsto(y,x)$. Otherwise \[ \Isom(X)_\infty \times_{-h} \Isom(Y)_\infty \subset \Isom(X\bowtie Y)\subset \Isom(X')_{\infty} \times_{-h} \Isom(Y')_{\infty}. \]
\end{Theorem}

For a horocyclic product $X\bowtie Y$, we will say $X\cong Y$ if the two are isometric by an isometry sending $\infty_X$ to $\infty_Y$. We extend this notation to the spaces $X'$ and $Y'$ as well. Sufficient conditions for $(x,y)\mapsto(y,x)$ to be an isometry are that $X'\cong Y'$ and that the norm defining the admissible metric is symmetric. As a special case, if the norm defining the metric is $L^\infty$, then $X\cong X'$ and $Y\cong Y'$, and so we can compute $\Isom(X\bowtie Y)$ exactly. Besides metrics arising from the $L^\infty$ norm, some other isometry groups have been calculated, such as $\h^2\bowtie \h^2$ with metric arising from (a normalization of) $L^2$.
    
\begin{Corollary}[manual-num=D]
\label{cor:Isom_cor}
    When the norm defining the horocyclic product metric is $L^\infty$, the isometry group $\Isom(X\bowtie Y)$ is isomorphic to  either \[ \Isom(X)_{\infty} \times_{-h} \Isom(Y)_{\infty} \quad \text{or} \quad \big{(}\Isom(X)_{\infty} \times_{-h} \Isom(Y)_{\infty}\big{)}\rtimes\Z/2\Z ,\] where the latter occurs if and only if $X \cong Y$.
\end{Corollary}

The isometry groups in Corollary \ref{cor:Isom_cor} have already been computed in some special cases: Bartholdi--Neuhauser--Woess \cite{HorocyclicProductsOfTrees} computed $\Isom(T_n \bowtie T_n)$ and $\Isom(\Ha^2\bowtie T_n)$ was computed independently by Farb--Mosher \cite{FM2} and Bendikov--Saloff-Coste--Salvatori--Woess \cite{IsomTreebolic} (see also \cite{Caplinger} Proposition 3.2). The first author \cite{Caplinger} also classified lattices in $\Isom(\Ha^2\bowtie T_n)$, thereby proving \cref{thm:boundary_implies_algebra} in the special case $X\bowtie Y = \Ha^2\bowtie T_n$.

\vspace{.4cm}

\noindent \textbf{Proof outline.} We will now outline the proofs of Theorem \ref{thm:boundary_implies_algebra} and Theorem \ref{thm:Upgrade_to_Millefeuille}, assuming \cref{thm:IsomGroupOfHorocyclicProduct}, which we obtain by direct calculation. Let $G$ be discrete and act geometrically on $X\bowtie Y$. A finite-index subgroup $K$ of $G$ is a subgroup of $\Isom(X')_\infty\times_{-h}\Isom(Y')_\infty$. The actions $\Isom(X')_\infty\curvearrowright X'$ and $\Isom(Y')_\infty\curvearrowright Y'$ have properties sufficient to conclude that both groups are amenable and hyperbolic. Using the structure theory of amenable hyperbolic groups, we obtain millefeuille spaces $Z$ and $W$ quasi-isometric to $X$ and $Y$ so that $K$ acts on both $Z$ and $W$. By rescaling the metrics as necessary, we find an action on $Z\bowtie W$, and show that it is geometric, proving \cref{thm:Upgrade_to_Millefeuille}.

The assumptions on the connectedness of the boundaries now tell us how many of $Z$ and $W$ are Heintze groups. In \cref{thm:boundary_implies_algebra} Part three, both $Z$ and $W$ are Heintze groups, so $K$ acts geometrically on a Sol-type group. It is then not difficult to see that the action on heights is discrete, and that the kernel of this action acts geometrically on a nilpotent Lie group. In \cref{thm:boundary_implies_algebra} Part two, exactly one of $Z,W$ is a Heintze group, and the other factor gives a cocompact action on a tree fixing an end. Bass-Serre theory then gives the desired HNN extension, with vertex groups acting geometrically on a nilpotent Lie group. For \cref{thm:boundary_implies_algebra} Part one, we prove the contrapositive: assume $G$ is finitely presented and use the Bieri-Strebel splitting lemma to show that the height function $h:K\to \Z$ splits $K$ as an ascending HNN extension over a finitely-generated subgroup of $\ker(h)$. Using Briton's Lemma, we conclude that any two points on the same horocycle of $Z\bowtie W$ can be connected by a coarse path never going above the horocycle. This is clearly not true if both $Z$ and $W$ are either trees or non-degenerate millefeuille spaces.

\subsection{Classification scheme for actions on horocyclic products.}

By \cref{thm:Upgrade_to_Millefeuille}, geometric actions on horocyclic products can be represented by actions on horocyclic products of Heintze groups, trees, or non-degenerate millefeuille spaces. Among the six possible combinations, we have discussed three: the groups $\Z^2\rtimes_A \Z$ ($|\tr(A)| > 2$), $\BS(1,n)$ and $\Z/n\Z\wr \Z$ are examples of groups modeled on horocyclic products of trees and Heintze groups. In the next subsection, we will exhibit a class of examples acting on a horocyclic product of a Heintze group and a non-degenerate millefeuille space---the simplest example is $\HNN(\Z^2,\tiny{\begin{pmatrix}3 & 1 \\ 1 & 1 \end{pmatrix}})$. The authors are not aware of any examples in the case where one of the factors is non-degenerate millefeuille and the other factor is not a Heintze group. The situation is summarized by the below table.

\vspace{.3cm}

\begin{center}
\begin{tabular}{c|c c c}
     & Heintze & tree & millefeuille \\
     \hline 
    Heintze &  $\Z^2 \rtimes_{\tiny{\begin{pmatrix}2 & 1 \\ 1 & 1\end{pmatrix}}} \Z$ & $\;\;\BS(1,n)\;\;$ & $\HNN(\Z^2,\tiny{\begin{pmatrix}3 & 1 \\ 1 & 1 \end{pmatrix}})$\\[2.5mm]
    tree & $-$ & $\Z/2 \wr\Z$ & ? \\[2.5mm]
    millefeuille & $-$ & $-$ & ?
\end{tabular}
\end{center}

\vspace{.3cm}

\noindent This suggests an obvious question:

\begin{question}
   Are there any discrete groups acting geometrically on a horocyclic product $Z\bowtie W$ where $Z$ is a nondegenerate millefeuille space and $W$ is a tree or non-degenerate millefeuille space? 
\end{question}

By \cref{thm:boundary_implies_algebra}, such a group is not finitely presented, and as we discuss in Section \ref{subsec:use_CCMT}, it must also be amenable.

Farb-Mosher \cite{FM1}, \cite{FM2}, \cite{FM3}, Eskin-Fisher-Whyte \cite{EFW1}, \cite{EFW2}, \cite{EFW3} and Peng \cite{Peng} have proven various quasi-isometric rigidity results for the groups appearing in the above table. Further work of Ferragut \cite{Ferragut:Visual_Boundary}, \cite{Ferragut:Rigidity} (which we draw on heavily in this paper) extended their techniques to a wider class of horocyclic products. Establishing the scope of these techniques (which rely on the horocyclic product structure) was one of our main motivations for this paper---any new examples of groups modeled on horocyclic products would be immediate candidates for a quasi-isometric rigidity theorem.

\subsection{Actions of HNN extensions of nilpotent groups on horocyclic products}
\label{subsec:Construct_millefeuille_example}

We will now describe a general class of ascending HNN extensions of nilpotent groups that act on horocyclic products, simultaneously generalizing the cases of $\BS(1,n)\curvearrowright T_n\bowtie \h^2$ and $\Z^2\ltimes_A\Z\curvearrowright \h^2\bowtie \h^2$.


Let $N_1,N_2$ be simply connected nilpotent Lie groups, and let $\Gamma \subset N := N_1\times N_2$ be a lattice. Let $\alpha_1: \R \to \Aut(N_1)$ and $\alpha_2: \R \to \Aut(N_2)$ be one parameter subgroups so that $\alpha: \R \to \Aut(N)$ given by $\alpha(t)(n_1,n_2) = (\alpha_1(t)(n_1), \alpha_2(t)(n_2))$ satisfies $\alpha(1)(\Gamma) \subset \Gamma$. Then $\alpha(1)$ descends to a map $\beta: N/ \Gamma \to N/\Gamma$ whose mapping torus $M_\beta$ has fundamental group $$\pi_1(M_\beta) = \HNN(\Gamma, \alpha(1)|_\Gamma) = \langle \Gamma, t \mid t\gamma t\inv = \alpha(1)(\gamma) \rangle.$$ This group acts by isometries on $(N_1\times N_2)\rtimes_\alpha \R$, and on the Bass-Serre tree $T_k$, for $k = [\Gamma: \alpha(\Gamma)]$. As observed by Farb-Mosher \cite{FM3}, it in fact acts \textit{geometrically} on the fibered product $$\{((n,t),v) \in (N\rtimes_\alpha \R) \times T_k \mid  t = -h_{T_k}(v)\}.$$ 

This space is not obviously a horocyclic product, since $N\rtimes_\alpha \R$ is not necessarily hyperbolic ($\alpha$ is not assumed to be expanding). We are particularly interested in the case where $\alpha_1$ is expanding and $\alpha_2$ is contracting (that is, $\alpha_2(-t)$ is expanding). In this case, $N\rtimes_\alpha \R$ is a Sol-type group, and is not hyperbolic. Nevertheless, as noted by Dymarz in \cite{Dymarz_Envelopes}, Section 9.2, the above space \textit{is} a horocyclic product with a different decomposition: there is an isometry
\[\{((n,t),v) \in (N\rtimes_\alpha \R) \times T_k \mid  t = -h_{T_k}(v)\} \to (N_1 \rtimes_{\alpha_1} \R) \bowtie (N_2\rtimes_{\alpha_2}\R)[k]\]
given by \[(((n_1,n_2),t), v) \to ((n_1,t),((n_2,t), v)).\]

For $a > 1$, let $\h^2_a$ denote the Heintze group $\R \rtimes_{a^t}\R$ equipped with the left-invariant Riemannian metric $dt^2 + e^{-2a t} dx^2$ and the height function given by projection to the second coordinate. Such spaces are isometric to the hyperbolic plane $\h^2$ rescaled to have curvature $-(\ln{a})^2$. Let $A\in \SL_2(\Z)$ have eigenvalues $\lambda_1 > 1 > \lambda_2$ and set $N_i \rtimes_{\alpha_i} \R = \R \rtimes_{\lambda_i^t} \R$.  When $\Gamma \subset N_1\times N_2  = \R^2$ is a well-chosen copy of $\Z^2$, the above construction gives a geometric action $\Z^2 \rtimes_A \Z \curvearrowright \Ha^2_{\lambda_1}\bowtie \Ha^2_{\lambda_2\inv}$. One can also obtain an action on $\Ha^2 \bowtie \h^2$ by choosing the curvature $-1$ metric on the Heintze groups, although $\Ha^2_{\lambda_1}\bowtie \Ha^2_{\lambda_2\inv}$ is is the more natural choice.

When $N_1\rtimes \R$ is $\R \rtimes_{n^t} \R \cong \Ha^2_n$, $N_2$ is trivial and $\Gamma = \Z \subset \R$, we obtain the example\footnote{The $\Ha^2 \bowtie T_n$ example can be obtained by choosing the curvature $-1$ metric on $\R \rtimes_{n^t} \R$ and declaring that the edges of $T_n$ have length $\ln{n}$ } $\BS(1,n)\curvearrowright \Ha^2_n \bowtie T_n$. The HNN extension $\HNN(\Z^2,\tiny{\begin{pmatrix}3 & 1 \\ 1 & 1 \end{pmatrix}})$ is an example involving a non-degenerate millefeuille space: the $N_i\rtimes_{\alpha_i}\R$ are $\R \rtimes_{\lambda_i^t} \R$, as in the $\Z^2 \rtimes_A\Z$ example, but the HNN extension is not a semidirect product, so the Bass-Serre tree is not a line. Then $\HNN(\Z^2,\tiny{\begin{pmatrix}3 & 1 \\ 1 & 1 \end{pmatrix}})$ acts geometrically on $\h^2_{\lambda_1} \bowtie \h^2_{\lambda_2\inv}[2]$.

We note that all horocyclic products $X\bowtie Y$ arising from this construction have factor spaces $X$ and $Y$ with identical exponential growth parameters (See \cite{Ferragut:Rigidity}, Definition 3.1, (E3)) and thus are not subject to Ferragut's geometric rigidity theorem for horocyclic products \cite{Ferragut:Rigidity}, Theorem A.

\subsection{Overview of the paper}

Section \ref{sec:Prelim} defines horocyclic products, proves basic properties about them, and discusses Ferragut's computation of their visual boundary. Section \ref{sec:CAT-1} establishes the main properties of $\CAT(-\kappa)$ spaces that we will use throughout the work. Section \ref{sec:NewMetricOnX} details the construction of the new metrics $d_X'$ and $d_Y'$ whose isometry groups are used in \cref{thm:IsomGroupOfHorocyclicProduct}. Section \ref{sec:Isom_group} proves \cref{thm:IsomGroupOfHorocyclicProduct}, describing the isometry group of a general horocyclic product. This theorem is then used in Section \ref{sec:Upgrade_to_Millefeuille} to upgrade the action on an arbitrary horocyclic product to a horocyclic product of millefeuille spaces, thereby proving \cref{thm:Upgrade_to_Millefeuille}. This upgraded action is then used in Section \ref{sec:boundary_implies_algebra} to prove \cref{thm:boundary_implies_algebra}. 

\noindent \textbf{Suggested itinerary.} The paper is organized to prove the most general version of the main theorem in logical order. On a first reading, the reader may wish to specialize to the case where the $L^\infty$ norm is used to construct the metric on $X\bowtie Y$. In that case, the metric spaces $X$ and $X'$ are isometric, and the entirety of Section \ref{sec:NewMetricOnX} (which contains several technical metric geometry arguments about $X'$) may be skipped, while replacing $X'$ with $X$ in all subsequent sections.  

\subsection{Acknowledgments}

The authors would like to thank Benson Farb and Tullia Dymarz for their advice and encouragement throughout this project. They would also like to thank Kevin Whyte, Daniel Groves, Tom Ferragut, Amie Wilkinson, Sven Sandfeldt, Wolfgang Woess, Gabriel Pallier, and Anthony Genevois for helpful conversations and suggestions. The authors would also like to thank Benson Farb for his extensive comments on early drafts.

\section{Horocyclic products and their visual boundaries}
\label{sec:Prelim}

In this section we define and prove some basic properties of horocyclic products and their visual boundaries. 

    \subsection{Negative curvature}

    We will be interested in horocyclic products of proper, geodesically-complete $\CAT(-\kappa)$ spaces. Here is the definition of the first two terms.
    
    \begin{Def} \label{def:BasicProperties}
        Let $X$ be a metric space. $X$ is \textit{proper} if the closures of metric balls are compact. $X$ is \textit{geodesic} if between any two points $x$ and $x'$ in $X$, there is a continuous path $\gamma:[0, 1]\to X$, so that $\gamma(0)=x$, $\gamma(1)=x'$, and where the length $\ell(\gamma)=d(x, x')$, where 
        \[\ell(\gamma)=\sup_{0=t_0<t_1<...t_n=1} \sum_{i = 1}^n d(\gamma(t_i), \gamma(t_{i+1})).\]
        Such a path is a \textit{geodesic}. A geodesic metric space is \textit{geodesically-complete} if every geodesic segment can be extended to a bi-infinite path isometric to $\R$.
    \end{Def}

    A space is $\CAT(-\kappa)$ if, informally, it displays at least as much negative curvature as the hyperbolic plane at every scale.

    \begin{Def}\label{def:CAT-1}
        Let $X$ be a proper geodesic metric space. Let $x_1$, $x_2$, $x_3$ be points in $X$. A \textit{curvature-}$\kappa$ \textit{comparison triangle} is a triple of points $X_i$ in the surface $\Sigma$ of constant curvature $\kappa$, so that $d_X(x_i, x_j)=d_{Z}(X_i, X_j)$. For any $y$ on a geodesic between $x_i$ and $x_j$, there is a \textit{comparison point} $Y$ on any choice of comparison triangle so that $d_X(x_i, y)=d_{\Sigma}(X_i, Y)$ and $d_X(x_j, y)=d_{\Sigma}(X_j, Y)$.

        $X$ is said to be $\CAT(\kappa)$ if, for any $y_1$ and $y_2$ on a geodesic triangle in $X$, and $Y_1$ and $Y_2$ their comparison points, $d_{X}(y_1, y_2)\le d_{\Sigma}(Y_1, Y_2)$.
    \end{Def}

    Henceforth, all metric spaces will be assumed to be at least proper, geodesically-complete, and $\CAT(-\kappa)$, unless a subset of these properties is listed explicitly as hypotheses for a result. While researchers are interested in $\CAT(\kappa)$ spaces for $\kappa>0$, for our purposes, when we say $\CAT(-\kappa)$, we implicitly mean $\kappa$ to be positive.

    We will only rarely use the full strength of the $\CAT(-\kappa)$ assumption. Usually, we will only need three properties of $\CAT(-\kappa)$ spaces. The first is the well-known fact that $\CAT(-\kappa)$ spaces are Gromov-hyperbolic, as described below. 

    \begin{Def}
        A proper geodesic metric space $X$ is said to be $\delta$\textit{-hyperbolic} if, for all geodesic triangles in $X$, each side is within the union of the $\delta$-neighborhoods of the other two.

        A space is \textit{Gromov-hyperbolic} if it is $\delta$-hyperbolic for some $\delta$.
    \end{Def}

    The second property that we use is that $\CAT(-\kappa)$ spaces are Busemann.

    \begin{Def}
        Let $X$ be a geodesic metric space. $X$ is \textit{Busemann} if, for any two geodesic segments $\gamma_1$ and $\gamma_2$ parameterized by length, the function $d(\cdot \, , \, \cdot):[0, \ell(\gamma_1)]\times[0,\ell(\gamma_2)]\to \R_{>0}$ is convex.
    \end{Def}
    
    Notice that we can immediately apply this property to rays and lines as well in any space where they exist.
    
    \begin{lemma} \label{lemma:CAT-1ImpliesBusemann}
        Let $X$ be $\CAT(-\kappa)$. Then $X$ is Busemann.
    \end{lemma}

    Lemma \ref{lemma:CAT-1ImpliesBusemann} is a corollary of a description of $\CAT(0)$ in \cite{FoertschLytchakSchroeder}, together with the fact that $\CAT(-\kappa)$ implies $\CAT(0)$. We present a simpler proof that works only in $\CAT(-\kappa)$ in Section \ref{sec:CAT-1}.

    The final property that we will use, that $\CAT(-\kappa)$ spaces are vertically convergent, requires some setup to explain. We first recall the boundary of a Gromov-hyperbolic space.

    \begin{Def}
        Let $X$ be Gromov-hyperbolic. For any two geodesic rays $\eta_1$, $\eta_2$ in $X$, we say $\eta_1\sim\eta_2$ if and only if the two are at finite Hausdorff distance (we will say that such rays are \textit{asymptotic to one another}). The collection
        \[\{\text{Geodesic rays in } X \}/\sim\]
        is defined to be the \textit{visual boundary of X}, denoted $\partial X$.

        Since any isometry $f$ of $X$ sends geodesic rays to geodesic rays, and preserves Hausdorff distance, any isometry has an induced map $\partial f$ on the boundary.
    \end{Def}

    See \cite{BridsonHaefliger}, chapter III.H.3 for details on this construction. From the general theory, we recall that between any two boundary points $\eta$ and $\xi$ in $\partial X$, there is a geodesic $\gamma:(-\infty, \infty)\to X$ so that $\gamma:[0, -\infty)=[\eta]$ and $\gamma:[0, \infty)\to X=[\xi]$ (\cite{BridsonHaefliger} Lemma III.H.3.2). 
    
    Gromov-hyperbolic spaces that are also Busemann spaces have a notion of height, measured in the direction of a preferred ray.

    \begin{Def}
        A \textit{Busemann function} on $X$ is a function $$\beta(x)=\lim_{t\to \infty} d(\gamma(t), x)-d(\gamma(t), x_0),$$ where $\gamma(t)$ is a fixed geodesic ray, and $x_0 \in \gamma(\R)$ is a choice of base point. Note that this limit always exists in a Busemann space. A level set of a Busemann function is a \textit{horosphere}, and a sublevel set is a \textit{horoball}. The negative of a Busemann function is a \textit{height function}, which we will denote $h$.
    \end{Def}

    One may compute directly that Busemann functions and height functions depend on their choice of base point $x_0$ only up to a constant: the differences $\beta(x)-\beta(y)$ and $h(x)-h(y)$ are well-defined and depend only on the choice of ray $\gamma$. Moreover, both functions are 1-Lipschitz by a direct application of the triangle inequality. Heuristically, the Busemann function measures how much closer or further a point $x$ is to $\gamma(\infty)$ than a fixed reference point $x_0$, with negative numbers indicating that $x$ is close to $\gamma(\infty)$ than $x_0$. 

    If $X$ has a height function on it, we will be particularly interested in rays asymptotic to the one that defines the height function.
    
    \begin{Def}
        Let $X$ be Gromov-hyperbolic and have a fixed height function on it defined by a ray $\eta$. A ray $\xi$ is said to be \textit{vertical} if $\xi\in[\eta]$.
    \end{Def}

    Ferragut proves in \cite{Ferragut:Visual_Boundary} Proposition 2.7 that, under the additional assumption that $X$ is Busemann, a ray is vertical if and only if it is parameterized by height, i.e. $\xi:[0, \infty)\to X$ is vertical if and only if $\lim h(\xi(t))=\infty$ and $d(\xi(t_1), \xi(t_2))=|h(\xi(t_1))- h(\xi(t_2))|$. We will often assume vertical rays are parameterized by height tacitly.

    In a general $\delta$-hyperbolic space, pairs of rays that are asymptotic to one another must eventually come within $5\delta$ of one another (\cite{BridsonHaefliger} Lemma III.H.3.3). However, in the sequel, we will require the infimum distance between vertical rays to go to $0$.
    
    \begin{Def}
        \label{def:vertically_Convergent}
        Let $\infty_X\in\partial X$. X is \textit{vertically convergent} (with respect to $\infty_X$) if for every pair of rays $\eta_1$, $\eta_2$ in $X$ asymptotic to $\infty_X$, parameterized by length, there is an $s$ so that
        \[ \lim_{t\to\infty} d(\eta_1(t), \eta_2(s+t))= 0.\] 
        Usually, $X$ will be equipped with at least a height function, in which case we will omit reference to the specific boundary point in question. It will always be the ray $\gamma$ with respect to which the height is defined.
    \end{Def}

    As mentioned, this is a property of all $\CAT(-\kappa)$ spaces for any choice of height function.

    \begin{lemma} \label{lemma:CAT-1ImpliesVerticalConvergence}
        Let $X$ be $\CAT(-\kappa)$ with height function $h$ defined by the ray $\infty_X$. Then $X$ is vertically convergent with respect to $\infty_X$.
    \end{lemma}

    We prove Lemma \ref{lemma:CAT-1ImpliesVerticalConvergence} in Section \ref{sec:CAT-1}.

    We make two more observations about $\CAT(-\kappa)$ spaces. The first is that vertical geodesic lines are uniquely determined by their endpoint other than $\infty_X$.
    
    \begin{lemma} \label{lemma:VerticalUniquenessInFactorSpaces}
       Let $X$ as before be vertically convergent with respect to $\infty_X$, and let $\xi\in\partial X\setminus \{\infty_X\}$. Then there is a unique vertical geodesic between $\xi$ and $\infty_X$.
    \end{lemma}

    Note that Busemann spaces are uniquely geodesic, and in hyperbolic Busemann spaces, there is always a single ray from any point to any boundary point. However, vertical convergence is required to obtain uniqueness of geodesic lines.
    
    \begin{proof}
        A vertical geodesic exists as $X$ is geodesically-complete. If $\gamma_1$ and $\gamma_2$ are any two vertical geodesics, then there is an $s$ so that $d(\gamma_1(t), \gamma_2(s+t))$ is a bounded convex function which tends to $0$, hence constantly $0$.
    \end{proof}
    
    Our second observation is that in a vertically-convergent space, the height function may be calculated from the Busemann function with respect to any vertical ray. In a general hyperbolic metric space, the Busemann functions with respect to different vertical rays may differ by a bounded amount.
    
    \begin{lemma}\label{lemma:HeightFromAnyVerticalRay}
        Let $X$ have height function $h$. Then for any vertical ray $\eta$ parameterized by arc length, and for any point $x\in X$,
        \[h(x) = h(\eta(0)) + \lim_{t\to\infty} \Big(t-d(x, \eta(t))\Big).\] 
    \end{lemma}

    Notice that the final term in the equality above is the negative of a Busemann function with respect to the ray $\eta$ and base point $\eta(0)$, i.e., a height function for the ray $\eta$. We emphasize also that this lemma does not rely on the Busemann property of $X$, which will be relevant in the sequel.

    
    \begin{proof}
        Let $h=-\beta_{\alpha}$, for some vertical ray $\alpha$ in $X$, where we parameterize $\alpha$ by height. By definition, there is an $x_0$ in $X$ so that
        \[ \beta_\alpha(x)=\lim_{t\to\infty} d(x, \alpha(t))-d(\alpha(t), x_0).\]
        so that
        \[ h(x)=\lim_{t\to\infty} d(\alpha(t), x_0)-d(x,\alpha(t)).\]

        Now, by vertical convergence, let $s$ be so that $\lim_{t\to\infty} d(\alpha(t), \eta(s+t))=0$. It follows from the triangle inequality that

        \[ h(x)=\lim_{t\to\infty} d(\eta(s+t), x_0)-d(x, \eta(s+t)).\]

        Applying this to $x=\eta(0)$ yields 

        \[ h(\eta(0))=\lim_{t\to\infty} d(\eta(s+t), x_0)-d(\eta(0), \eta(s+t))=\lim_{t\to\infty}d(\eta(s+t), x_0)-(s+t).\]

        We compute that 

        \[ h(x)-h(\eta(0))=\lim_{t\to\infty}s+t-d(x, \eta(s+t))\]

        which is equivalent to the desired equality.
    \end{proof}

    As a consequence, we see that isometries of vertically-convergent Gromov-hyperbolic Busemann spaces fixing the special point at infinity must have a well-defined height-change.
    
    \begin{prop}\label{prop:HeightChangeHom}
        Let $X$ have height function $h$ defined by a ray in $[\infty_X]$. Let $f:X\to X$ be an isometry fixing $\infty_X$. Then $h(f(x))-h(x)$ does not depend on $x$.
    \end{prop}

    \begin{proof}
        Suppose $h$ is defined by a ray $\eta$ tending to $\infty_X$. Because $X$ is vertically-convergent, $f(\eta)$ is a ray also tending to $\infty_X$, and there is some $s$ so that $d(\eta(t), f(\eta(t+s)))\to0$. Notice that $s=-h(f(x_0))$, for if not, the height difference between $\eta(t)$ and $f(\eta(s+t))$ would be a constant nonzero lower bound for their distance. We compute

        \begin{align*}
            h(x)&= \lim_{t\to\infty} d(\eta(t), x_0)-d(\eta(t), x)\\
             &= \lim_{t\to\infty} d(f(\eta(t)), f(x_0))-d(f(\eta(t)), f(x))\\
            &= \lim_{t\to\infty} d(\eta(t-s), f(x_0))-d(\eta(t-s), f(x))\\
            &=\lim_{t\to\infty} d(\eta(t-s), f(x_0))-d(\eta(t-s), x_0)+d(\eta(t-s), x_0)-d(\eta(t-s), f(x))\\
            &=-h(f(x_0))+h(f(x)),
        \end{align*}

        so that $h(x)-h(f(x))=s$, which depends only on $\eta$ and $f$.
    \end{proof}

    We term this difference $h(f(x))-h(x)=h(f(x_0))$ the \textit{height change of f}. The following corollary is immediate.

    \begin{cor}
    \label{cor:Factors_height_change_wellDef}
        Let $X$ be Gromov-hyperbolic, Busemann, and vertically convergent to $\infty_X$ with height function $h$.
        The height change map $f \mapsto h(f(x)) - h(x)$ is a well-defined homomorphism from $\Isom_\infty(X)$ to $\R$. \qed
    \end{cor}

    By slight abuse of notation, we will denote this homomorphism by $h_X$, or just $h$ when it does not cause confusion.

\subsection{Metric properties of horocyclic products}    
 
 The main characters of this work are horocyclic products, defined as follows.
    
    \begin{Def}
    \label{def:horocyclic_prod}
        The \textit{horocyclic product} of two spaces $X$ and $Y$ with height functions $h_X$ and $h_Y$, denoted $X\bowtie Y$, is defined to be $$X\bowtie Y = \{(x, y): h_X(x)=-h_Y(y)\}.$$
    \end{Def}

    We assign a horocyclic product a metric as follows.

    \begin{Def}\label{def:MetricOnHorocyclicProduct}
        Let $N$ be a norm on $\R^2$ normalized so that $N(1, 1)=1$. $N$ is \textit{admissible} if $N\ge \frac{L^1}{2}$. A norm will further be said to be \textit{monotone} if it is non-decreasing in both factors.

        Given such a norm, we get a distance $d_{\bowtie}$ on $X\bowtie Y$ defined as follows. Denote $d_X$ and $d_Y$ on $X\bowtie Y$ to be distance between the $X$ and $Y$ coordinates of a pair of points respectively. Consider paths $\gamma:[0,1]\to X\bowtie Y$ linking a pair of points $(x, y)$ and $(x', y')$. The length of such a path is defined to be 
        \[\ell_N(\gamma)=\sup_{0=t_0<t_1<...t_n=1} \sum N\bigl(d_X(\gamma(t_i), \gamma(t_{i+1})), d_Y(\gamma(t_i), \gamma(t_{i+1}))\bigr).\]
        Then $d_{\bowtie}((x, y), (x', y'))$ is the infimal such path length. We call such a metric \textit{admissible}.
    \end{Def}

    As we will repeatedly have to deal with chains of points $\gamma(0)=\gamma(t_0), \gamma(t_1), ... \gamma(t_n)=\gamma(1)$, we will fix the notation that $\mathscr{C}$ is the collection of all finite chains $0=t_0<t_1<...<t_n=1$.

    It is a straightforward exercise to show that an admissible monotone norm $N$ is bounded by the supremum norm, and thus by $L^1$. We will sometimes use this fact without explicitly mentioning it. We will henceforth assume all norms used to define metrics on horocyclic products are admissible and monotone, and again, sometimes not make this assumption explicitly.

    \begin{figure}
    \centering\includegraphics[scale=0.6]{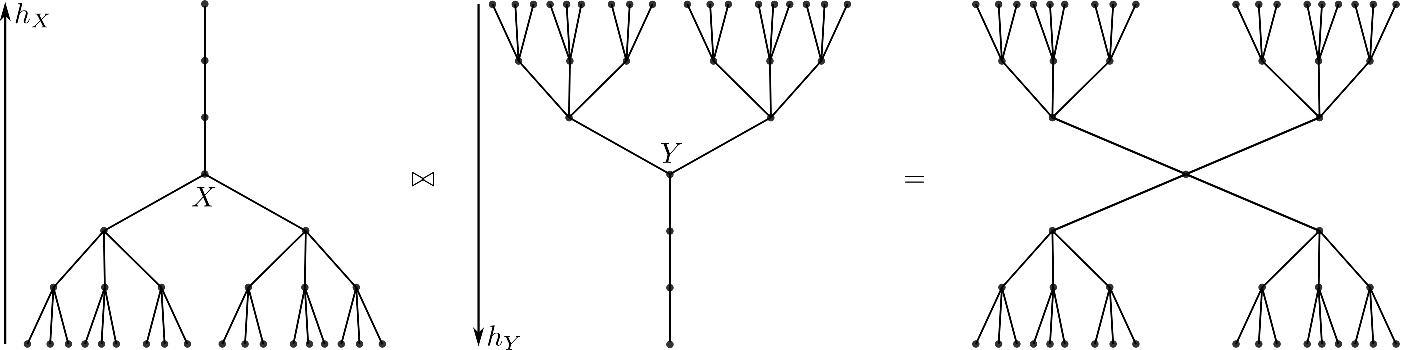}
    \caption{The horocyclic product of two trees might itself be a tree if we choose the height function badly.}
    \label{figure:Counterexample}
    \end{figure}

    If the spaces $X$ and $Y$ are trees, or more generally sufficiently-nice graphs, then $X\bowtie Y$ is isometric to a graph regardless of choice of norm $N$. This fact explains \cref{thm:boundary_implies_algebra}'s assumption that the height functions on $X$ and $Y$ correspond to points at infinity which are not isolated in the topology of $\partial X$ and $\partial Y$. For example consider Figure \ref{figure:Counterexample}. The spaces $X$ and $Y$ arise as follows. Consider a rooted tree $T$ whose root has two children, and all other vertices have three children. Attach a half-line $[0,\infty)$ to the root. Choose a height function with distinguished point given by the ray $[0,\infty)$ and so that the root has height $0$. Then $X\bowtie Y$ is the regular tree of valence four, and hence admits a geometric action of $F_2$. $X\bowtie Y$ therefore does not satisfy the conclusion of \cref{thm:boundary_implies_algebra}. In addition, $\Isom(X\bowtie Y)$ cannot be amenable because it contains a discrete free group, and hence cannot take the form described in \cref{thm:IsomGroupOfHorocyclicProduct}.
    
    In fact, there are many other possible counterexamples to \cref{thm:boundary_implies_algebra} when a distinguished point at infinity is isolated. As we will see in Corollary \ref{cor:RILowerBound}, if $Y$ is a line, then $X\bowtie Y$ is a hyperbolic space quasi-isometric to $X$, and therefore (outside trivial cases) cannot receive a geometric action by any of the groups described in \cref{thm:boundary_implies_algebra}, because no such group is nonelementary hyperbolic. The exact step where our arguments break down in the case where one space is a line is contained in Section \ref{subsec:IsomsPreserveUpperAndLower}.

    We next provide a basic characterization of admissible metrics on horocyclic products, which says roughly that two points are near one another if they are near one another in both coordinates. The converse is a basic exercise in the definition of admissible distances. 
    \begin{lemma}\label{lemma:BoundedCoordinateDinstancesImpliesBoundedDistance}
        Let $X\bowtie Y$ be any horocyclic product of geodesically-complete metric spaces, with admissible metric $d_{\bowtie}$. Suppose $(x_1, y_1)$ and $(x_2, y_2)$ are points in $X\bowtie Y$. Then $d_{\bowtie}((x_1, y_1),(x_2, y_2))$ is bounded above by $d_X(x_1, x_2)+d_Y(y_1, y_2)+\min\{d_X(x_1, x_2),d_Y(y_1, y_2)\}$.
    \end{lemma}

    Note that \cite{Ferragut:Visual_Boundary} Theorem A also provides a statement that points in a horocyclic product are near to one another if their coordinate distances are small. However, that theorem involves a bounded additive error, and in the sequel we will need to be able to give upper bounds on distances that we can send to $0$. Therefore, we need Lemma \ref{lemma:BoundedCoordinateDinstancesImpliesBoundedDistance}.
    
    \begin{proof}
        Let $\gamma_X$ and $\gamma_Y$ be the geodesics between $x_1$ and $x_2$, and between $y_1$ and $y_2$ respectively. Let $V_{X, x_2}$ and $V_{Y, y_1}$ be vertical lines in $X$ and $Y$ through $x_2$ and $y_1$ respectively, parameterized by height.
			
        Denote path concatenation by $*$, and consider the path
        \[\bigl(\gamma_X(t), V_{Y, y_1}(-h_X(\gamma_X(t)))\bigr) *\bigl(V_{X, x_2}|_{[h_X(x_2), h_X(x_1)]},V_{Y, y_1}|_{[-h_Y(y_2), -h_Y(y_1)]}\bigr) * \bigl(V_{X, x_2}(-h_Y(\gamma_Y(t))), \gamma_Y(t)\bigr),\]
	as illustrated Figure \ref{figure:DistanceUpperBound}.
			
	\begin{figure}
            \centering\includegraphics[scale=.45]{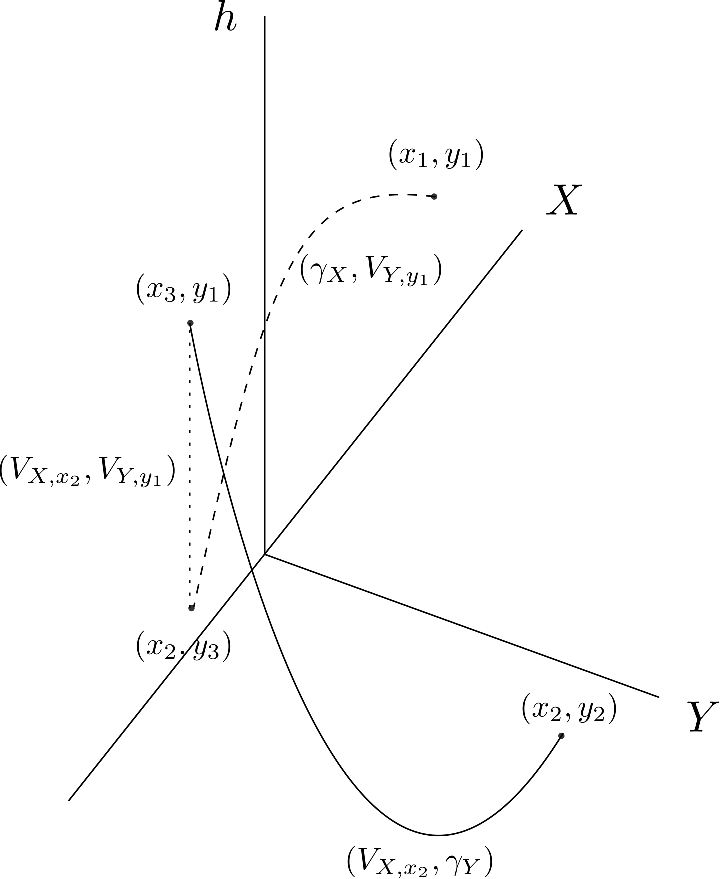}
            \caption{Illustration of the proof of Lemma \ref{lemma:BoundedCoordinateDinstancesImpliesBoundedDistance}. The first, second, and third subpaths are dashed, dotted, and solid respectively. $y_3$ is a point on $V_{Y, y_1}$ at height $-h_X(x_2)$, and $x_3$ is a point on $V_{X, x_2}$ at height $-h_Y(y_1)$}
            \label{figure:DistanceUpperBound}
        \end{figure}
			
        The first leg has $X$-length $d_X(x_1, x_2)$, and $Y$-length equal to the total variation of the height along $\gamma_X$. Since the height is $1$-Lipschitz, this variation is no more than $d_X(x_1, x_2)$. Hence the length of the first leg is no more than $d_X(x_1, x_2)$. By the same token, the length of the third leg is no more than $d_Y(y_1, y_2)$. The coordinates of the middle leg are both vertical segments, and therefore the length of this leg is its total height change. Since the height is $1$-Lipschitz, this height change is no more than $\min\{d_X(x_1, x_2), d_Y(y_1, y_2)\}$. Therefore, the distance between $(x_1, y_1)$ and $(x_2, y_2)$ is no more than $(d_X(x_1, x_2)+d_Y(y_1, y_2))+\min\{d_X(x_1, x_2), d_Y(y_1, y_2)\}$.
    \end{proof}

\subsection{Certain subspaces of $X\bowtie Y$}

    For any vertical line $V \subset Y$, consider the subspace $X\bowtie V \subset X\bowtie Y$. The projection $\pi_X:X\bowtie V$ onto the $X$-factor is a homeomorphism. However, $\pi_X$ is not typically an isometry between $(X\bowtie V, d_{\bowtie}|_{X\bowtie V})$ and $(X, d_X)$, contrary to an assertion without proof in \cite{Ferragut:Visual_Boundary}. The only cases where $\pi_X$ will be an isometry are when $X$ is a line (which we rule out by the assumption that the point $\infty_X$ is not isolated), or when the defining norm $N$ is the $L^\infty$ norm. We next describe the metric on these subspaces. All of our discussion will also go through for subspaces $V\bowtie Y$ for vertical $V$ in $X$, but is phrased in terms of $X$ for notational convenience.
        
    \begin{Def} \label{def:d'_initial}
        Let $V$ be any vertical geodesic in $Y$, and consider the subspace $X\bowtie V\subset X\bowtie Y$. Points in $X\bowtie V$ are uniquely determined by their $X$-coordinate, so that projection $\pi_X$ onto the first factor provides a bijection between $X$ and $X\bowtie V$. The metric $d'_{X}$ on $X$ will be metric so that $d'_{X}\circ \pi_X=d_{\bowtie}|_{X\bowtie V}$. 

        We will also write $X'$ to denote the metric space $(X, d'_X)$.
    \end{Def}

    We will show in Corollary \ref{cor:EquivalentDefinitions} that this equivalent to the following definition.

    \begin{Def}\label{def:d'_equivalent}
        Consider $X$ as a topological space, and define the length of a path $\gamma:[0,1]\to X$ to be \[ \sup_\mathscr{C}\sum_{i=0}^{n-1} N\bigl(d_X(\gamma(t_i), \gamma(t_{i+1})),|h_X(\gamma(t_i))-h_X(\gamma(t_{i+1}))|\bigr).\] 
        Then $d'_{X}(x_1, x_2)$ may be defined as the infimal path length.
    \end{Def}
    
    One sees from Definition \ref{def:d'_equivalent} that $d'_X$ does not depend on $Y$, or the choice of vertical geodesic $V$. This definition also shows that $\Isom(X, d_X)\subset \Isom(X, d'_X)$. We will use this definition in Section \ref{sec:NewMetricOnX} to deduce that $d'_X$ is proper and $\delta$-hyperbolic, and to show that isometries of $(X, d_X')$ change the height of every point by the same amount.
    
\subsection{Visual boundaries of horocyclic products}\label{subsec:FerragutVisualBoundaries}

The boundary theory of horocyclic products was introduced by Farb-Mosher \cite{FM1} in the special case of $\Ha^2 \bowtie T_n$. There, the upper and lower boundaries were defined via (quasi-isometric) embeddings of $\Ha^2$ and $\R $ to $\Ha^2 \bowtie T_n$. In this paper, we use the \textit{visual boundary} of $X\bowtie Y$, defined below. By a theorem of Ferragut, this boundary coincides with the boundary used by Farb-Mosher. Since $X\bowtie Y$ is not Gromov-hyperbolic, more care is required in the definition.

\begin{Def}
    Let $X$ be a geodesic metric space with basepoint $p$. 
    The visual boundary of $X$ with respect to $p$, denoted $\partial_{p} X$ is the space of equivalence classes of geodesic rays $\gamma:[0,\infty) \to X$ with $\gamma(0) = p$, where two geodesic rays $\gamma, \eta$ are considered equivalent if the Hausdorff distance $d_{\text{Haus}}(\eta, \gamma)$ is finite. 
\end{Def}

Notice that, unlike the Gromov-hyperbolic case, the boundary may depend on a choice of base point. This is because, in a proper geodesic $\delta$-hyperbolic metric space $X$, we are guaranteed a geodesic ray asymptotic to a ray $\eta$ starting at any point $p$ in $X$, so that the boundary does not depend on choice of base point. In analogy to the hyperbolic case, we will write $\partial_p X$ as $\partial X$ when the boundary does not depend on the basepoint, regardless of whether $X$ is hyperbolic or not.

\begin{Def}[Vertical geodesics in $X\bowtie Y$]
    Let $X$, $Y$ be Busemann spaces. A geodesic (or geodesic ray) $\gamma$ in $X\bowtie Y$ is said to be \textit{vertical} if the composition $h \circ \gamma$ is an isometric embedding. 
\end{Def}

Given vertical geodesics $V:\R \to X$ and $W :\R \to Y$, parameterized so that $h\circ V = h\circ W = \id_\R$, the map $V\bowtie W \to X\bowtie Y$ given by $t \to (V(t),W(-t))$ is a vertical geodesic. Its two endpoints at infinity represent points in $\partial(X\bowtie Y)$. Since $V$ and $W$ are vertical, they are specified by their endpoints not equal to $\infty_X$ or $\infty_Y$ (Lemma \ref{lemma:VerticalUniquenessInFactorSpaces}). Any other choice of vertical $V' \subset X$ would have $V\bowtie W(t)$ and $V'\bowtie W(t)$ converging in the $X$ factor as $t \to \infty$. Then the $+\infty$ endpoint of $V\bowtie W$ arising from this construction is specified by $W(\infty) \in \partial Y \setminus \{\infty_Y\}$. Likewise, the $-\infty$ endpoint of $V\bowtie W$ is specified by $V(-\infty) \in \partial X \setminus \{\infty_X\}$.

Bi-infinite vertical geodesics in $X\bowtie Y$ are thus parameterized by $(\partial X \setminus\{\infty_X\}) \times (\partial Y \setminus \{\infty_Y\})$. We call these two spaces the upper and lower boundaries of $X\bowtie Y$. For notational simplicity, these will be denoted as $$\partial^u Y = \partial Y \setminus \{\infty_Y\} \qquad \quad \partial_\ell X  = \partial X \setminus \{\infty_X\}.$$

The following theorem due to Ferragut says that the upper and lower boundaries in fact parameterize the entire visual boundary of a horocyclic product.

\begin{thm}[Ferragut, \cite{Ferragut:Visual_Boundary}, Corollary C]
    Let $\gamma:[0,\infty)\to X\bowtie Y$ be a geodesic ray. Then there is a vertical geodesic ray in $\eta$ with $d_{\text{Haus}}(\gamma,\eta) < \infty$. Consequently, 
    \[\partial(X\bowtie Y) = \partial^u Y \sqcup \partial_\ell X,\label{eqn:boundary_decomp}\]
    where we understand elements of $\partial^u Y$ and $\partial_\ell X$ to represent boundary points of vertical geodesics, as in the preceding paragraph. In particular, the visual boundary does not depend on the basepoint.
\end{thm}

\section{Three facts about $\CAT(-\kappa)$ spaces}
\label{sec:CAT-1}

    In this section, we prove three facts about $\CAT(-\kappa)$ spaces, two of which were asserted without proof in Section \ref{sec:Prelim}. The third fact will be used in the proof of \cref{thm:IsomGroupOfHorocyclicProduct}. This is the only section of the paper which makes direct use of comparison triangles. 
    
    We now fix some notation. Throughout this section, lowercase letters will refer to points in the $\CAT(-\kappa)$ space $X$, and points in corresponding comparison triangles will be denoted with the corresponding uppercase letter. The line segment connecting points $x$ and $y$ will be denoted $\overline{xy}$. The (bi-infinite) line connecting $X$ and $Y$ will be denoted $\overleftrightarrow{XY}$. 

    \theoremstyle{plain}
    \newtheorem*{BusemannLemma}{Theorem \ref{lemma:CAT-1ImpliesBusemann}}

    \begin{BusemannLemma}
            Let $X$ be $\CAT(-\kappa)$, and let $\gamma_1$ and $\gamma_2$ be any two geodesic segments parameterized by arc length. Then $d:[0, \ell(\gamma_1)]\times [0, \ell(\gamma_2)]\to \R$ is a convex function.
    \end{BusemannLemma}

    \begin{proof}
        Let $x_1$ and $x_2$ be the starting and ending points of $\gamma_1$, and $x_3$ and $x_4$ the starting an ending points of $\gamma_2$. Let $x_5$ and $x_6$ be midpoints of $\overline{x_1x_2}$ and $\overline{x_3x_4}$ respectively.

        Since distances are continuous, it suffices to show that for each such pair of segments, 
        \[d_X(x_5, x_6) \le \frac{d_X(x_1, x_3)+d_X(x_2, x_4)}{2}.\]

        \begin{figure}
            \includegraphics{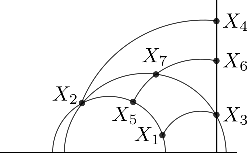}
            \caption{Illustration of the proof of Lemma \ref{lemma:CAT-1ImpliesBusemann}. Triangles $\triangle X_1X_2X_3$ and $\triangle X_2X_3X_4$ are comparison triangles for $\triangle x_1x_2x_3$ and $\triangle x_2x_3x_4$ respectively.}
            \label{figure:Busemann}
        \end{figure}

        As in Figure \ref{figure:Busemann}, consider the triangles $\triangle x_1x_2x_3$ and $\triangle x_2x_3x_4$. Find comparison triangles for them sharing the points $X_2$ and $X_3$. Reflecting about $\overleftrightarrow{X_2X_3}$ if necessary, we may assume that the interiors of $\overline{X_1X_2}$ and $\overline{X_3X_4}$ are on opposite sides of $\overleftrightarrow{X_2X_3}$. Take $X_5$ and $X_6$ the comparison points to $x_5$ and $x_6$.

        We claim that the unique geodesic, $\overline{X_5X_6}$, between $X_5$ and $X_6$ intersects the segment $\overline{X_2X_3}$. Certainly it intersects $\overleftrightarrow{X_2X_3}$ because $X_5$ and $X_6$ are on opposite sides of $\overleftrightarrow{X_2X_3}$. A segment between $X_6$ and a point on $\overline{X_1X_2}$ cannot meet a point on the opposite side of $\overleftrightarrow{X_1X_2}$ from $X_6$, which includes all points on $\overleftrightarrow{X_2X_3}$ past $X_2$. Similarly, a segment between $X_5$ and a point on $\overline{X_3X_4}$ cannot meet a point on the opposite side of $\overleftrightarrow{X_3X_4}$ from $X_5$, which rules out points on $\overleftrightarrow{X_2X_3}$ past $X_3$.

        Denote by $X_7$ the point $\overline{X_5X_6}\cap\overline{X_2X_3}$, and let it be the comparison point for $x_7$ in $X$ on $\overline{x_2x_3}$. We then compute

        \begin{align}
            d_X(x_5, x_6)&\le d_X(x_5, x_7)+d_X(x_7, x_6)\\
            &\le d_{\h^2}(X_5, X_7)+d_{\h^2}(X_7, X_6)\\
            &=d_{\h^2}(X_5, X_6)\\
            &\le \frac{d_{\h^2}(X_1, X_3)+d_{\h^2}(X_2, X_4)}{2}\\
            &= \frac{d_X(x_1, x_3)+d_X(x_2, x_4)}{2}
        \end{align}
        where we used the convexity of the distance function in negatively curved homogeneous surfaces on line 4.
    \end{proof}

    We next show that $\CAT(-\kappa)$ spaces are vertically convergent.

    \newtheorem*{VCLemma}{Theorem \ref{lemma:CAT-1ImpliesVerticalConvergence}}
    
    \begin{VCLemma}
        Let $X$ be $\CAT(-\kappa)$ with height function $h$ defined by the ray $\infty_X$. Then $X$ is vertically convergent with respect to $\infty_X$.
    \end{VCLemma}

    The proof will rely on an alternate characterization of the boundary of $X$, whose properties we will briefly explain. The boundary $\partial X$ compactifies the space $X$ in the sense that there is a compact topology on $X\sqcup \partial X$, in which $X$ is an open subset with $\partial X$ as its topological boundary (\cite{BridsonHaefliger} Proposition III.H.3.7). This topology is defined as follows: fix a base point $p$, and define a sequence $(x_i)\in X$ to converge to $x_\infty\in X\sqcup\partial X$ if there are choices of geodesic segments $\xi_i$ parameterized by length between $p$ and $x_i$ converging to a segment or ray between $p$ and $x_\infty$ uniformly on compact subsets (\cite{BridsonHaefliger} Proposition III.H.3.5). In particular, for any rays $\eta_1$ and $\eta_2$ in the same class, the segments between $\eta_2(t_0)$ and $\eta_1(t)$ converge uniformly on compact subsets to the ray $\eta_2$ as $t\to \infty$. With this fact in hand, we proceed to prove Lemma \ref{lemma:CAT-1ImpliesVerticalConvergence}.
    
   \begin{proof}
        Let $\eta_1$ and $\eta_2$ be two vertical rays in $X$ and let $s=h(\eta_1(0))-h(\eta_2(0))$. WLOG we assume $s\ge 0$. The points $\eta_1(t)$ and $\eta_2(s+t)$ are at the same height.

        \begin{figure}
            \includegraphics[scale=0.75]{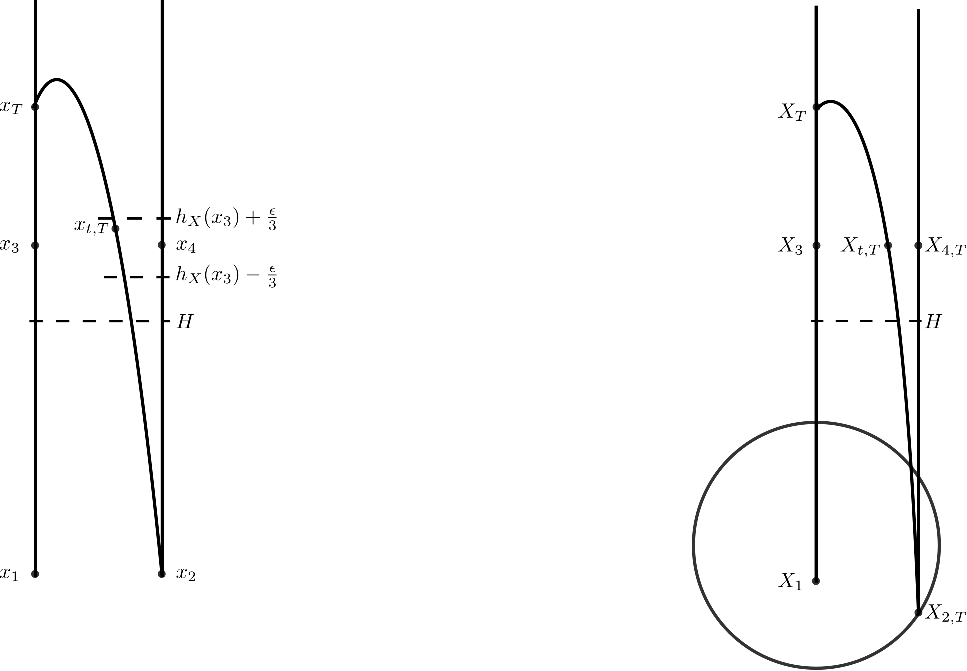}
            \caption{Illustration of the proof of Lemma \ref{lemma:CAT-1ImpliesVerticalConvergence}. On the left are points in $X$, and on the right are points in $\h^2$. The point $X_{2,T}$ varies along the righthand circle.}
        \end{figure}

        Let $\epsilon>0$. Let $x_1=\eta_1(0)$, $x_2=\eta_2(s)$, and $x_T=\eta_1(T)$, where $T>4$ so as not to create a notation conflict.
        
        Choose for the remainder of the proof a point $X_1$ at height $0$ in $\h^2$, and let $X_T$ be chosen so that $h_{\h^2}(X_T)=T=d_{\h^2}(X_1, X_T)$. Denote by $X_{2,T}$ a point in $\h^2$ so that $\triangle X_1X_TX_2$ is a comparison triangle for $\triangle x_1x_Tx_2$ in $X$. Let $\xi_1=\overrightarrow{X_1X_T}$ denote the vertical ray starting at $X_1$. Let $\xi_{2, T}$ denote the vertical ray starting at $X_{2,T}$. 
        
        Notice that $X_{2, T}$ varies with $T$, but is confined to a circle of radius $d_X(x_1, x_2)$ about $X_1$. Since $X_{2, T}$ varies along a compact set, there is a height $H$ so that points on $\xi_1$ and $\xi_{2, T}$ at the same height above $H$ are less than $\frac{\epsilon}{9}$ of one another. The convexity of the distance function between geodesics shows that, when $T>H$, the segment $\overline{X_TX_{2, T}}$ is always within a $\frac{\epsilon}{9}$-neighborhood of $\xi_{2, T}$. Therefore, each point on $\overline{X_TX_{2, T}}$ is less than $\frac{2\epsilon}{9}$ of a point on $\xi_{2, T}$ of matching height.
        
        Choose $t$ so that $h_X(\eta_1(t))>H$, and denote $\eta_1(t)=x_3$ and $X_3$ its comparison point. Let $x_4$ be the point on $\eta_2$ of matching height, namely $\eta_2(s+t)$. Consider $d_X(x_3, x_4)$. Choose $T>H$ sufficiently large so that the segment $\overline{x_Tx_2}$ is within less than $\frac{\epsilon}{6}$ of $\eta_2$ in the height range $[h_X(x_3)-\frac{\epsilon}{3}, h_X(x_3)+\frac{\epsilon}{3}]$. This can always be done because the segments $\overline{x_Tx_2}$ converge to $\eta_2$ in the compact-open topology.

        Since $h_X(x_3)>H$, for any choice of $T>H$, we see that $X_3$ is within $\frac{\epsilon}{9}$ of the point on $\gamma_{2, T}$ of matching height, which we denote $X_{4, T}$. When $T>t>H$, there is a point of matching height on $\overline{X_TX_2}$ at distance less than $\frac{2\epsilon}{9}$ of $X_{4, T}$. It follows that there is some point, $X_{t, T}$ on $\overline{X_TX_2}$ at distance less than $\frac{\epsilon}{3}$ from $X_3$.

        Since $\triangle X_1X_TX_2$ is a comparison triangle for $\triangle x_1x_Tx_2$, the point $X_{t,T}$ is a comparison point for $x_{t, T}$ which is at distance less than $\frac{\epsilon}{3}$ of $x_3$, therefore within the height range $(h_X(x_3)-\frac{\epsilon}{3}, h_X(x_3)+\frac{\epsilon}{3})$. But since $T$ was chosen large enough that $\overline{x_Tx_2}$ is within $\frac{\epsilon}{6}$ of $\eta_2$ in this height range, it follows that there is a point on $\eta_2$ at distance less than $\frac{\epsilon}{2}$ from $x_3$. Therefore, there is a point at height $h_X(x_3)$ on $\eta_2$ at distance less than $\epsilon$ from $x_3$. But the point at height $h_X(x_3)$ on $\eta_2$ is $x_4$. Thus, as long as $t>H$, we see that $d_X(\eta_1(t), \eta_2(t+s))<\epsilon$.
    \end{proof}
    
    The next lemma will be used in the proof of \cref{thm:IsomGroupOfHorocyclicProduct}. It says that if the height difference of two points $x,y \in X$ with $h(x) > h(y)$ is close to their distance, then $x$ is close to the vertical ray at $y$. Note that this is false for $\R^2$. 

    This lemma is the only stage at which we use the full strength of the $\CAT(-\kappa)$ assumption---everything else in the paper holds for spaces which are assumed to be Gromov-hyperbolic, Busemann and vertically convergent in place of $\CAT(-\kappa)$. 

    \begin{lemma}\label{lemma:AlmostVerticalSegments}
            
        Let $X$ have height function $h$. There is a continuous non-decreasing function $f:\R_{>0}\to \R_{>0}$ such that $\lim_{t\to 0}f(t)=0$ satisfying the following. Let $x,y\in X$ with $h(x)\ge h(y)$, and suppose $d(x,y)\leq h(y) -h(x) + C$. Let $V_{y}$ be the vertical geodesic ray starting at $y$. Then $d(x,V_y) \leq f(C)$.  
        
        Moreover, if $V_x$ is the vertical geodesic starting from $x$, then $V_x$ is in the $f(C)$-neighborhood of $V_{y}$.
            
    \end{lemma}

    \begin{proof}

    We will show that the function $f(C) = \cosh\inv(e^{C/2})$ suffices when $X$ is $\CAT(-1)$. Once we do, we can analyze every other case as follows. Since rescaling the hyperbolic plane by a factor $c$ multiplies the curvature by $\frac{1}{c^2}$, for a $\CAT(-\kappa)$ space, we rescale the metric by a factor of $\sqrt{\kappa}$ and apply the $\CAT(-1)$ upper bound $d_{\h^2}(x, V_y)\le f(\sqrt{\kappa}C)$. Then $d_X(x, V_y)\le\frac{1}{\kappa}f(\sqrt{\kappa}C)$, which is an upper bound with the desired behavior. So it suffices to prove the inequality for the $\CAT(-1)$ case.
    
    Let $\epsilon > 0$ with $\epsilon < (h(x) - h(y))/2$ and let $x' \in V_y$ be the point with height $h(x') = h(x)$. Choose a point $z \in V_y$ high enough so that $d(x,z) \leq d(x',z) +\epsilon$. The geodesic triangle $\triangle xyz$ satisfies $$\ell(\overline{xz}) + \ell(\overline{xy}) \leq \ell(\overline{yz}) + C + \epsilon.$$ Set $D = C+ \epsilon$. It suffices to show that $d(x,V_y) \leq \cosh\inv(e^{D/2})$ by letting $\epsilon \to 0$. Let $\triangle XYZ$ be the comparison triangle for $\triangle xyz$ and $X'$ the comparison point for $x'$. Using the $\CAT(-1)$ inequality, it suffices to show that $d(X, \overline{YZ}) \leq \cosh\inv(e^{D/2})$. Hence, we have reduced the problem to hyperbolic trigonometry.
    
    We will imagine the comparison triangle to have vertical $\overline{YZ}$ side in the upper half plane model, with $Y$ below $Z$. See Figure \ref{fig:hyperbolic_trig}. Since $$d(X,Z) \leq d(X',Z) + \epsilon \leq d(X',Z) + (h(x)-h(y))/2 = d(X',Z) + d(X',Y)/2,$$ we conclude that $X$ is above $Y$. Let $W$ be the closest point projection of $X$ onto $\overleftrightarrow{YZ}$. Since $X$ is above $Y$, It follows that $W$ lies above $Y$. We may choose $z$ high enough so that $W$ also lies below $Z$. Then in fact $W$ lies on $\overline{YZ}$.

    \begin{figure}
        \centering
        \includegraphics[width=2.5cm]{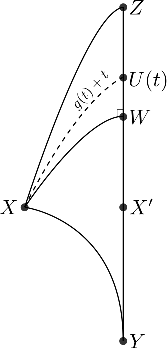}
        \caption{The comparison triangle in the proof of \ref{lemma:AlmostVerticalSegments}}
        \label{fig:hyperbolic_trig}
    \end{figure}

    Note that $\triangle XWZ$ and $\triangle XWY$ are both right triangles. Recall that for a right triangle in the hyperbolic plane, we have the identity \begin{equation} \cosh(\text{hyp}) = \cosh(\text{opp}) \cosh(\text{adj}), \label{Trig_Identity}\end{equation} where $\text{hyp},\; \text{opp},\; \text{adj}$ denote the lengths of the hypotenuse, opposite and adjacent sides. 

    
    
    Consider the function $$ g(t) = \cosh\inv\big(\cosh(\ell(\overline{XW})) \cdot \cosh(t)\big) - t.$$ In light of the identity (\ref{Trig_Identity}), this is the function used to compute the height of $X$ relative to $
    W$: We have $h(X) = h(W) - \lim_{t\to\infty} d(X, U(t)) - t = h(W) - \lim_{t\to\infty} g(t)$, where $U(t) \in \overline{YZ}$ is the point on $\overline{YZ}$ above $W$ with $d(W, U(t)) = t$. In particular, $g$ is decreasing.

    
    The identity (\ref{Trig_Identity}) also gives the relations $$g(\ell(\overline{ZW}))  = \ell(\overline{XZ}) - \ell(\overline{ZW}) \quad \text{and} \quad g(\ell(\overline{YW})) = \ell(\overline{XY}) - \ell(\overline{YW}).$$ 


    
    Since $\ell(\overline{YZ}) = \ell(\overline{YW}) + \ell(\overline{ZW})$, adding these relations gives $$g(\ell(\overline{ZW})) + g(\ell(\overline{YW})) = \ell(\overline{XZ}) + \ell(\overline{XY}) - \ell(\overline{YZ}) \leq D.$$
    
    We now explicitly compute the limit $\lim_{t\to \infty} g(t)$. Since $g$ is decreasing, this is also a lower bound for $g$.
    
    \begin{align*}
        \lim_{t\to\infty} g(t) &= \lim_{t\to\infty} \cosh\inv(\cosh(\ell(\overline{XW}))\cosh(t)) - t \\
        &= \lim_{t\to\infty} \ln(2\cosh(\ell(\overline{XW})) \cosh(t)) - t \\
        &= \lim_{t\to\infty} \ln(\cosh(\ell(\overline{XW}))) + \ln(2\cosh(t)) - t \\
        &= \ln (\cosh(\ell(\overline{XW}))) + \lim_{t\to\infty} \left[ \ln\left(2 \frac{e^{t} + e^{-t}}{2}\right) - t  \right] \\ 
        &= \ln(\cosh(\ell(\overline{XW})))
    \end{align*}

    
    \noindent Where the first step used the fact that $\lim_{t\to\infty} \cosh\inv(t) - \ln(2t) = 0$.
    
    We now have the bounds $$ \ln(\cosh(\ell(\overline{XW})))\leq g(\ell(\overline{ZW})) \quad \text{and} \quad \ln(\cosh(\ell(\overline{XW})))\leq g(\ell(\overline{YW})).$$ Adding these bounds gives $$2 \ln(\cosh(\ell(\overline{XW}))) \leq g(\ell(\overline{ZW})) + g(\ell(\overline{YW}))\leq D,$$ so $d(X,\overline{YZ}) = \ell(\overline{XW})\leq \cosh\inv(e^{D/2}),$ as desired.


    
    Finally, let $V_x$ be the vertical geodesic ray starting from $x$. Let $s > 0$ satisfy $d(x, V_y(s)) \leq f(C)$. Since $d(V_y(s+t), V_x(t))$ is a convex, decreasing function $[0,\infty) \to \R$, hence bounded by its initial value, which is itself bounded by $f(C)$.
    \end{proof}

\section{The metric space $X'$}\label{sec:NewMetricOnX}

    Recall that the metric space $X'$ is the space $X$, with the metric $d_{\bowtie}|_{X\bowtie V}$ for some  choice of vertical line $V$ in $Y$. Everything in this section will go through for $Y$ and $Y'=(Y, d_{\bowtie}|_{V\bowtie Y})$ for a vertical line $V$ in $X$ as well.

    In this section we will first show that $X'$ may equivalently be given the length metric associated to $N\bigl(d_X,\Delta h_X\bigr)$, where $N$ is the norm defining the metric $d_{\bowtie}$ and $\Delta h_X$ refers to the absolute difference in height between a pair of points. This will require showing that $X'$ is geodesic. We will then compare $d_X'$ to $d_X$ to show that $X'$ is proper and hyperbolic. Finally, we will show that isometries of $(X,d_X')$ have a well-defined height change.

    \subsection{Equivalence of the two definitions of $d_X'$.}

        We first show that $d_{\bowtie}|_{X\bowtie V}$ is geodesic using the monotonicity of $N$.
    
        \begin{lemma} \label{lemma:MonotoneImpliesConvex}
            Let $(x_1, y_1)$ and $(x_2, y_2)$ be in $X\bowtie V$. Then there is a geodesic between $(x_1, y_1)$ and $(x_2, y_2)$ that lies in $X\bowtie V$.
        \end{lemma}
    
        \begin{proof}
            Let $\gamma:[0,1]\to X\bowtie Y$ be any geodesic path between $(x_1, y_1)$ and $(x_2, y_2)$ in $X\bowtie Y$. Consider the path $\gamma_X=\pi_X(\gamma)$. Denote by $i_V$ the inverse map to $\pi_X|_{X\bowtie V}$, and consider the path $\gamma'=i_V\circ\pi_X\circ\gamma$. This is a path in $X\bowtie V$ whose starting and ending points have $X$-coordinates $x_1$ and $x_2$ respectively. Therefore, $\gamma'$ is a path path connecting $(x_1, y_1)$ to $(x_2, y_2)$.
    
            The length of $\gamma'$ is given by \[\sup_\mathscr{C}\sum_{i=0}^{n-1} N\bigl(d_X\circ\pi_X(\gamma'(t_i), \gamma'(t_{i+1})),|h_X(\gamma'(t_i))-h_X(\gamma'(t_{i+1}))|\bigr).\]
    
            Since the height function is $1$-Lipschitz, it follows that $|h_X(\gamma'(t_i))-h_X(\gamma'(t_{i+1}))|\le d_Y\circ\pi_Y(\gamma(t_i), \gamma(t_{i+1})).$ Monotonicity of the norm shows that the length of $\gamma'$ is therefore at most 
            \[\sup_\mathscr{C}\sum_{i=0}^{n-1} N\bigl(d_X\circ\pi_X(\gamma'(t_i),\gamma'(t_{i+1})),d_Y\circ\pi_Y(\gamma(t_i), \gamma(t_{i+1}))\bigr).\]
    
            But $\pi_X(\gamma'(t))=\pi_X(\gamma(t))$, so that the above expression is the same as 
            \[\sup_\mathscr{C}\sum_{i=0}^{n-1} N\bigl(d_X\circ\pi_X(\gamma(t_i),\gamma(t_{i+1})),d_Y\circ\pi_Y(\gamma(t_i), \gamma(t_{i+1}))\bigr).\]
    
            The supremum of this expression among all chains is the length of $\gamma$. So we have shown that $\gamma'$ is no longer than a geodesic, and so it must be geodesic itself.
        \end{proof}
    
        We deduce the equivalence of the two definitions of $d'$ immediately.
    
        \begin{cor} \label{cor:EquivalentDefinitions}
            Definitions \ref{def:d'_initial} and \ref{def:d'_equivalent} are equivalent.
        \end{cor}
    
        \begin{proof}
            Definition \ref{def:d'_equivalent} is the infimal $d_{\bowtie}$-length of a path between a pair of points on $X\bowtie V$ that lies on $X\bowtie V$, and $X\bowtie V$ is a geodesic subspace for $d_{\bowtie}$. 
        \end{proof}

    \subsection{Rough isometry of $d_X$ and $d'_X$.}

        In this subsection, we will show that  $d'_X$ is hyperbolic and proper by showing that $d'_X$ differs from $d_X$ by a bounded amount. In an early draft of this paper, the authors learned that this fact follows from Ferragut's Theorem A in \cite{Ferragut:Visual_Boundary} together with Lemma \ref{lemma:MonotoneImpliesConvex}. However, we present the proof from scratch because all the pieces will be necessary in the following subsection as well.
        
        The upper bound on $d'_X-d_X$ is easily shown to be $0$.
    
        \begin{lemma}\label{lemma:RIUpperBound}
            The metric $d'_X$ satisfies
            \[d'_X\le d_X,\] 
            with $d'_X=d_X$ for points connected by vertical segments. 
        \end{lemma}
    	
        \begin{proof}
            Let $p$ and $q$ be any two points in $X$, and let $\gamma:[0,1]\to X$ be a $d_X$-geodesic between the two. We wish to compute the length of $\gamma$ with respect to $d'_X$. So consider any pair of points, $\gamma(t_1)$ and $\gamma(t_2)$. Since height change is $1$-Lipschitz, $|h(\gamma(t_1))-h(\gamma(t_2))|\le d_X(\gamma(t_1),\gamma(t_2))$. Then taking any finite chain of points in $\mathscr{C}$, the length associated to this chain is
            \begin{align*}
                \sum_{i=0}^{n-1} N\bigl(d_X(\gamma(t_i), \gamma(t_{i+1})), |h(\gamma(t_i))- h(\gamma(t_{i+1}))|\bigr) &\le \sum_{i=0}^{n-1} N\bigl(d_X(\gamma(t_i), \gamma(t_{i+1})), d_X(\gamma(t_i), \gamma(t_{i+1}))\bigr)\\
                &=\sum_{i=0}^{n-1} d_X(\gamma(t_i), \gamma(t_{i+1}))\\
                &=\ell_{d_X}(\gamma),
            \end{align*}
                
            where we used monotonicity in the first line and admissibility in the second. Hence, $\ell_{d'_X}(\gamma)\le \ell_{d_X}(\gamma)=d_X(p,q)$. 
                
            Moreover, if $\gamma$ is vertical, then the first line of the above sequence is an equality, so that $\ell_{d'_X}(\gamma)=\ell_{d_X}(\gamma)$. Monotonicity of $N$ then shows that no shorter path exists.
        \end{proof}
    
        The map $i_V:(X, d_X)\to (X\bowtie V, d_X')$ is therefore distance-non-increasing, and so continuous. Obtaining a bound on how much $i_V$ may shrink distances involves several steps. We denote $V_c(h_X)$ the total variation of the function $h$ along the path $c$.
        First, we will show that $d_X$-geodesics $\gamma$ have $V_\gamma(h_X)$ close to their $d_X$-length. Then we will show that $d_X'$-geodesics $\gamma'$ have $V_{\gamma'}(h_X)$ not much less than $V_\gamma(h_X)$, where $\gamma$ is the $d_X$-geodesic connecting the same pair of points. Using the fact that height variation is a lower bound for length, this will show that $\ell_{d_X'}(\gamma')$ is not much less than $V_\gamma(h_X)$, which is in turn not much less than $\ell_{d_X}(\gamma)$.
    	
       \begin{lemma}
       \label{lemma:BoxGeodesicsAreUniformlyRough}
            Let $X$ be $\delta$-hyperbolic. For any geodesic $\gamma$ in $X$ between points $p$ and $q$, $\ell_{d_X}(\gamma)-4\delta\le V_\gamma(h_X)$. Indeed, there is a point $r$ on $\gamma$ so that $\ell(\gamma)-4\delta\le |h_X(r)-h_X(p)|+|h_X(r)-h_X(q)|$. As a consequence, $\ell_{d_X'}(\gamma)\ge \ell_{d_X}(\gamma)-4\delta$.
            
            Moreover, if $h_X(r)\le \min\{h_X(p), h_X(q)\}$, then $d_X(p,q)\le 4\delta$.
        \end{lemma}
    	
    	
        \begin{proof}
            Let $p$ and $q$ be any two points in $X$, with a $d_X$-geodesic $\gamma$ between them. Consider the vertical rays $\eta_p$ and $\eta_q$ starting at $p$ and $q$.
    
            \begin{figure}
                \centering\includegraphics{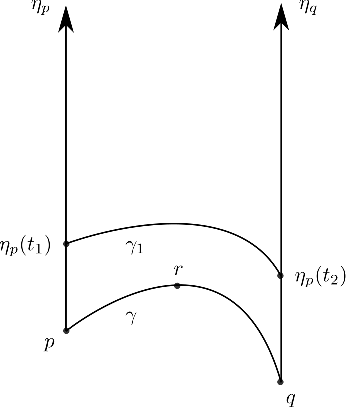}
                \caption{Proof of Lemma \ref{lemma:BoxGeodesicsAreUniformlyRough}}
                \label{fig:HeightVariationOfGeodesics}
            \end{figure}
        
            The triangle whose sides are $\gamma$, $\eta_p$, and $\eta_q$ is $\delta$-thin. So, as shown in Figure \ref{fig:HeightVariationOfGeodesics}, let $r$ be a point within $\delta$ of points $\eta_p(t_1)$ and $\eta_q(t_2)$. It follows that $|h_X(r)-h_X(\eta_p(t_1))|\le \delta$ and $|h_X(r)-h_X(\eta_q(t_2))|\le \delta$. Let $\gamma'$ be a geodesic between $\eta_p(t_1)$ and $\eta_q(t_2)$, necessarily of length no more than $2\delta$. Denote $*$ to be path concatenation, and consider the path 
            \[\xi=\eta_p|_{[0, t_1]}*\gamma'* \eta_q|_{[t_2, 0]}.\]
            
            Then we calculate
            
            \begin{align*}
                \ell(\gamma)-2\delta &\le \ell(\xi)-2\delta\\
                &\le |h_X(\eta_p(t_1))-h_X(p)|+|h_X(\eta_q(t_2))-h_X(q)|\\
                &\le |h_X(r)-h_X(p)|+\delta + |h_X(r)-h_X(q)|+\delta\\
                &\le V_{\gamma}(h_X)+2\delta.
            \end{align*}
            
            Therefore, $\ell_{d_X}(\gamma)-4\delta\le V_\gamma(h_X)$ as desired. Since $\ell_{d_X'}(\gamma)\ge V_\gamma(h_X)$, we determine that $\ell_{d_X'}(\gamma)\ge \ell_{d_X}(\gamma)-4\delta$
                
            For the final assertion, suppose $h_X(r)\le \min\{h_X(p), h_X(q)\}$. Then the point $r$ is within $\delta$ of points $\eta_p(t_1)$ and $\eta_q(t_2)$, and, since height is $1$-Lipschitz, these points cannot be above heights $h_X(p)+\delta$ and $h_X(q)+\delta$ respectively. Then
            \begin{align*}
                d_X(p,q)&\le d_X(p, \eta_p(t_1))+d_X(\eta_p(t_1), r)+d_X(r, \eta_q(t_2))+d_X(\eta_q(t_2), q)\\
                &\le 4\delta.
            \end{align*}
        \end{proof}
    	
        To show that $d_X'$-geodesics have large height variation, we must rule out the existence of paths in $X$ that are somewhat longer than the $d_X$-geodesics, but have much less height variation. The key observation in this analysis is a lemma due of Le Donne-Pallier-Xie.
    	\begin{lemma}[\cite{LDPX}, Lemma 2.3 (1)]\label{lemma:LDPXLengthLowerBound}
            Let $X$ be a proper geodesic $\delta$-hyperbolic metric space, $\infty_X$ a boundary point with associated height function $h$, and $S$ a horosphere about $\infty_X$ bounding horoball $B$. There is a constant $C$ depending only on $\delta$ so that the following holds. If $x_1$ and $x_2$ are points on $S$ and $c:[0, l]\to X\setminus B$ is a path between $x_1$ and $x_2$, then\[\ell_{d_X}(c)\ge 2(\max_c(h_X)-\min_c(h_X))+2^{\frac{d_X(x_1, x_2)-C-2}{2\delta}}-C-5d(x_1, x_2).\]
        \end{lemma}
    
        Using this lemma, we can show that a $d_X'$-geodesic cannot have much less height variation than the corresponding $d_X$-geodesic.
            
        \begin{prop} \label{prop:NewGeodesicsHaveLargeHeightVariation}
            There is a constant $C$ depending only on $\delta$ so that, for any $d_X'$-geodesic $\gamma'$ and $d_X$-geodesic $\gamma$ between points $p$ and $q$, $V_{\gamma'}(h_X)\ge V_{\gamma}(h_X)-C$.    
        \end{prop}
    
        \begin{proof}
            From Lemma \ref{lemma:BoxGeodesicsAreUniformlyRough}, we know that there is a point $r$ on $\gamma$ so that $V_{\gamma}(h_X)-C_1\le \ell_{d_X}(\gamma)-C_1\le |h_X(r)-h_X(p)| + |h(r)-h_X(q)|\le V_\gamma(h_X)\le \ell_{d_X}(\gamma).$ We consider cases based on the height of $r$. Assume WLOG $h_X(p)\ge h_X(q)$.
                
            \textbf{Case I:} $h_X(p)\ge h_X(q)\ge h_X(r)$.
                
            In this case, by the final piece of Lemma \ref{lemma:BoxGeodesicsAreUniformlyRough}, $d_X(p, q)\le 4\delta$. It follows that $V_\gamma(h_X)\le 4\delta$, so that $V_{\gamma'}(h_X)\le V_\gamma(h_X)-4\delta$.
    
            \textbf{Case II:} $h_X(p)\ge h_X(r)\ge h_X(q)$.
                
            In this case, $|h_X(r)-h_X(p)|+|h_X(r)-h_X(q)|=h_X(p)-h_X(q)$. Now, $V_{\gamma'}(h_X)\ge h_X(p)-h_X(q)$, which is at least $V_\gamma(h_X)-4\delta$.
    
            \textbf{Case III:} $h_X(r)\ge h_X(p)\ge h_X(q)$.
                
            In this case, we may as well replace $h_X(r)$ with $\max_\gamma (h_X)$ to get the inequality chain 
            \[V_{\gamma}(h_X)-4\delta\le \ell_{d_X}(\gamma)-4\delta\le 2\max_\gamma(h_X)-h_X(p)-h_X(q)\le V_\gamma(h_X)\le \ell_{d_X}(\gamma).\] 
            
            We now follow an argument from Le Donne-Pallier-Xie, depicted in Figure \ref{fig:PathsBelowMaxHeightAreLong}.
    
            \begin{figure}
                \centering\includegraphics{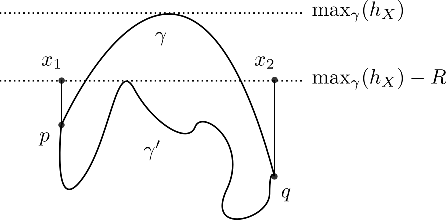}
                \caption{Proof of Proposition \ref{prop:NewGeodesicsHaveLargeHeightVariation}}
                \label{fig:PathsBelowMaxHeightAreLong}
            \end{figure}
                
            Let $R>0$, and suppose $\max_{\gamma'}(h_X)=\max_\gamma(h_X)-R$. Let the vertical geodesic rays from $p$ and $q$ be $\eta_p$ and $\eta_q$ respectively, and concatenate $\gamma'$ with segments of these rays so that the resulting path connects a point $x_1$ on $\eta_p$ at height $\max_\gamma(h_x)-R$ to a point $x_2$ on $\eta_q$ at the same height. Call this path $c$. Then $\ell_{d_X}(c)=2(\max_\gamma(h_X)-R)-h_X(p)-h_X(q)+\ell_{d_X}(\gamma')$. Applying Lemma \ref{lemma:LDPXLengthLowerBound} shows that
            \[ \ell_{d_X}(\gamma')\ge 2(\textstyle{\max_\gamma}(h_X)-R-\textstyle{\min_{\gamma'}}(h_X))+2^{\frac{d_X(x_1, x_2)-C_1-2}{2\delta}}-C_1-5d_X(x_1, x_2)-(2(\textstyle{\max_\gamma}(h_X)-R)-h_X(p)-h_X(q)),\]
            where $C_1$ depends only on $\delta$.
                
            Simplifying, and using the fact that $h_X(p)$ and $h_X(q)$ are evidently at least $\min_{\gamma'}(h_X)$,
                
            \[ \ell_{d_X}(\gamma')\ge 2^{\frac{d_X(x_1, x_2)-C_1-2}{2\delta}}-C_1-5d_X(x_1, x_2).\]
                        
            Applying the triangle inequality, we see that 
            \[d_X(x_1, x_2)\ge d_X(p, q)-(2\max_\gamma(h_X)-h_X(p)-h_X(q))+2R\ge (d_X(p,q)-V_\gamma(h_X))+2R\ge 2R.\]
            Consequently, there is some $R_0$ depending only on $C_1$ and $\delta$ (hence depending only on $\delta$) for which, if $R>R_0$ and a path $\gamma'$ from $p$ to $q$ has maximum height is $\max_\gamma(h_X)-R$, then $\ell_X(\gamma')\ge 2d_X(p, q)$. Since the norm $N$ is admissible, $N(a,b)\ge \frac{a+b}{2}$. Therefore, the $d'_X$-length of such a path must be at least $\frac{\ell_X(\gamma')}{2}\ge d_X(p, q)$, which renders $\gamma'$ longer than $\gamma$ in $d_X'$-length. So a $d_X'$-geodesic $\gamma'$ between $p$ and $q$ must have height variation at least $2(\max_{\gamma}(h_X)-R_0)-h_X(p)-h_X(q)\ge V_{\gamma}(h_X)-2R_0-4\delta$.
        \end{proof}
    
        \begin{cor}\label{cor:RILowerBound}
            There is a $C$ depending only on $\delta$ and the norm $N$ so that $d_X-C\le d'_X\le d_X$.
        \end{cor}
    
        \begin{proof}
            The upper bound is Lemma \ref{lemma:RIUpperBound}. For the lower bound, $V_\gamma'(h_X)\ge V_{\gamma}(h_X)-C_1$, where $C_1$ is the constant in Proposition \ref{prop:NewGeodesicsHaveLargeHeightVariation}. By Lemma \ref{lemma:BoxGeodesicsAreUniformlyRough}, $V_\gamma(h_X)\ge \ell_{d_X}(\gamma)-4\delta$. Since height variation is a lower bound on length for $d_X'$, we see that $\ell_{d'_X}(\gamma')\ge \ell_{d_X}(\gamma)-4\delta-C_1.$
        \end{proof}
    
        Corollary \ref{cor:RILowerBound} immediately implies that $(X, d'_X)$ is hyperbolic. 
            
        As an additional consequence, Proposition \ref{cor:RILowerBound} implies that $d'_X$ is a proper metric, by way of a basic exercise.
    
        \begin{lemma}
            Let $Z$ and $W$ be metric spaces, where $W$ is proper, and let $q:W\to Z$ be a continuous and surjective quasi-isometry. Then $Z$ is proper. \hfill\qed
        \end{lemma}
    
        \begin{cor}
            $(X, d'_X)$ is a proper metric space. \hfill \qed
        \end{cor}

    \subsection{Height changes for isometries of $X'$.}

        The final goal for this section is to show that isometries of $X'$ fixing $\infty_X$ have well-defined $h_X$-change. If $X'$ were $\CAT(-1)$ this would be easy, but we do not know whether this is true. 
        
        First of all, a note on terminology. To avoid confusion about which geodesic rays can be parameterized by height, we will call \textit{$d_X$-vertical geodesics} those geodesics for $d_X$ in the class $[\infty_X]$, which are necessarily also geodesic for $d'_X$ and possible to parameterize by height. In contrast, we will use the phrase ``$d'_X$ geodesics in $[\infty_X]$", rather than the (equivalent but confusing) ``$d_X'$-vertical geodesics", since such geodesics may not be possible to parameterize by height.
        
        To show that $\Isom_{\infty_X}(X')$ has a well-defined height change, we will first introduce an equivalent form for $h_X$ in terms of $d_X'$ Busemann functions about $d_X$-vertical rays. Then we will show that $h_X$ can also be computed from any $d_X'$-geodesic ray in the class $[\infty_X]$. This will allow us to compute the $h_X$-change using only $d_X'$, for which we know our map is an isometry.

        \begin{lemma} \label{lemma:VerticalRayHeightsInX'}
            Let $\eta$ be a $d_X$-vertical geodesic, parameterized by height, and let $x_0$ be any point at height $0$. Then 
            \[h_X(x)=\lim_{t\to\infty} d_X'(\eta(t), x_0)-d_X'(\eta(t), x).\]
        \end{lemma}
    
        \begin{proof}
            Let $\eta_x$ and $\eta_{x_0}$ be $d_X$-vertical rays starting at $x$ and $x_0$ respectively. Parameterize each one by height. Then vertical convergence of $d_X$-vertical rays in $X$ shows that \[ \lim_{t\to\infty}d_X(\eta(t), \eta_x(t))=\lim_{t\to\infty}d_X(\eta(t), \eta_{x_0}(t))=0.\]
            Since $d_X\ge d_{X}'$, it follows that 
            \[ \lim_{t\to\infty}d_X'(\eta(t), \eta_x(t))=\lim_{t\to\infty}d_X'(\eta(t), \eta_{x_0}(t))=0.\]

            But since $d_X'$-distances are bounded below by height difference, we see that $t-h_X(x)\le d_X'(\eta(t), x)$ and $t\le d_X'(\eta(t), x_0)$. Since $t-h_X(x)=d'_X(\eta_x(t), \eta_x(h_X(x)))$, the triangle inequality shows that $t-h_X(x)+d'_X(\eta(t), \eta_X(t))\ge d'_X(\eta(t), x)$. But $d'_X(\eta(t), \eta_X(t))\to 0$ so that, by the squeeze theorem, $\lim_{t\to\infty} d_X'(\eta(t), x)-t+h(x)=0$ and $\lim_{t\to\infty} d_X'(\eta(t), x_0)-t=0$. Subtracting the second limit from the first gives
            \[h_X(x)=\lim_{t\to\infty} d'_X(\eta(t), x)-d_X'(\eta(t), x_0).\]
        \end{proof}
    
        In order to show that we can calculate heights using $d_X'$-geodesic rays in $[\infty_X]$, we will first need an elementary estimate of how far from $d_X$-geodesic such rays can be.
        
        \begin{lemma}\label{d_X'GeodesicsAreAlmostd_XGeodesic}
            Let $\gamma'$ be a $d_X'$ geodesic. Then there is a constant $C$ depending only on $\delta$ so that $\ell_{d_X}(\gamma'|_{[a,b]})\le d_X(\gamma'(a),\gamma'(b))+C.$
        \end{lemma}
    
        \begin{proof}            
            Let $\gamma$ be a $d_X$-geodesic from $\gamma'(a)$ to $\gamma'(b)$. From Proposition \ref{prop:NewGeodesicsHaveLargeHeightVariation} we have a constant $C_1$ depending only on $\delta$ so that $V_{\gamma'|_{[a,b]}}(h_X)\ge V_{\gamma}(h_X)-C_1$. From Lemma \ref{lemma:BoxGeodesicsAreUniformlyRough}, we know that  $V_{\gamma}(h_X)\ge \ell_{d_X}(\gamma)-4\delta$. Hence $V_{\gamma'|_{[a,b]}}(h_X)\ge \ell_{d_X}(\gamma)-(C_1+4\delta)$.

            Now, since the norm $N$ is admissible, $N(c, d)\ge \frac{|c|+|d|}{2}$. By the additivity of the operation $(c, d)\mapsto \frac{|c|+|d|}{2}$ for positive $c$ and $d$, one sees immediately that $$\ell_{d_X}(\gamma'|_{[a,b]})\ge \frac{V_{\gamma'|_{[a,b]}}(h_X)+\ell_{d_X}(\gamma'|_{[a,b]})}{2}$$ (This fact is also \cite{Ferragut:Visual_Boundary} Property 3.4). Now, $\ell_{d_X}(\gamma'|_{[a,b]})\ge \ell_{d_X}(\gamma)$ because $\gamma$ is a $d_X$-geodesic connecting $\gamma'(a)$ to $\gamma'(b)$. Putting these pieces all together, we see that 

            \begin{align*}
                \ell_{d'_X}(\gamma'|_{[a,b]})&\ge\frac{V_{\gamma'|_{[a,b]}}(h_X)+\ell_{d_X}(\gamma'|_{[a,b]})}{2}\\
                &\ge \frac{V_{\gamma'|_{[a,b]}}(h_X)+\ell_{d_X}(\gamma)}{2}\\
                &\ge \frac{2\ell_{d_X}(\gamma)-C_1+4\delta}{2}\\
                &\ge \ell_{d_X}(\gamma)-\frac{C_1}{2}-2\delta\\
                &=d_X(\gamma'(a), \gamma'(b))-\frac{C_1}{2}-2\delta.
            \end{align*}

            Finally, since $d_X\ge d'_X$, we see that $\ell_{d_X}(\gamma'|_{[a,b]})\ge \ell_{d'_X}(\gamma'|_{[a,b]})$, and so the result follows.
        \end{proof}
        
        Next, we show that the tails of rays in $[\infty_X]$ satisfying the conclusions of Lemma \ref{d_X'GeodesicsAreAlmostd_XGeodesic} are well-approximated by vertical lines. The following lemma is written in somewhat more generality than needed here, because we will re-use it in the next section in application to $X\bowtie Y$.
    
        \begin{lemma}\label{lemma:VerticalApproximation}
            Let $(Z, d_Z)$ be any geodesic metric space with a $1$-Lipschitz function $h_Z:Z\to \R$, which we will call a height function, and suppose that $\gamma'$ is a continuous and rectifiable ray (or line) parameterized by length so that for all $s, t$,  $d_Z(\gamma'(s), \gamma'(t))\ge |s-t|-C_1.$ Let $\eta$ be a geodesic ray (or line) parameterized by length starting at a point at height $0$, such that $h_Z(\eta(s))-h_Z(\eta(t))=s-t$, and $\gamma'\subset \mathscr{N}_{C_2}(\eta)$ for some $C_2>0$. Then for any $\epsilon>0$, there is a $t_\epsilon$ so that $s\ge t\ge t_{\epsilon}$ implies $h_Z(\gamma'(s))-h_Z(\gamma'(t))\ge s-t-\epsilon.$ Additionally, if $\gamma'$ is a line, there is also a $t_\epsilon'$ so that $s'\le t'\le t_{\epsilon}'$ implies $h_Z(\gamma'(t))-h_Z(\gamma'(s))\ge t-s-\epsilon.$
        \end{lemma}
    
        It is not hard to see that the first assumption shows that $\gamma'$ is roughly geodesic, and thus that if $(Z, d_Z)=(X, d_X)$, then $\gamma'$ always has a Hausdorff-close $d_X$-geodesic ray or line. If $\gamma'$ is in $[\infty_X]$, this geodesic will then be required to be $d_X$-vertical. So this lemma and the last one together tell us that if $\gamma'$ is a $d_X'$-geodesic in $[\infty_X]$, then we can find a point after which it is close to vertical.
    
        \begin{proof}
            First, suppose $\gamma'$ is a ray.
            
            Consider the function $f(t)=h_Z(\gamma'(t))-t$, and notice that it is non-increasing. Notice also that each point on $\gamma'$ is within $C_2$ of a point on $\eta$, and thus within $2C_2$ of a point of matching height. That is, $d_Z(\gamma'(t), \eta(h_Z(\gamma'(t))))\le 2C_2$. Then by the triangle inequality
    
            \begin{align*}
                f(t)-f(0) &= h_Z(\gamma'(t))-h_Z(\gamma'(0))-t\\
                &\ge h_Z(\gamma'(t))-h_Z(\gamma'(0))-d_Z(\gamma'(t), \gamma'(0))-C_1\\
                &\ge h_Z(\gamma'(t))-h_Z(\gamma'(0))-d_Z\Big{(}\gamma'(t), \eta(h_Z(\gamma'(t)))\Big{)}\\&-d_Z\Big{(}\eta(h_Z(\gamma'(t))), \eta(h_X(\gamma'(0)))\Big{)}-d_Z\Big{(}\eta(h_Z(\gamma'(0))), \gamma'(0)\Big{)}-C_1\\
                &\ge -4C_2-C_1.
            \end{align*}
    
            The function $f$ is therefore bounded below and non-increasing, so it has a horizontal asymptote as $t\to \infty$. Now, let $t_{\epsilon}$ be the time such that $f(t_{\epsilon})-\lim_{t\to\infty}f(t)=\epsilon$. Let $s\ge t\ge t_{\epsilon}$. We compute that $f(t)-f(s)\le \epsilon$, so that
            \[_Z(\gamma'(t))-t-h_Z(\gamma'(s))+s\le \epsilon.\]
            This rearranges to 
            \[h_Z(\gamma'(s))-h_Z(\gamma'(t))\ge s-t-\epsilon.\]
    
            When $\gamma'$ is a line, we split it into two rays $\gamma'_1=\gamma'|_{[0, \infty)}$ and $\gamma'_2=\gamma'|_{(-\infty, 0]}$ parameterized in reverse, and run the above argument twice.
        \end{proof}
    
        In light of Lemma \ref{lemma:AlmostVerticalSegments}, the tail of a $d_X'$-geodesic $\gamma'$ in $[\infty_X]$ can be approximated arbitrarily closely by vertical rays. We are therefore prepared to show that we compute $h_X$ using the $d_X'$-Busemann function about an arbitrary $d_X'$-geodesic ray $\gamma'$ in $[\infty_X]$, whether or not the $\gamma'$ is $d_X$-vertical.
    
        \begin{prop} \label{prop:ArbitraryRayHeightsInX'}
            Let $\gamma'$ be a $d_X'$-geodesic ray in $[\infty_X]$, and $x_0\in X$ any point at height $0$. Then \[\lim_{t\to\infty}  d_X'(\gamma'(t), x_0)-d_X'(\gamma'(t), x)= h_X(x)\]
        \end{prop}
        
        \begin{proof}
            Let $\gamma'$ be a $d_X'$-geodesic ray in $[\infty_X]$, and let $\epsilon>0$. By Lemma \ref{d_X'GeodesicsAreAlmostd_XGeodesic}, $\gamma'$ is roughly geodesic with respect to $d_X$, and therefore within bounded distance of a $d_X$-vertical geodesic $\eta$. Also by Lemma \ref{d_X'GeodesicsAreAlmostd_XGeodesic}, $\gamma'$ is $d_X$-rectifiable, so we parameterize it by $d_X$ arc length. Using Lemma \ref{lemma:VerticalApproximation}, find $t_{\epsilon'}$ after which $h_X(\gamma'(s))-h_X(\gamma'(t))\ge s-t-\epsilon'$. Choose $\epsilon'$ small enough that $f(\epsilon')\le \frac{\epsilon}{4}$, where $f$ is the function in Lemma \ref{lemma:AlmostVerticalSegments}. Now, since $\gamma'$ is parameterized by $d_X$ arc length, $d_X(\gamma'(s), \gamma'(t))\le |s-t|$ for all $s$ and $t$. Therefore, when $s>t>t_{\epsilon'}$, we see that $d_X(\gamma'(s), \gamma'(t)) \le s-t\le h_X(\gamma'(s))-h_X(\gamma'(t))+\epsilon'$. Applying Lemma \ref{lemma:AlmostVerticalSegments}, we see that for $t>t_{\epsilon'}$, $\gamma'|_{[t, \infty)}$ is within $\frac{\epsilon}{4}$ of the $d_X$-vertical line $\eta_{\gamma'(t)}$ passing through $\gamma'(t)$. It follows that for $s>t>\epsilon'$, $\gamma'(s)$ is within $\frac{\epsilon}{2}$ of a point on $\eta_{\gamma'(t)}$ of matching height.
    
            As in the previous proof, we know that $h_X(\gamma'(t))$ is at least $t-C$ and asymptotic to $t-C$ for some fixed constant $C$. Thus for any $x\in X$, $d'_X(\gamma'(t),x)-t$ is decreasing and bounded below by $-h_X(x)-C$. The same is true for $d'_X(\eta_{\gamma'(t_{\epsilon'})}(h(\gamma'(t))),x)-t$. Choosing $x_0$ to be a point at height $0$, and applying the above twice, we see that both
    
            \[\lvert \lim_{t\to \infty} d_X'(\gamma'(t), x)-d_X'(\eta_{\gamma'(t_{\epsilon'})}(h(\gamma'(t))), x)\rvert\]
            and
            \[\lvert \lim_{t\to\infty} d_X'(\gamma'(t), x_0)-d_X'(\eta_{\gamma'(t_{\epsilon'})}(h(\gamma'(t))), x_0)\rvert\]
            exist. It follows that they are at most $\frac{\epsilon}{2}.$ 
    
            But then
            \begin{align*}
                \left\lvert \lim_{t\to \infty} d_X'(\gamma'(t), x)-d_X'(\eta_{\gamma'(t_{\epsilon'})}(h(\gamma'(t))), x)\right\rvert + \left\lvert \lim_{t\to\infty} d_X'(\gamma'(t), x_0)-d_X'(\eta_{\gamma'(t_{\epsilon'})}(h(\gamma'(t))), x_0)\right\rvert &\le \epsilon\\
                \left\lvert \lim_{t\to \infty} d_X'(\eta_{\gamma'(t_{\epsilon'})}(h(\gamma'(t))), x)-d_X'(\gamma'(t), x)\right\rvert + \left\lvert \lim_{t\to\infty} d_X'(\gamma'(t), x_0)-d_X'(\eta_{\gamma'(t_{\epsilon'})}(h(\gamma'(t))), x_0) \right\rvert &\le \epsilon.
                \intertext{The triangle inequality for absolute values yields}                
                \left\lvert \lim_{t\to \infty} \Big{[}d_X'(\eta_{\gamma'(t_{\epsilon'})}(h(\gamma'(t))), x)-d_X'(\gamma'(t), x)\Big{]} + \lim_{t\to\infty} \Big{[} d_X'(\gamma'(t), x_0)-d_X'(\eta_{\gamma'(t_{\epsilon'})}(h(\gamma'(t))), x_0)\Big{]}\right\rvert &\le \epsilon.\\
                \intertext{We then exchange terms between the limits}
                \left\lvert \lim_{t\to \infty} \Big{[}d_X'(\eta_{\gamma'(t_{\epsilon'})}(h(\gamma'(t))), x)-d_X'(\gamma'(t), x)+d_X'(\gamma'(t), x_0)-d_X'(\eta_{\gamma'(t_{\epsilon'})}(h(\gamma'(t))), x_0)\Big{]}\right\rvert &\le \epsilon\\
                \left\lvert \lim_{t\to \infty} \Big{[}d_X'(\gamma'(t), x_0)-d_X'(\gamma'(t), x)\Big{]} - \lim_{t\to\infty}\Big{[}d_X'(\eta_{\gamma'(t_{\epsilon'})}(h(\gamma'(t))), x_0)-d_X'(\eta_{\gamma'(t_{\epsilon'})}(h(\gamma'(t))), x)\Big{]}\right\rvert &\le \epsilon
            \end{align*}
    
            where the last line is meant only if one, and thus both, of the limits exist. Using again the fact that $h(\gamma'(t))\ge t-C$ and is asymptotic to $t-C$, we see that 
            \[ \lim_{t\to\infty}\Big{[}d_X'(\eta_{\gamma'(t_{\epsilon'})}(h(\gamma'(t))), x_0)-d_X'(\eta_{\gamma'(t_{\epsilon'})}(h(\gamma'(t))), x)\Big{]}=\lim_{t\to\infty}\Big{[}d_X'(\eta_{\gamma'(t_{\epsilon'})}(t), x_0)-d_X'(\eta_{\gamma'(t_{\epsilon'})}(t), x)\Big{]},\]
    
            where again we mean that one limit exists if and only if the other does. Applying Lemma \ref{lemma:VerticalRayHeightsInX'}, we see that the righthand side exists and is equal to $h_X(x).$ As a consequence, we substitute to obtain
    
            \[\lvert \lim_{t\to \infty} \Big{[}d_X'(\gamma'(t), x_0)-d_X'(\gamma'(t), x)\Big{]} - h_X(x)\rvert \le \epsilon. \]
    
            Sending $\epsilon$ to $0$ gives the desired result.
        \end{proof}
    
        We are finally prepared to show that isometries of $X'$ fixing $\infty_X$ have well-defined $h_X$-change.
    
        \begin{thm}
            Let $f:X'\to X'$ be in $\Isom(X')_{\infty_X}$. Then $h_X(f(x))-h_X(x)$ does not depend on the choice of $x$.
        \end{thm}
    
        \begin{proof}
            By Lemma \ref{lemma:VerticalRayHeightsInX'}, we choose a $d_X$-vertical line $\eta$ and a point $x_0$ at height $0$ so that,
            \[h_X(f(x))-h_X(x)= \lim_{t\to\infty} d_X'(\eta(t), x_0)-d_X'(\eta(t), f(x)) -  d_X'(\eta(t), x_0)+d_X'(\eta(t), x).\]
    
            Now, since $f$ is an isometry fixing $\infty_X$, $f(\eta)=\gamma'$ is a $d_X'$-geodesic in $[\infty_X]$. We can therefore compute
    
            \begin{align*}
            h_X(f(x))-h_X(x)&=\lim_{t\to\infty} d_X'(\eta(t), x_0)-d_X'(\eta(t), f(x))-d_X'(\eta(t), x_0)+d_X'(\eta(t), x)\\
            &=\lim_{t\to\infty} d_X'(\eta(t), x_0)-d_X'(\eta(t), f(x))-d_X'(\gamma'(t), f(x_0))+d_X'(\gamma'(t), f(x))\\
            &=\lim_{t\to\infty}\Big{[} \big{(}d_X'(\eta(t), x_0)-d_X'(\eta(t), f(x))\big{)}+\big{(}d_X'(\gamma'(t), f(x))-d_X'(\gamma'(t), x_0)\big{)}\Big{]}\\
            &\qquad +\lim_{t\to\infty} d_X'(\gamma'(t), x_0)-d_X'(\gamma'(t), f(x_0)).
            \end{align*}
    
            Applying Lemma \ref{lemma:VerticalRayHeightsInX'} to the first two terms on the third line, and Proposition \ref{prop:ArbitraryRayHeightsInX'} to the second two terms, we see that 
            \[h_X(f(x))-h_X(x)=h_X(f(x))-h_X(f(x))+h_X(f(x_0)) = h_X(f(x_0)),\]
            which does not depend on the choice of $x$.
        \end{proof}

\section{Isometry groups of horocyclic products}
\label{sec:Isom_group}

In this section we will prove \cref{thm:IsomGroupOfHorocyclicProduct}, which says that essentially every isometry $f \in \Isom(X\bowtie Y)$ takes the form $f(x,y) = (f_X(x), f_Y(y))$ for $f_X \in \Isom(X')_\infty$ and $f_Y \in \Isom(Y')_\infty$ satisfying $h(f_X) = -h(f_Y)$. This will be proven in three steps:

\begin{enumerate}

    \item Show that the only geodesic connecting two points on opposite boundaries in $\partial (X\bowtie Y) = \partial_\ell X \sqcup \partial^u Y$ is the vertical geodesic.
    \item Show that the action of any isometry $f \in \Isom(X\bowtie Y)$ on $\partial (X\bowtie Y) = \partial_\ell X \sqcup \partial^u Y$ respects the upper-lower partition. In particular, isometries send vertical geodesics to vertical geodesics.
    
    \item Use the above facts to construct coordinate maps $f_X$ and $f_Y$ for each map $f$ stabilizing $\partial^u Y$ and $\partial_\ell X$, satisfying $f(x,y) = (f_X(x), f_Y(y))$, and show that they are isometries for $d'_X$ and $d'_Y$. If $f$ does not stabilize $\partial^u Y$ and $\partial_\ell X$, use the above facts instead to construct an embedded copy of $Y'$ that $f$ sends isometrically to an embedded copy of $X'$.
\end{enumerate}

Throughout this section, $N$ will be an admissible monotone norm giving rise to a metric $d_{\bowtie}$ on $X\bowtie Y$, where $X$ and $Y$ are $\delta$-hyperbolic, proper, geodesically-complete, Busemann, and vertically-convergent, and neither one is a line.

\subsection{Step 1: Points in opposite boundaries are connected by a unique geodesic.}

Let $\xi_X \in \partial X \setminus \{\infty_X\}$ and $\xi_Y \in \partial Y \setminus \{\infty_Y\}$. By Lemma \ref{lemma:VerticalUniquenessInFactorSpaces}, there are unique vertical geodesics $V \subset X$, $W\subset Y$ limiting to $\xi_X$ and $\xi_Y$ respectively. Then $V\bowtie W\subset X\bowtie Y$ is a vertical geodesic connecting $\xi_X,\xi_Y \in \partial X\bowtie Y$. Our goal in this subsection is to prove that $V\bowtie W$ is the only such geodesic.

    \begin{prop}
    \label{prop:Opposite_Boundaries_Unique_Geodesic}
        Let $(X, d_X)$ and $(Y, d_Y)$ be have height functions $h_X$ and $h_Y$ defined about boundary points $\infty_X$ and $\infty_Y$. Equip $X\bowtie Y$ with an admissible metric $d_{\bowtie}$. If $\xi_X\ne \infty_X$ and $\xi_Y\ne \infty_Y$ are any two points in $\partial X$ and $\partial Y$, then there is a unique geodesic in $X\bowtie Y$ between $\xi_X$ and $\xi_Y$, and this geodesic is vertical.
    \end{prop}
    
    \begin{proof}

        Let $d_{\bowtie}$ arise from an admissible norm $N$.

        \textbf{Step I: Construct vertical geodesics}
        
        Existence is straightforward. $X$ and $Y$ have the property that between any two boundary points there is a geodesic. Let $\eta_X$ (resp. $\eta_Y$) be geodesics between $\xi_X$ and $\infty_X$ (resp. between $\xi_Y$ and $\infty_Y$). Ferragut Proposition 2.7 shows that the arc-length parameterization of $\eta_X$ and $\eta_Y$ is the height parameterization, up to a constant. So we may assume that $h_X(\eta_X(t))=t$, and $h_Y(\eta_Y(t))=t$. The path $\eta(t)=(\eta_X(t), \eta_Y(-t))$ is in $X\bowtie Y$. One sees immediately that this path is a geodesic: 
        \[d_{\bowtie} \Big((\eta_X(t_1), \eta_Y(t_1)),(\eta_X(t_2), \eta_Y(t_2))\Big)\ge |t_1-t_2|\]
        because the height function is 1-Lipschitz, and the length of $\gamma$ in restriction to this segment is exactly $|t_1-t_2|$. This is the only vertical geodesic between the two, which follows from applying Lemma \ref{lemma:VerticalUniquenessInFactorSpaces} to each space. So it suffices to show that any geodesic connecting $\xi_X$ and $\xi_Y$ must be vertical.
        
        Therefore, let $\gamma=(\gamma_X, \gamma_Y)$ be a geodesic connecting $\xi_X$ and $\xi_Y$ that is not vertical, so that there is an $\epsilon>0$, and times $t_1<t_2$ so that $h(\gamma(t_2))-h(\gamma(t_1))=t_2-t_1-\epsilon$. 

       \textbf{Step II: Find a region $[a,b]$ outside which $\gamma$ is close to vertical.}

       By Ferragut \cite{Ferragut:Visual_Boundary} Corollary \textbf{C}, $\gamma$ is within an $M$-neighborhood of a vertical geodesic $\eta$ in $X\bowtie Y$, where $M$ depends only on $d_{\bowtie}$. Parameterize $\gamma$ by arc length starting at any point on $\gamma$ with height $0$.


       Let $\epsilon'>0$ so that $\epsilon'<\frac{\epsilon}{8}$ and $f(2\epsilon')<\frac{\epsilon}{8}$ where $f$ is the function in Lemma \ref{lemma:AlmostVerticalSegments}. Using Lemma \ref{lemma:VerticalApproximation} we see that there is an interval $[a,b]$, which we can take to include the points $t_1$ and $t_2$, so that for $s>t>b$ or $t<s<a$,
       \[h(\gamma(s))-h(\gamma(t))\ge s-t-\epsilon. \]
       Notice also that since $[a,b]$ contains $[t_1, t_2]$, we know that $h(\gamma(b))-h(\gamma(a))\le b-a-\epsilon$, which implies that
       \[d_{\bowtie}(\gamma(b), \gamma(a))=b-a\ge h(\gamma(b))-h(\gamma(a))+\epsilon.\]
       Our goal will be to find a shortcut between two points $\gamma(a_1)$ and $\gamma(b_1)$ where $a_1\le a$ and $b_1\ge b$.

       \textbf{Step III: Bound the distance between the tails of $\gamma_X$ and $\gamma_Y$ and vertical geodesics}
        
        \begin{figure}
            \centering
            \includegraphics[scale=0.6]{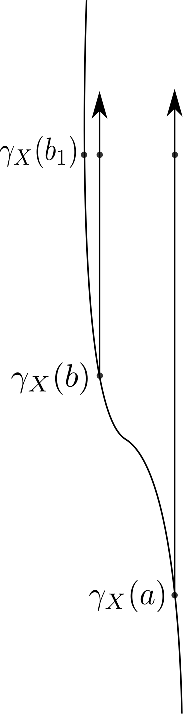}
            \caption{Past $b_1$, the path $\gamma_X$ is close to the vertical ray starting at $\gamma_X(b)$, and the vertical ray starting at $\gamma_X(b)$ is close to the vertical ray starting at $\gamma_X(a)$.}
            \label{figure:CloseDistanceFromA}
        \end{figure}
        
        Let $\eta_{X, \gamma_X(a)}$ and $\eta_{X, \gamma_X(b)}$ be the vertical geodesic rays in $X$ starting at $\gamma_X(a)$ and $\gamma_X(b)$ respectively. Consider a point $\gamma_X(t)$ for $t>b$. We wish to apply Lemma \ref{lemma:AlmostVerticalSegments} to show that $\gamma_X(t)$ is close to $\eta_{X, \gamma_X(b)}$. In order to do this, we must show that the $X$-distance between $\gamma_X(t)$ and $\gamma_X(b)$ is almost the height difference. 
        
        As shown in Figure \ref{figure:CloseDistanceFromA}, since $t>b$, the statement that $f(b)$ is near its asymptote tells us that $h(\gamma(t))-h(\gamma(b))\ge t-b-\epsilon'$. It follows that $d_X(\gamma_X(b), \gamma_X(t))\le t-b+\epsilon'$, for otherwise, for any chain of times $b=t_0<t_1<...t_n=t$, we compute the length of $\gamma$ along this chain as
        \begin{align*}
            \sum_{i=0}^{n-1} N\Big(d_X(\gamma_X(t_i), \gamma_X(t_{i+1})), d_Y(\gamma_Y(t_i), \gamma_Y(t_{i+1}))\Big)
            &\ge \frac{\sum_{i=0}^{n-1} d_X(\gamma_X(t_i), \gamma_X(t_{i+1})) + \sum_{i=0}^{n-1}d_Y(\gamma_Y(t_i), \gamma_Y(t_{i+1}))}{2}\\
            &\ge \frac{d_X(\gamma_X(b), \gamma_X(t)), d_Y(\gamma_Y(b), \gamma_Y(t))}{2}\\
            &> \frac{t-b+\epsilon'+t-b-\epsilon'}{2}\\
            &=t-b,
        \end{align*}
        where we used the fact that the height is $1$-Lipschitz in $Y$ in the second-to-last step. Hence, $$d_X(\gamma_X(t), \gamma_X(b))\le h(\gamma_X(t))-h(\gamma_X(b))+2\epsilon'.$$ 
        
        Therefore, applying Lemma \ref{lemma:AlmostVerticalSegments} to points $\gamma_X(t)$ for $t>b$ shows that each point of $\gamma_X([b, \infty))$ is less than $\frac{\epsilon}{8}$ from a point on $\eta_{X, \gamma_X(b)}$. Hence each point is less than $\frac{2\epsilon}{ 8}$ from the point on $\eta_{X, \gamma_X(b)}$ of matching height. 
        
        Now, consider the ray $\eta_{X, \gamma_X(a)}$. There is some height $H$, after which the distance between points at the same height on $\eta_{X, \gamma_X(a)}$ and $\eta_{X, \gamma_X(b)}$ are less than $\frac{\epsilon}{8}$ of one another by assumption. It follows that, past height $H$, the distance between a point on $\gamma_X$ and the point on $\eta_{X, \gamma_X(a))}$ is less than $\frac{3\epsilon}{8}$. Take $b_1$ so that $\gamma(b_1)$ is a point at height $H$.
        
        Next, reverse the roles of $X$ and $Y$, and apply the same argument to the interval $[a, b_1]$ to find an $a_1<a$ past which the $y$-distance between points at the same height on the vertical ray starting at $\eta_{Y, \gamma_Y(-b_1)}$ and on $\gamma_Y$ is smaller than $\frac{3\epsilon}{8}$.

        \textbf{Step IV: Find a single vertical segment that is close to $\gamma(a_1)$ and $\gamma(b_1)$}

        At this point we know that $\gamma_X(b_1)$ and $\eta_{X, \gamma_X(a)}(h_X(\gamma_X(b_1)))$ are near one another, and the same for $\gamma_Y(-a_1)$ and $\eta_{Y, \gamma_Y(-b_1)}(h_Y(\gamma_Y(-a_1)))$. If we take vertical segment 
        \[\eta_1=\eta_{X, \gamma_X(a)}\bowtie \eta_{Y, \gamma_Y(-b_1)},\]
        we have control over the quantities
        
        \[ d_X\Big{(}\gamma(b_1), \eta_1(h(\gamma(b_1)))\Big) \quad d_Y\Big{(}\gamma(b_1), \eta_1(h(\gamma(b_1)))\Big{)} \quad d_X\Big{(}\gamma(a), \eta_1(h(\gamma(a)))\Big{)} \quad d_Y\Big{(}\gamma(a_1), \eta_1(h(\gamma(a_1)))\Big{)}.\]
        
        Lemma \ref{lemma:BoundedCoordinateDinstancesImpliesBoundedDistance} therefore shows that $\eta_1$ comes close to $\gamma(b_1)$, but we do not know whether $\eta_1$ comes close to $\gamma(a_1)$. To construct the desired shortcut, we need a vertical segment $\eta_2$ that comes close to both $\gamma(a_1)$ and $\gamma(b_1)$, so we will need to make a slight adjustment to the $X$-coordinate of $\eta_1$.
        
        \begin{figure}
            \centering
            \includegraphics[scale=0.6]{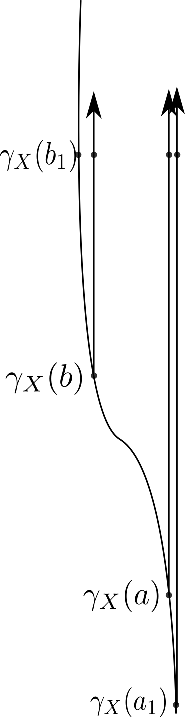}
            \caption{Since $a_1<a$, we again see that $\gamma_X(a)$ is near the vertical ray from $\gamma_X(a_1)$. Therefore, past $b_1$, $\gamma_X$ is close to the vertical ray from $\gamma_X(a_1)$}
            \label{figure:CloseDistanceFromA1}
        \end{figure}

        
         As shown in Figure \ref{figure:CloseDistanceFromA1}, since $a_1<a$, it follows that $h(\gamma(a))-h(\gamma(a_1))\ge a-a_1-\epsilon'$. Then $a-a_1-\epsilon'$ bounds below the $Y$-distance between $\gamma_Y(a)$ and $\gamma_Y(a_1)$. Then the $X$-distance between $\gamma_X(a)$ and $\gamma_X(a_1)$ can be at most $a-a_1+\epsilon'$. By the second part of the conclusion of Lemma \ref{lemma:AlmostVerticalSegments}, and the fact that $f(2\epsilon')<\frac{\epsilon}{8}$, each point of $V_{X, \gamma_X(a)}$ is within $\frac{\epsilon}{8}$ in $X$-distance from a point on $V_{X, \gamma_X(a_1)}$, and thus within distance $\frac{2\epsilon}{8}$ of a point at the same height.
        
        Using the triangle inequality, we determine that, at all heights $b_1$ and above, the $X$-distance between points of matching heights on $\gamma_X$ and $V_{X, \gamma_X(a_1)}$ is no more than $\frac{5\epsilon}{8}$. 

        \textbf{Step V: Constructing the shortcut}
        
        \begin{figure}
            \centering
            \includegraphics[scale=.4]{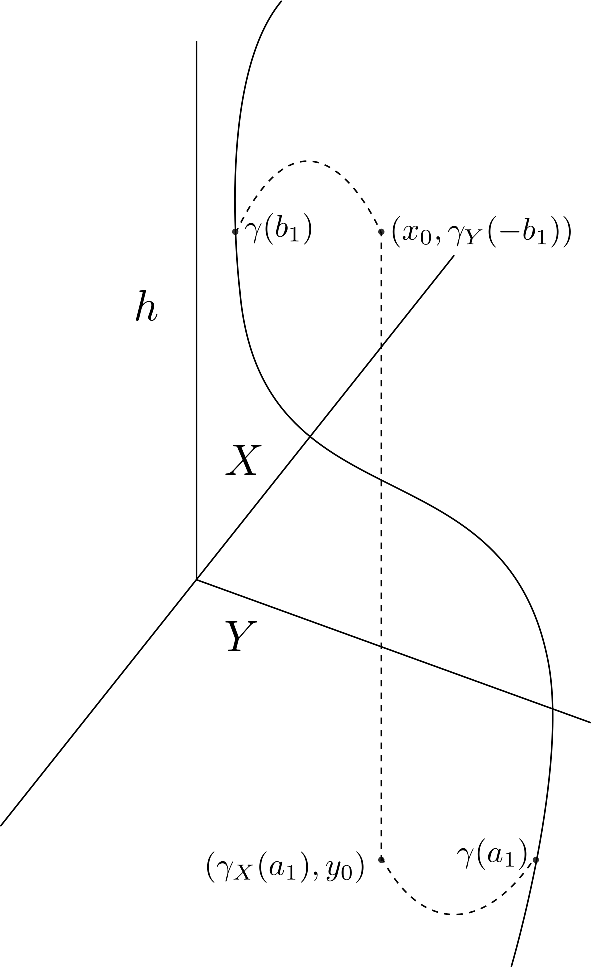}
            \caption{Construction of the final shortcut. $x_0$ and $y_0$ are the $x$- and $y$-coordinates of vertical geodesics starting at $\gamma_X(a_1)$ and $\gamma_Y(-b_1)$ respectively.}
            \label{figure:FinalShortcut}
        \end{figure}
        
        We will finally construct a shortcut between $\gamma(a_1)$ and $\gamma(b_1)$ using the vertical line $\eta_{X, \gamma_X(a_1)}\bowtie \eta_{Y, \gamma_Y(b_1)}$ as shown in Figure \ref{figure:FinalShortcut}. Recall that $\gamma|_{[a_1, b_1]}$ has length $b_1-a_1$, but that the height difference between $\gamma(b_1)$ and $\gamma(a_1)$ is no more than $b_1-a_1-\epsilon$. Consider the vertical geodesic 
        \[\eta_2=\eta_{X, \gamma_X(a_1)}\bowtie \eta_{Y, \gamma_Y(-b_1)}\]
        of length $h(\gamma(b_1))-h(\gamma(a_1))\le b_1-a_1-\epsilon$. Its starting point has $X$-coordinate $\gamma_X(a_1)$ and $Y$-coordinate at distance less than $\frac{3\epsilon}{8}$ from $\gamma_Y(a_1)$. Lemma \ref{lemma:BoundedCoordinateDinstancesImpliesBoundedDistance} therefore tells us that $d_{\bowtie}(\eta_2(h(\gamma(a_1))), \gamma(a_1))<\frac{3\epsilon}{8}$. Similarly, the endpoint of $\eta_2$ has $Y$-coordinate equal to $\gamma_Y(b_1)$, and $x$-coordinate at distance at most $\frac{5\epsilon}{8}$ from $\gamma_X(b_1)$, so that Lemma \ref{lemma:BoundedCoordinateDinstancesImpliesBoundedDistance} implies $d_{\bowtie}(\eta_2(h(\gamma(b_1))), \gamma(b_1))<\frac{5\epsilon}{8}$. But then by the triangle inequality, $d_{\bowtie}(\gamma(a_1), \gamma(b_1))<b_1-a_1-\epsilon+\frac{3\epsilon}{8}+\frac{5\epsilon}{8}$. So $\gamma$ was not geodesic.
    \end{proof}

\subsection{Step 2: Isometries respect the boundary decomposition and height.} \label{subsec:IsomsPreserveUpperAndLower}

By Proposition \ref{prop:Opposite_Boundaries_Unique_Geodesic}, points on opposite boundaries are connected by a unique geodesic (the vertical one). We would like to use this to conclude that vertical geodesics are mapped to vertical geodesics under isometries, but we do not yet know that isometries send pairs of opposite boundary to opposite boundary points. This possibility is addressed by the following lemma.

\begin{lemma}
    \label{lem:preserves_partition}
    Let $f \in \Isom(X\bowtie Y)$. If $\xi_1,\xi_2 \in \partial (X\bowtie Y)$ lie on opposite boundaries, then $\partial f (\xi_1)$ and $\partial f(\xi_2)$ lie on opposite boundaries. In particular, the action of $\Isom(X\bowtie Y)$ preserves the partition $\partial (X\bowtie Y) = \partial_\ell X \sqcup \partial_u Y$.
    
\end{lemma}

\begin{proof}
    Let $\eta_1,\eta_2$ lie in the same boundary---without loss of generality, say $\eta_1,\eta_2 \in \partial_\ell X$. We will show that there are multiple geodesics connecting $\eta_1,\eta_2$ in $X\bowtie Y$. Since points on opposite boundaries are connected by a unique geodesic (Proposition \ref{prop:Opposite_Boundaries_Unique_Geodesic}), this implies that $\partial f(\eta_1)$ and $\partial f(\eta)$ lie on the same boundary, proving the Lemma.

    The two geodesics connecting $\eta_1$ and $ \eta_2$ will be constructed as the image of a geodesic in $X'$ under two distinct embeddings $X' \to X\bowtie \alpha_i$, for distinct vertical geodesics $\alpha_1,\alpha_2 \subset Y$. Since $X'$ is quasi-isometric to $X$, we can consider $\eta_1,\eta_2$ as boundary points of $X'$. Lemma \ref{lemma:MonotoneImpliesConvex} implies that any two boundary points in $X'$ are connected by a geodesic, $\gamma$. 
    
    Since $Y$ is not a quasi-line (as $\infty_Y$ is not isolated), there are at least two vertical geodesics $\alpha_1,\alpha_2 \subset Y$, giving distinct embeddings $X' \to X\bowtie \alpha_i\subset X\bowtie Y$. This is not sufficient to produce distinct geodesics connecting $\eta_1,\eta_2$, since the geodesics $\alpha_i$ may agree above some height. Let $H$ be the maximal height of $\gamma$. If $\alpha_1(-H) \neq \alpha_2(-H)$, then the two images of $\gamma$ are distinct. Two such vertical geodesics $\alpha_1,\alpha_2$ are guaranteed to exist by the following lemma.

\end{proof}

\begin{lemma}
    If $\{\infty_Y\}$ is not isolated, then for all heights $H$, there are vertical geodesics $\alpha_1,\alpha_2$ parameterized by height satisfying $\alpha_1(H)\neq \alpha_2(H)$. 
\end{lemma}

\begin{proof}
    We prove the contrapositive. Assume that there is some height $H$ so that for any height-parameterized vertical geodesics $\alpha_i$, there is a ``bottleneck point" $\alpha_1(H) = \alpha_2(H)$. Since they are vertical geodesics, for any $t \geq H$, we also have $\alpha_1(t) = \alpha_2(t)$. Equivalently, the $H$-horoball is a ray.

    Choose a basepoint $y_0 = \alpha_1(H+1)$. Let $\gamma$ be the geodesic ray based at $y_0$ representing $\infty_Y$. Set $K = [0,1]$, and $U = \Nbhd(\gamma,1/2)$, the $1/2$-neighborhood of $\gamma$. Then $$\{\delta:[0,\infty) \to Y \text{ geodesic ray } \mid \delta(0) = y_0 \; \delta(K)\subset U\}/\sim$$ is open in $\partial Y$ by definition of the compact-open topology. By the previous paragraph, this set contains only $[\gamma] = \infty_Y$. Then $\{\infty_Y\}$ is open, meaning $\infty_Y$ is an isolated point in $\partial Y$.

\end{proof}

Next we show that there is a well-defined action of $\Isom(X\bowtie Y)$ on the heights. Recall that Corollary \ref{cor:Factors_height_change_wellDef} gives a well-defined height change $\Isom(X)_\infty \to \R$ on the factor spaces. For the moment, we restrict attention to $\Stab(\partial^u Y) \subset \Isom(X\bowtie Y)$, a subgroup with index at most two.

\begin{lemma}
    Let $f \in \Isom(X\bowtie Y)$ stabilize the upper boundary. Then the quantity $h(f(x,y)) - h(x,y)$ is independent of the choice of $(x,y)\in X\bowtie Y$.
\end{lemma}

\begin{proof}
    Let $(x,y),(x,y') \in X\bowtie Y$ have the same $X$-coordinate. Let $\gamma:\R \to X$ be a vertical geodesic containing $x$ and $\alpha_1, \alpha_2: \R \to Y$ be vertical geodesics containing $y$ and $y'$ respectively. Further assume that $\gamma,\alpha_1,\alpha_2$ are all parameterized by height. Let $V = \gamma \bowtie \alpha_1$ and $W = \gamma \bowtie \alpha_2$. Note that $$V(-\infty) = W(-\infty) = \gamma(-\infty) \in \partial_\ell X \subset \partial (X\bowtie Y).$$

    Let $f\in \Isom(X\bowtie Y)$ stabilize the upper boundary. Proposition \ref{prop:Opposite_Boundaries_Unique_Geodesic} and Lemma \ref{lem:preserves_partition} together imply that $f\circ V$ and $f\circ W$ are vertical geodesics. Let $V'$ and $W'$ be the height-parameterized geodesics with the same image as $f\circ V$ and $f\circ W$ respectively. Then there are constants $a,b \in \R$ such that $$f\circ V(t) = V'(t+a)\quad \text{and}\quad f\circ W(t) = W'(t+b).$$ These constants $a$ and $b$ are exactly the height-change constants $h(f(x,y)) - h(x,y)$ and $h(f(x,y')) - h(x,y')$. Note that $a$ and $b$ depend only on the geodesics $V$ and $W$, not on $(x,y)$ and $(x',y')$.

    Since $V$ and $W$ are asymptotic and height-parameterized, $$\lim_{t\to-\infty} d(V(t),W(t)) = 0.$$ By applying $f$, we have that $$\lim_{t\to -\infty}d(V'(t + a), W'(t+b)) = 0.$$ But since $V',W'$ are height parameterized, $d(V'(t + a), W'(t+b))$ will always be at least $|a-b|$. Then $a = b$, or equivalently $$h(f(x,y)) - h(x,y) = h(f(x,y')) - h(x,y'),$$ so point with the same $X$-coordinate have the same height change. The same argument shows that points with the same $Y$-coordinate have the same height change. If $h(x,y) = h(z,w)$, then $$h(f(x,y)) - h(x,y) = h(f(z,y)) - h(z,y) = h(f(z,w)) - h(z,w),$$ so points at the same height have the same height-change. If $h(x,y) \neq h(z,w)$, then we can use the previous observation that the height change is constant along vertical geodesics to conclude that height change is independent of the point chosen.    
\end{proof}

\subsection{Step 3: Complete the proof of \cref{thm:IsomGroupOfHorocyclicProduct}}

\begin{proof}[Proof of \cref{thm:IsomGroupOfHorocyclicProduct}]
    Let $f \in \allowbreak \Stab_{\Isom(X\bowtie Y)}\allowbreak (\partial^u Y)$. Our goal is to find isometries $f_X \in \Isom(X')_\infty$ and $f_Y \in \Isom(Y')_\infty$ so that $f(x,y) = (f_X(x), f_Y(y))$. 

    We will first find the map $f_X$. Let $p = (x,y)$ and $q = (x,y')$ have the same $X$-coordinate. Let $\gamma:\R \to X$ be a vertical geodesic containing $x$ and $\alpha_1, \alpha_2: \R \to Y$ be vertical geodesics containing $y$ and $y'$ respectively. Form the geodesics $V = \gamma \bowtie \alpha_1$ and $W = \gamma \bowtie \alpha_2$. Note that $V(-\infty) = W(-\infty) = \gamma(-\infty) \in \partial_\ell X$. 

    Since $f$ leaves the upper and lower boundaries invariant, Proposition \ref{prop:Opposite_Boundaries_Unique_Geodesic} implies that $f\circ V$ and $f\circ W$ are vertical---write $f\circ V = (V'_X,V'_Y)$ and $f\circ W = (W'_X,W'_Y)$ for vertical geodesics $V'_X,W'_X \subset X$, $V'_Y,W'_Y \subset Y$. Since $f\circ V$ and $f\circ W$ remain downward-asymptotic, $V_X'(-\infty) = W_X'(-\infty)$, and therefore $V_X' = W_X'$ by Lemma \ref{lemma:VerticalUniquenessInFactorSpaces}. Since $f$ is height-respecting, $h(f(p)) = h(f(q))$, and therefore $\pi_X(f(p))$ and $\pi_X(f(q))$ both lie on $V_X'$ at the same height. Points on vertical geodesics are uniquely specified by their height, so $\pi_X(f(p)) = \pi_X(f(q))$. That is, the $X$-component of $f(x,y)$ is independent of $y$. Repeating this argument for points with the same $Y$-component shows that $f(x,y) = (f_X(x), f_Y(y))$ for bijections $f_X, f_Y$ of $X$ and $Y$ respectively. 

    We will now show that $f_X$ and $f_Y$ are isometries of $X'$ and $Y'$. Let $\gamma_Y$ be any vertical geodesic in $Y$. Then 
    \[X' \cong X\bowtie \gamma_Y\subset X\bowtie Y\]
    is an isometrically embedded copy of $X'$. Composing with $f$ gives another isometrically embedded copy of $X'$. All elements of $X\bowtie \gamma_Y$ lie on a vertical geodesic of the form $\alpha \bowtie \gamma_Y$. All geodesics of this form are upward-asymptotic, so all $f\circ (\alpha\bowtie \gamma_Y)$ are also upward-asymptotic. Then there is some $\eta_Y$ so that for all vertical $\alpha \subset X$, we have $f\circ (\alpha\bowtie \gamma_Y) = \alpha'\bowtie \eta_Y$ for some $\alpha'$. In particular, $f(X\bowtie \gamma_Y) \subset X\bowtie \eta_Y$. 

    By construction, $f_X$ is exactly the composition $$X' \cong X \bowtie \gamma_Y \hookrightarrow X\bowtie Y \overset{f}{\to} X\bowtie Y \overset{\pi_X}{\to} X'.$$ The first map is an isometry by construction, and the final two maps are isometries when restricted to a subset of the form $X \bowtie \eta_Y$. Then $f_X$ is an isometry. Repeating this argument for $f_Y$ completes the proof for the case that $\Stab_{\Isom(X\bowtie Y)}(\partial^u Y)=\Isom(X\bowtie Y)$.

    Suppose now that $f$ sends some point $\xi_Y$ of $\partial^uY$ to a point $\xi_X$ of $\partial_\ell X$. By Lemma \ref{lem:preserves_partition}, $f(\partial^uY)=\partial_\ell X$ and vice versa, so that the two are at least homeomorphic. Moreover, $f$ sends vertical lines to vertical lines. Therefore, $f$ sends the collection of vertical lines incident to $\xi_Y$ to the collection of vertical lines incident to $\xi_X$. 
    
    Denote $V_Y$ (resp. $V_X$) the vertical line in $Y$ (resp. $X$) between $\xi_Y$ (resp. $\xi_X$) and $\infty_Y$ (resp. $\infty_X$). Then the set of vertical lines in $X\bowtie Y$ incident to $\xi_Y$ is exactly $X\bowtie V_Y$, while the set of vertical lines incident to $\xi_X$ is $V_X\bowtie Y$. Therefore, $f$ sends an isometrically embedded copy of $Y'$ to an isometrically embedded copy of $X'$. It follows that $X'$ and $Y'$ are isometric. If so, then it is clear that the map $f:X\bowtie Y\to X\bowtie Y$ sending $(x, y)$ to $(y,x)$ is an isometry and that $f\circ (g_1, g_2)\circ f=(g_2, g_1)$ for any pair of maps $g_1:X\to X$ and $g_2:Y\to Y$, so that $\langle f\rangle$ normalizes $\Stab_{\Isom(X\bowtie Y)}(\partial^u Y)$.

\end{proof}

\section{Proof of \cref{thm:Upgrade_to_Millefeuille}: upgrading the action to millefeuille spaces.}
\label{sec:Upgrade_to_Millefeuille}

In this section, we will prove \cref{thm:Upgrade_to_Millefeuille}, which says that a geometric action of a discrete group $G$ on a horocyclic product $X\bowtie Y$ can be upgraded to a geometric action on a horocyclic product $Z\bowtie W$ of two millefeuille spaces $Z, W$ by possibly passing to a finite-index subgroup of $G$. We first recall the definition of a millefeuille space.

\begin{Def}
    Let $X$ be a homogeneous, negatively curved Riemannian manifold\footnote{The definition in \cite{CCMT} allows $X$ to be a $\CAT(-\kappa)$ space, although we will only be interested in the case where $X$ is a Heintze group. In particular, Theorem \ref{thm:CCMT} below involves millefeuille spaces, as defined here.} and $T_k$ be the regular $(k+1)$-valent tree. Equip both $X$ and $T_k$ with a Busemann function 
    $\beta_X,\beta_{T_k}$. The millefeuille space $X[k]$ is $$X[k] = \{(x,y) \in X \times T_k \mid \beta_X(x) = \beta_{T_k}(y) \}.$$ 
\end{Def}

\cref{thm:Upgrade_to_Millefeuille} is proven in three broad steps:

\begin{enumerate}
    \item Use structure theory of amenable hyperbolic groups to find millefeuille spaces $Z,W$ quasi-isometric to $X,Y$ and actions of $G$ on $Z$ and $W$.
    \item Show that those two actions give an action on $Z\bowtie W$.
    \item Show that the action on $Z\bowtie W$ is geometric.
\end{enumerate}



\subsection{Step 1: Obtaining millefeuille spaces $Z,W$.}
\label{subsec:use_CCMT}

The millefeuille spaces appearing in \cref{thm:Upgrade_to_Millefeuille} are ultimately found by invoking the following theorem due to Caprace-Cornulier-Monod-Tessera \cite{CCMT}. This specific statement is taken from the survey \cite{Yves_Survey}.

\begin{thm}[\cite{CCMT}, \cite{Yves_Survey} Theorem 19.14]
\label{thm:CCMT}
    Let $G$ be a noncompact, compactly generated locally compact group. If $G$ is amenable and hyperbolic, then it admits a continuous, proper and cocompact action by isometries on a Millefeuille space $M[k]$ fixing a point on the boundary. 
\end{thm}

We first observe that this theorem can be applied to $\Isom(X')_\infty$ and $\Isom(Y')_\infty$ in our context. 

\begin{lemma}
\label{lem:CCMT_applies}
    Let $G$ act geometrically on $X\bowtie Y$. Then $\Isom(X')_\infty$ and $\Isom(Y')_\infty$ are noncompact, compactly generated, locally compact, amenable and hyperbolic.
\end{lemma}

\begin{proof}
    By \cref{thm:IsomGroupOfHorocyclicProduct}, there is an action of $G$ on $X'$ under which the projection $\pi_{X}: X\bowtie Y \to X'$ $G$-equivariant. Since $G$ acts cocompactly on $X\bowtie Y$, it also acts cocompactly on $X'$, since $\pi_{X}$ is distance-decreasing and $X'$ is proper. In particular, $\Isom(X')_\infty$, and also $\Isom(X')$ act cocompactly on $X'$. Since $X'$ is proper, this implies that the action of $\Isom(X')$ on $X'$ is geometric (see \cite{MGLCG}, proposition 5.B.10). Because $\Isom(X')_\infty \subset \Isom(X')$ acts cocompactly and is a closed subgroup of a group acting geometrically, the action of $\Isom(X')_\infty$ is also geometric. Moreover, the properness of $X'$ implies that $\Isom(X')$ is locally compact, hence also $\Isom(X')_\infty$. It now follows from the Milnor-Schwarz lemma for locally compact groups (See \cite{MGLCG}, Theorem 4.C.5) that $\Isom(X')_\infty$ is compactly generated and quasi-isometric to $X'$, hence hyperbolic.

    Because $X'$ is noncompact and the action of $\Isom(X')_\infty$ on $X'$ is geometric, we conclude that $\Isom(X')_\infty$ is also noncompact. That $\Isom(X')_\infty$ is amenable follows directly from \cite{Adams}, Theorem 6.8 and Lemma 3.3, stated below.
\end{proof}

\begin{thm}[Adams]
    Let $X$ be a proper, geodesic hyperbolic space having cocompact isometry group\footnote{Adams actually uses the condition that $X$ is ``at most exponential". This means that there is a constant $\alpha > 1$ such that, for every $x \in X$ and every $r > 0$, the ball $B_X(x,r)$ contains at most $a^r$ pairwise disjoint balls of radius 1. As observed in \cite{CCMT}, this is implied by $X$ having bounded geometry, which is itself implied by $X$ having a cocompact group of isometries. }. Then for every $\xi \in \partial X$, the stabilizer $\Isom(X)_\xi$ is amenable.
\end{thm}

Combining Theorem \ref{thm:CCMT} and Lemma \ref{lem:CCMT_applies} give Millefeuille spaces $Z, W$ and continuous, proper and cocompact actions $\Isom(X')_\infty \to \Isom(Z)$ and $\Isom(Y')_\infty \to \Isom(W)$.

\subsection{Step 2: Obtaining an action on $Z\bowtie W$.}
\label{subsec:upgrade_actions_to_horocyclicprod}

The previous step together with \cref{thm:IsomGroupOfHorocyclicProduct} gives actions of $G$ on millefeuille spaces $Z,W$, but the product action on $Z\times W$ may not preserve $Z\bowtie W$---perhaps the image of $G$ does not lie in $\Isom(Z)_\infty\times_{-h} \Isom(W)_\infty$. But this is almost true:  we will show that $Z$ and $W$ are coarsely $G$-equivariantly quasi-isometric to $X$ and $Y$ (respectively), and hence any $g \in G$ acts by the same type of isometry (i.e. hyperbolic or parabolic) on all of $X,Y,Z,W$. This will allow us to re-scale the metrics on $Z$ and $W$ to find an isometric action on $Z\bowtie W$.

We have continuous, proper and cocompact actions $\Isom(X')_\infty \to \Isom(Z)$ and $\Isom(Y')_\infty \to \Isom(W)$. These actions are therefore geometric, so the Milnor-Schwarz lemma says that the orbit maps 
\[X' \overset{\mathcal{O}_X}{\longleftarrow} \Isom(X')_\infty \overset{\mathcal{O}_Z}{\longrightarrow} Z \quad \text{and}\quad Y' \overset{\mathcal{O}_Y}{\longleftarrow} \Isom(Y')_\infty \overset{\mathcal{O}_W}{\longrightarrow} W\]
are quasi-isometries. We denote by $\mathcal{O}_{*}\inv$ their coarse inverses. Really, the Milnor-Schwarz lemma gives two different pseudo-metrics on $\Isom(X')_\infty$ which make the respective orbit maps quasi-isometries. But the identity map on $\Isom(X')_\infty$ is itself an equivariant quasi-isometry between these pseudo-metrics, and since we only use compositions of the above maps, we will suppress the metric on $\Isom(X')_\infty$ and $\Isom(Y')_\infty$. 

Next we will show that the maps 
\[{XZ} = \mathcal{O}_Z \circ \mathcal{O}_X\inv, \quad f_{YW} = \mathcal{O}_W \circ \mathcal{O}_Y\inv, \quad f_{ZX} = \mathcal{O}_X\circ\mathcal{O}_Z\inv \quad \text{and} \quad f_{WY} = \mathcal{O}_Y\circ\mathcal{O}_W\inv \]
are coarsely $G$-equivariant. This means the following: a function $f:A \to B$ between metric spaces $A,B$ with isometric $G$ actions is coarsely $G$-equivariant if there is some $D > 0$ so that for all $g \in G$ and $a \in A$, 
\[d_B(g\cdot f(a), f(g\cdot a)) \leq D.\]
Note that the orbit maps are (genuinely) $G$-equivariant, hence coarsely $G$-equivariant, so the following lemma suffices to prove the coarse equivariance of $f_{XZ},f_{ZX},f_{YW}$ and $f_{WY}$.

\begin{lemma}
    Let $A,B,C$ be (pseudo-)metric spaces admitting isometric actions of $G$. 
    \begin{enumerate}
        \item Let  $f:A\to B$ be a coarsely $G$-equivariant quasi-isometry, and let $f\inv$ denote a coarse inverse. Then $f\inv$ is also coarsely $G$-equivariant.
        \item If $h: B\to C$ is coarsely $G$-equivariant, then $h\circ f$ is also coarsely $G$-equivariant.
    \end{enumerate} 
\end{lemma}

\begin{proof}    
    Let $M$ be the maximum of all constants involved in this problem: the quasi-isometry constants for $f,f\inv$ and $h$, the coarse equivariance constants for $f$, and the constants $d_\infty(\id_A, f\inv \circ f), d_\infty(\id_B, f \circ f\inv)$. 
    
    Let $b \in B$ and $g \in G$. Then
    
    \begin{align*}
        d_B(b, f(f\inv(b))&\le M\\
        d_B\Big(g\cdot b, g\cdot f(f\inv(b))\Big)&\le M \\
        d_B\Big(g\cdot f(f\inv(b)), f(g\cdot f\inv(b))\bigr) &\le M\\
        d_B\Big(g\cdot b, f(g\cdot f\inv(b))\Big)&\le 2M
    \end{align*}
    
    Applying $f\inv$ to this, we obtain
    
    \begin{align*}
        d_B\Big(f\inv (g\cdot b),f\inv(f(g\cdot f\inv(b)))\Big) &\le 2M^2+M\\
        d_B\Big(f\inv(f(g\cdot f\inv(b))),g\cdot f\inv ( b)\Big) &\le M\\
        d_B\Big(f\inv (g\cdot b),g\cdot f\inv ( b)\Big) &\le 2M^2+2M,
    \end{align*}
    
    so $f\inv$ is coarsely $G$-equivariant. 

    Let $a \in A$ and $g \in G$. Applying $h$ to the upper bound $d_B\Big(f(g\cdot a) ,g\cdot f(a)\Big)\le M$ gives 
    
    \begin{align*}
        d_C\Big(h\circ f(g\cdot a),h(g\cdot f(a))\Big)&\le M^2+M\\
        d_C\Big(h(g\cdot f(a)),g\cdot h\circ f(a)\Big) &\le M\\
        d_C\Big(h\circ f(g\cdot a),g\cdot h\circ f(a)\Big) &\le M^2+2M,
    \end{align*}
    
    so $h\circ f$ is also coarsely $G$-equivariant.
\end{proof}

We are especially concerned with coarsely $G$-equivariant maps between Gromov-hyperbolic spaces. We now briefly remind the reader of some facts about isometries of Gromov-hyperbolic spaces.

There is a natural topology on $\overline{X} = X\cup \partial X$ so that any quasi-isometry $f:X \to X$ extends to a map $\overline{f}: \overline{X} \to \overline{X}$ whose restriction to the boudnary is the map $\partial f$ described previously. In this topology, a sequence $x_i \in X$ converges to a boundary point if and only if $(x_i, x_j)_p \to \infty$ as $i,j \to \infty$, where $p$ is any basepoint, and $(x,y)_p$ denotes the Gromov product 
\[(x,y)_p = \frac{1}{2}\left[ d(x,p) + d(y,p) - d(x,y) \right].\]
Two sequences $x_i$ and $y_i$ limit to the same point at infinity if and only if $(x_i, y_j)_p \to \infty$ as $i,j \to \infty$. See \cite{BridsonHaefliger}, Chapter III.H, section 3 for a complete treatment of $\overline{X}$.

If $G$ acts on a Gromov-hyperbolic space $X$, the \textit{limit set} is $\partial_XG = \overline{G\cdot x_0} \cap \partial X$. We will say that an isometry $g$ of $X$ is 

\begin{enumerate}
    \item elliptic if $\partial_X \langle g\rangle$ is empty, or equivalently if $g$ has bounded orbits.
    \item parabolic if $\partial_X \langle g\rangle$ is a singleton, or equivalently $g$ has unbounded orbits and 
    \[\lim_{n \to \infty} \frac{1}{n} d(x,g^n\cdot x) = 0.\]
    \item hyperbolic if $\partial_X \langle g\rangle$ consists of exactly two points, or equivalently if $\lim_{n \to \infty} \frac{1}{n} d(x,g^n\cdot x) > 0$.
\end{enumerate}

Gromov \cite{Gromov_Hyperbolic_Groups} proved that any isometry Gromov-hyperbolic space $X$ must be exactly one of the above types. We note that any isometry $g$ fixing a point at infinity with non-zero height change $h(g)$ has $\lim_{n \to \infty} \frac{1}{n} d(x,g^n\cdot x) \geq |h(g)| > 0$, and is therefore hyperbolic.

\begin{lemma}
\label{lem:equivariant_preserves_type}
    Let $X$ and $Z$ be Gromov-hyperbolic spaces admitting isometric $G$ actions. Let $f:X \to Z$ be a coarsely $G$-equivariant quasi-isometry. Then 
    \begin{enumerate}
        \item The boundary map $\partial f: \partial X \to \partial Z$ is a (genuinely) $G$-equivariant map.
        \item The action of $g \in G$ on $X$ is  hyperbolic (resp. parabolic, elliptic) if and only if its action on $Z$ is hyperbolic (resp. parabolic, elliptic).
        \item If $g$ acts hyperbolically on $X$ with attracting fixed point $\xi \in \partial X$, then $\partial f (\xi)$ is the attracting fixed point of $g$ acting on $Z$.
    \end{enumerate}
    
\end{lemma}

\begin{proof}
    Let $\gamma$ be a geodesic ray representing a point $\xi = [\gamma] \in \partial X$. Since $f$ is coarsely $G$-equivariant, $f(g\cdot \gamma)$ has finite Hausdorff distance from $g\cdot f(\gamma)$. Then 
    \[\partial f(g\cdot \xi) = [f(g\cdot \gamma) ] = [g\cdot f(\gamma)] = g\cdot \partial f(\xi),\]
    so $\partial f$ is (genuinely) $G$-equivariant.

    For part two, we will show that $\xi \in \partial X$ is a limit point of $\langle g \rangle$ if and only if $\partial f(\xi)$ is a limit point of $\langle g \rangle$. 
    
    Let $\xi$ be a limit point of $\langle g \rangle$, so that there is a sequence of orbit points $g^{k_i}\cdot x_0$ limiting to $\xi$. Then $f(g^{k_i}\cdot x_0)$ limits to $\partial f(\xi)$. We now show that the two sequences $f(g^{k_i}\cdot x_0)$ and $g^{k_i}f(x_0)$ limit to the same point at infinity. Indeed, coarse $G$-equivariance implies that
    \[2(f(g^{k_i} x_0), g^{k_j}f(x_0))_{f(x_0)} = d(f(g^{k_i} x_0), f(x_0)) + d(g^{k_j}f(x_0),f(x_0)) - d(f(g^{k_i} x_0), g^{k_j}f(x_0))\]
    has uniformly bounded difference from
    \[d(f(g^{k_i} x_0), f(x_0)) + d(f(g^{k_j}x_0),f(x_0)) - d(f(g^{k_i} x_0), f(g^{k_j} x_0)) = (f(g^{k_i}x_0), f(g^{k_j}x_0))_{f(x_0)},\]
    which converges to $\infty$ as $i,j\to \infty$, since $f(g^{k_i}x_0)$ limits to a point at infinity. 
    
    Applying this argument to $f\inv$ shows that $\xi$ is a limit point of $\langle g \rangle$ if and only if $\partial f(\xi)$ is also a limit point of $\langle g\rangle$. This completes the proof of part 2, since the type of an element is determined by the number of limit points of $\langle g \rangle$.

    For part three, note that by part one,  $g$ acts hyperbolically on both $X$ and $Z$ (part two), and that $\partial f$ conjugates the dynamical systems $(\partial X, g)$ and $(\partial Z, g)$, (part one) hence sends attracting points to attracting points.
\end{proof}

This has the following useful consequence: say $g\in G$ fixes $\xi \in \partial X$ and let $\beta_X$ and $\beta_Z$ be two Busemann functions based at $\xi$ and $\partial f (\xi)$ (respectively) with corresponding height functions $h_{X,\xi}:\Isom(X)_\xi\to \R$ and $h_{Z,\partial f(\xi)}:\Isom(Z)_{\partial f(\xi)}\to \R$ (which exist when $X,Z$ are vertically convergent, see Proposition \ref{prop:HeightChangeHom}). Then $h_{X,\xi}(g)$ and $h_{Z,\partial f(\xi)}(g)$ have the same sign. 


\begin{lemma}[Discrete height change]
    \label{lem:discrete_heights}
    Let $X,Y$ be proper, $\CAT(-\kappa)$ spaces and let 
    \[G \subset \Isom(X\bowtie Y) \subset \Isom(X')_\infty \times_{-h} \Isom(Y')_\infty \]
    be a discrete group acting geometrically. Then the height change homomorphism $h:G \to \R$ has discrete image.
\end{lemma}

\begin{proof}
    Assume for the sake of contradiction that there are elements $g_1,g_2,\ldots$ with $h(g_i)> 0$ strictly decreasing and $\lim_{i\to\infty } h(g_i) = 0$. Since $h(g_i)\neq 0$, $g_i$ has exactly two fixed points on $\partial (X\bowtie Y)$. Let $\Axis(g_i) = V_i$ be the unique vertical geodesic connecting these points. Since $g_i$ preserves the two endpoints, and $V_i$ is the unique geodesic connecting them, we have $g_i\cdot V_i = V_i$, and hence $h(g_i) = d(g_i\cdot x, x)$ for any $x \in V_i$.

    Let $K\subset X\bowtie Y$ be a compact set whose $G$-translates cover $X\bowtie Y$. For each $i \in \N$, choose some $k_i\in G$ so that $k_i\cdot V_i$ intersects $K$ non-trivially, and set $h_i = k_ig_ik_i\inv$. Note that $$\Axis(h_i) = \Axis(k_i g_i k_i\inv) = k_i\Axis(g_i) = k_i \cdot V_i$$ intersects $K$. Also, $h(h_i) = h(g_i)$. 

    Let $D$ be the diameter of $K$ and pick some $x_0 \in K$. The triangle inequality shows that 
    \[d(h_i\cdot x_0, x_0) \leq 2D + h(h_i) \leq 2D + h(g_i).\]
    Because $X\bowtie Y$ is proper, the evaluation map $\ev_{x_0}:\Isom(X\bowtie Y) \to X$ is proper, and hence $\ev_{x_0}\inv (B_{x_0}(2D + h(g_1)))$ is compact. Clearly $h_i \in \ev_{x_0}\inv (B_{x_0}(2D + h(g_1)))$. Since $G \subset \Isom(X\bowtie Y)$ is discrete, it has finite intersection with $\ev_{x_0}\inv (B_{x_0}(2D + h(g_1)))$. But the $h_i$ have distinct height changes, so there are infinitely many elements of $G$ in $\ev_{x_0}\inv (B_{x_0}(2D + h(g_1)))$. This is a contradiction.
\end{proof}

\begin{lemma}
\label{lem:obtain_action_on_Millefeuille}
    Let $G$ be a discrete group acting geometrically on $X\bowtie Y$ and let $Z,W$ be the Millefeuille spaces obtained from \cite{CCMT}, as above. After possibly re-scaling the metrics on $Z$ and $W$, the two actions $G \to \Isom(X')_\infty \to \Isom(Z)_\infty$ and $G \to \Isom(Y')_\infty \to \Isom(W)_\infty$ give a well-defined action of $G$ on $Z\bowtie W$.
\end{lemma}

\begin{proof}
    By Theorem $\ref{thm:IsomGroupOfHorocyclicProduct}$, it suffices to show that, after possibly re-scaling the metrics, the actions of $G$ on $Z$ and $W$ satisfy $h_Z(g) = -h_W(g)$ for all $g$. By Lemma \ref{lem:equivariant_preserves_type}, $$\ker(h_Z) = \ker(h_X) = \ker(h_Y) = \ker(h_W),$$ so both maps $h_Z,h_W:G\to \R$ factor through $G/\ker(h_Z)$, which, by \ref{lem:discrete_heights}, is isomorphic to $\Z$. Moreover, if $h_X(g) > 0$, then $h_Z(g) > 0$ also. Choose $g \in G$ with minimal positive height change $h_X(g)$. Then $h_Z(g) = \min\{h_Z(h) \mid h \in G, h_Z(h) > 0\}$ also. Likewise, $h_Y(g),$ $h_W(g)$ have the maximal negative height change among their respective actions.   
    Then the two maps $h_Z,h_W:G\to \R$ are determined by the image of $g$. Re-scale the metric on $Z$ so that $h_Z(g) = 1$, and likewise re-scale the metric on $W$ so that $h_W(g) = -1$. This is possible since $h_Z(g) > 0$ and $h_W(g) < 0$.

    We then have actions of $G$ on $Z$ and $W$ both fixing a point at infinity satisfying $h_Z(g) = -h_W(g)$ for all $g$. This gives an action on $Z\bowtie W$.
\end{proof}

Note that this proof in fact constructs an action of $G$ by elements of $\Isom(Z)\times_h \Isom(W)$, not just by $\Isom(Z')\times_h \Isom(W')$.

\subsection{Step 3: The action on $Z\bowtie W$ is geometric}

In this section, we finish the proof of \cref{thm:Upgrade_to_Millefeuille} by showing that the action on $Z\bowtie W$ is geometric.

It will be convenient to re-scale the metrics on $X$ and $Y$ so that $h_{X \bowtie Y}(G) = \Z \subset \R$, and therefore\footnote{Alternatively, one may re-scale the metrics on $Z$ and $W$ in the proof of Lemma \ref{lem:obtain_action_on_Millefeuille} so that $h_X(g) = h_Z(g)$, without assuming they are integers.} $h_X(g) = h_Z(g)$. The quasi-isometries $f_\ast$ are then \textit{coarsely height preserving}, in the sense given by the following lemma.


\begin{lemma}
    Let $G$ act cocompactly on proper $\CAT(-\kappa)$ spaces $X,Z$, fixing a point at infinity. Let $h_X,h_Z$ be two height functions based at this distinguished point at infinity, and assume $h_X(g) = h_Z(g)$. Let $f:X\to Z$ be a coarsely $G$-equivariant quasi-isometry. Then there is a $D>0$ so that for all $x \in X$, we have $$|h_Z(f(x)) - h_X(x)| < D.$$  
\end{lemma}

\begin{proof}
    Fix a basepoint $x_0 \in X$. Let $M$ be the maximum of the quasi-isometry constants, the coarse $G$-equivariance constants, and the coarse density constant of the orbits. Let $x \in X$. Then there is a $g \in G$ with $d(g\cdot x_0, x) < M$. Note that $|h_X(g\cdot x_0) - h_X(x)| < M$. Applying $f$, we obtain $$d(f(g\cdot x_0), f(x)) < M^2 + M \quad \text{and by equivariance,} \quad d(g\cdot f(x_0),f(x)) < M^2 + 2M,$$ so 

    \begin{align*}
        M^2 + 2M &> |h_Z(f(x) - h_Z(g\cdot f(x_0))| \\
        &= |h_Z(f(x)) - (h_Z(f(x_0)) + h(g))|\\
        &= |h_Z(f(x)) - h_Z(f(x_0)) + h_X(x_0) - h_X(g\cdot x_0)|.
    \end{align*}

Then $$|h_Z(f(x)) - h_X(g\cdot x_0)| < M^2 + 2M + |h_X(x_0) - h_Z(f(x_0))|.$$ Finally, apply $|h_X(g\cdot x_0) - h_X(x)| < M$ to obtain $$|h_Z(f(x)) - h_X(x)| < M^2 + 3M + |h_X(x_0) - h_Z(f(x_0))|.$$

\end{proof}

We are now prepared to prove \cref{thm:Upgrade_to_Millefeuille}.

\begin{proof}[Proof of \cref{thm:Upgrade_to_Millefeuille}]
    Lemma \ref{lem:obtain_action_on_Millefeuille} gives (possibly degenerate) Millefeuille spaces $Z,W$ with $G$-actions giving a $G$-action on $Z\bowtie W$. The preceding discussion also found coarsely $G$-equivariant, coarsely height preserving quasi-isometries $f_{XZ}:X' \to Z$ and $f_{YW}:Y'\to W$ with coarse inverses $f_{ZX}$ and $f_{WY}$ respectively. It remains to show that the action on $Z\bowtie W$ is geometric---that is, properly discontinuous and cocompact.

    We first show that the action has a coarsely dense orbit. Since $Z\bowtie W$ is proper, this is equivalent to the action being cocompact. The goal is to find a point $(z_0, w_0)$ in $Z\bowtie W$ so that, for each $(z, w)$ in $Z\bowtie W$, there is an element $g\in G$ so that $d_{\bowtie}\Big(g\cdot (z_0, w_0), (z,w)\Big)$ is uniformly bounded. In light of Lemma \ref{lemma:BoundedCoordinateDinstancesImpliesBoundedDistance}, it suffices to find uniform bounds on $d_Z(z, g\cdot z_0)$ and $d_W(w, g\cdot w_0)$ for the same element of $G$. 

    Fix a point $(x_0, y_0)$ in $X\bowtie Y$. Choose $M$ the maximum of the following: the coarse height preservation constants, coarse equivariance constants, and quasi-isometry constants of the maps $f_{XZ}, f_{ZX}, f_{WY}, f_{YW}$, the distance to the identity of the compositions $f_{ZX}\circ f_{XZ}$, $f_{XZ}\circ f_{ZX}$, $f_{WY}\circ f_{YW}$, and $f_{YW}\circ f_{WY}$, and the coarse density of the orbit of $(x_0, y_0)$ in $X\bowtie Y$ under the group $\Stab_{\Isom(X\bowtie Y)}(\partial^u(Y))$. 
    
    Since $f_{XZ}$ and $f_{YW}$ are only $M$-coarsely height-preserving, the point $\Big(f_{XZ}(x_0), f_{YW}(y_0)\Big) \in Z \times W$ may not lie on $Z\bowtie W$. However, we can choose $z_0$ and $w_0$ within $M$ of $f_{XZ}(x_0)$ and $f_{YW}(y_0)$ so that $(z_0, w_0)$ is on $Z\bowtie W$. We wish to show that the orbit of $(z_0, w_0)$ is coarsely dense.

    As with $\Big(f_{XZ}(x_0), f_{YW}(y_0)\Big)$, the point $(f_{ZX}(z), f_{WY}(w)) \in X\times Y$ may not lie on $X\bowtie Y$, but we can find a point $(x, y)$ in $X\bowtie Y$ so that $d'_X(f_{ZX}(z), x)\le M$ and $d'_Y(f_{WY}(w),y)\le M$. Choose $g\in \Stab_{\Isom(X\bowtie Y)}(\partial^u(Y))$ so that $d_{X\bowtie Y}\Big(g\cdot(x_0, y_0), (x,y)\Big)\le M$. Since $g$ acts factorwise, and the metric is admissible, it follows that $d'_X(g\cdot x_0, x), d'_Y(g\cdot y_0, y) \le 2M$. Consider then the point $g\cdot (z_0, w_0)$.  

    To bound $d_Z(g\cdot z_0, z)$, observe that
    
    \[d'_X\Big(x, f_{ZX}(z)\Big)\le M\]

    and
    
    \[d'_X\Big(x, g\cdot x_0\Big)\le M\]
    
    so that

    \[d'_X\Big(g\cdot x_0, f_{ZX}(z)\Big)\le 2M.\]

    Applying $f_{XZ}$ to both sides shows that
    
    \[d_Z\Big(f_{XZ}(g\cdot x_0), f_{XZ}\circ f_{ZX}(z)\Big)\le 2M^2+M.\]

    But of course, 

    \[d_Z\Big(f_{XZ}\circ f_{ZX}(z), z\Big)\le M,\]

    so that 

    \[d_Z\Big(f_{XZ}(g\cdot x_0), z\Big)\le 2M^2+2M.\]

    Now, by the coarse equivariance of $f_{XZ}$, we see that 

    \[d_Z\Big(g\cdot f_{XZ}(x_0), z\Big)\le 2M^2+3M.\]

    However, $z_0$ is within $M$ of $f_{XZ}(x_0)$, and $g$ acts isometrically on $Z$. Therefore, 

    \[d_Z\Big(g\cdot z_0, z\Big)\le 2M^2+4M.\]

    Analogous logic applies for $Y$ and $W$. This shows that the orbit of $(z_0, w_0)$ is coarsely dense.

    We will now show that the action is properly discontinuous. Let $B(R)$ be a closed ball of radius $R$ in $Z\bowtie W$. We will show that there are only finitely many $g \in G$ satisfying $B(R) \cap g\cdot B(R) \neq \emptyset$. 
    
    There exist compact $C_X \subset X$ and $C_Y \subset Y$ such that $f_{ZX}(\pi_Z \circ B(R)) \subset C_X$ and $f_{WY}(\pi_W \circ B(R))\subset C_Y$. Consider the set $K=\Nbhd_{2M}(C_X)\bowtie \Nbhd_{2M}(C_Y)$ (where neighborhoods are taken with respect to the metrics $d'_X$ and $d'_Y$). $K$ is evidently compact, so there are only finitely many elements $g$ of $G$ such that $gK\cap K\ne\emptyset$. We will show that $gB(R)\cap B(R)\ne\emptyset$ implies that $gK\cap K\ne\emptyset$, so that only finitely many elements of $G$ fail to move $B(R)$ off of itself.

    Let $(z,w)\in B(R)$ satisfy $g\cdot (z,w) \in B(R)$. Then $(f_{ZX}(z), f_{WY}(w))$ and $(f_{ZX}(g\cdot z), f_{WY}(g\cdot w))$ are both in $C_X \times C_Y$ by construction.

    By $M$-coarse height preservation, we can find $(x,y)$ in $\Nbhd_{M}(C_X)\bowtie \Nbhd_{M}(C_Y)\subset K$ so that $d'_X(f_{ZX}(z), x), d'_Y(f_{WY}(w), y)\le M$. We wish to show that $g\cdot (x,y)\subset K$. To do this, note that 

    \[d'_X(g\cdot f_{ZX}(z), g\cdot x), d'_Y(g\cdot f_{WY}(w), g\cdot y)\le M\]

    because $g$ acts isometrically on $d'_X$ and $d'_Y$. Using coarse equivariance, we see that

    \[ d'_X( f_{ZX}(g\cdot z), g\cdot x), d'_Y(f_{WY}(g \cdot w), g\cdot y)\le 2M.\]

    But $f_{ZX}(g\cdot z)\in C_X$ and $f_{WY}(g \cdot w)\in C_Y$, so that $g\cdot(x,y)$ is in $K$ as desired.
\end{proof}

\section{Proof of \cref{thm:boundary_implies_algebra}: Connectedness of the boundaries detects algebra}
\label{sec:boundary_implies_algebra}

In this section, we prove \cref{thm:boundary_implies_algebra}, which recovers algebraic properties of groups acting on horocyclic products from the connectedness of the upper and lower boundaries. 

The starting point for the proof is the conclusion of \cref{thm:Upgrade_to_Millefeuille}, which allows us to upgrade the action of $G$ on $X\bowtie Y$ to an action of a finite-index subgroup of $G$ on $Z\bowtie W$, for $Z$ and $W$ millefeuille spaces quasi-isometric to $X$ and $Y$ respectively. Since both trees and non-degenerate millefeuille spaces give rise to disconnected upper and lower boundaries, $\partial_\ell X$ (resp. $\partial^u Y$) is connected if and only if $Z$ (resp. $W$) is a Heintze group. The hypotheses of parts 1, 2 and 3 thus specify how many of $Z,W$ are Heintze groups. Part 2 is the most complicated of the cases and contains all the essential arguments for part 3.

\subsection{Proof of \cref{thm:boundary_implies_algebra}, Part 1: both boundaries disconnected.}

In this subsection, we will prove Part one of \cref{thm:boundary_implies_algebra}, which says that if $G$ acts geometrically on a horocyclic product $X\bowtie Y$ with both upper and lower boundaries disconnected, then $G$ is not finitely presented. Really we prove the contrapositive: if $G$ is finitely presented, then at least one boundary is connected.

The basic strategy is to apply the Bieri-Strebel splitting theorem to the map $G\to \Z$ to conclude that $G$ is an ascending HNN extension over a finitely-generated group. Britton's lemma then implies that every $g \in G$ can be written as $g = t^{-x} h t^{y}$ for $x,y \geq 0$. Geometrically, this means any point in $Z\bowtie W$ on the $0$-horocycle can be connected to the basepoint by a (coarse) path avoiding the set $\{h >0\}$ or $\{h < 0\}$ (depending on whether $h(t)$ is positive or negative). If both $X$ and $Y$ have disconnected boundaries, then $Z$, $W$ are either trees or non-degenerate millefeuille spaces. This is a contradiction in either case, since there are points on the $0$-horocycle which cannot be connected by paths avoiding the $0$-horoball. 

The first part of this proof is a direct application of the Bieri-Strebel splitting theorem:

\begin{lemma}
\label{lem:HNN_from_finitely_presented_action}
    Let $G$ be a finitely presented group acting geometrically on $X\bowtie Y$. Then there is a finitely-generated group $H$ and an injection $f:H\to H$ such that $$G \cong \HNN(H,f) = \langle H, t \mid tst\inv = f(s) \rangle$$ and the homomorphism $G \to \Z$ given by the exponent sum of $t$ is either the height change homomorphism $h:G \to \Z$ or $-h$. 
    
\end{lemma}

\begin{proof}

By Lemma \ref{lem:discrete_heights}, the height change homomorphism $h: G \to \R$ has discrete image, hence giving a homomorphism $G \to \Z$. By the Bieri-Strebel splitting theorem \cite{Bieri-Strebel}, there is a subgroup $H \subset \ker(h)$, an element $t \in h\inv(1)$, subgroups $L,K \subset H$ and a map $\phi: L\to K$ such that $$G \cong \HNN(H,L,K,\phi) = \langle H, t \mid t\ell t\inv = \phi(\ell), \; \ell \in L \rangle.$$ An HNN extension as above is said to be \textit{ascending} if either $H = L$ or $H = K$. If $G$ is non-ascending, the action on the Bass-Serre tree $T$ does not fix an end (see Lemma \ref{lem:get_HNN}) and therefore contains two hyperbolic elements with different endpoints. The ping-pong lemma on $\partial T$ shows that sufficiently high powers of these hyperbolic elements must generate a free group. But $G$ is a finite extension of a discrete subgroup of the amenable group $\Isom(X\bowtie Y)$. It is therefore amenable and cannot contain a free group. We conclude that the above HNN extension is ascending, and after possibly exchanging $K$ and $L$ while inverting $\phi$, $G$ takes the desired form.
\end{proof}

By possibly reversing the roles of $X$ and $Y$, we may assume that $h(t) > 0$. The next step of the argument uses this ascending HNN extension to prove the following: for each $p$ in the $0$-horosphere of $X$, there is a (coarse) path connecting $p$ to the basepoint $x_0 \in X$ which avoids the $0$-horoball. This is a geometric incarnation of Britton's lemma \cite{Britton}, stated below.

\begin{lemma}[Britton, HNN case]
\label{lem:Britton}
    Let $G \cong \langle H,t \mid tht\inv = f(h) \rangle$ be an ascending HNN extension. Then every $g \in G$ can be written $g = t^{-x} h t^{y}$ for $h \in H$ and $x,y \geq 0$. 
\end{lemma}

\begin{lemma}
\label{lem:Connected_outside_horoball}
    Let $G = \HNN(H,f)$ act geometrically on $X\bowtie Y$ with $h(t) > 0$ and $H \subset \ker(h)$ finitely generated. Let $D>0$, and fix a point $x_0\in X$ with $h_X(x_0)=0$. Then there is a constant $C>0$ so that for all $x_1 \in X$ in the $D$-neighborhood of the $0$-horocycle, there is a sequence of points $x_1 = p_0,p_1,\ldots, p_n=x_0$ with $d_X(p_i,p_{i+1}) < C$ and $h(p_i) \leq 0$.  
\end{lemma}

\begin{proof}

    By possibly re-scaling the metric on $X\bowtie Y$, we may assume that $h(G) = \Z \subset \R$ (not merely isomorphic to $\Z$). Let $S$ be a finite generating set for $H$. Then $\langle S \cup \{t\}\rangle  = G$. Choose a basepoint $(x_0, y_0) \in X\bowtie Y$. Since the action is geometric, the orbit map $\mathcal{O}:\Cay(G,S\cup \{t\}) \to X\bowtie Y$ given by $g \to g \cdot (x_0, y_0)$ is a quasi-isometry. It is moreover height-preserving in the sense that $h_{\Isom(X\bowtie Y)}(g) = h_{X\bowtie Y}(\mathcal{O}(g))$. 

    For any $g \in \ker(h)$, use Lemma \ref{lem:Britton} to write $g = t^{-a}kt^{b}$ for $x=y \geq 0$ and $k \in H$. Write $k = s_1\cdots s_k$ a product of elements of $S$. Then the sequence of group elements 
    \[t\inv, t^{-2},\ldots, t^{-x}, t^{-x}s_1,\ldots t^{-x}s_1\cdots s_k, t^{-x}s_1\cdots s_k t,\ldots t^{-x}s_1\cdots s_kt^x\]
    forms a coarse $1$-path in the Cayley graph. Applying $\mathcal{O}$ to this path gives a $C'$-coarse path connecting $(x_0,y_0)$ to $\mathcal{O}(g)$, for $C'$ depending only on the quasi-isometry constants. Note that every element of this coarse path has negative height. 

    The lemma now follows by noting that $\mathcal{O}(\ker(h))$ is coarsely dense in the $0$-horosphere of $X\bowtie Y$, (hence also in its $D$-neighborhood) and that the projection $\pi_X:X\bowtie Y \to X$ is $2$-Lipschitz (because $N\ge \frac{L^1}{2})$ and height preserving.
\end{proof}

We are now prepared to prove Part 3 of \cref{thm:boundary_implies_algebra}

\begin{proof}[Proof of Part 3 of \cref{thm:boundary_implies_algebra}]

Let $G$ act geometrically on $X\bowtie Y$. From Lemma \ref{lem:HNN_from_finitely_presented_action}, $G$ decomposes as an ascending HNN extension $G = \HNN(H,f)$ for $H \subset \ker(h)$ finitely generated. By possibly swapping $X$ and $Y$, we may assume that the stable letter $t$ has positive height change $h(t)>0$. Sections \ref{subsec:use_CCMT} and \ref{subsec:upgrade_actions_to_horocyclicprod} give a millefeuille space $Z$ and a coarsely equivariant and coarsely height-preserving quasi-isometries $f_{ZX}:Z \to X$, $f_{XZ}:X\to Z$. 

Since both\footnote{Note that we needed both boundaries to be disconnected, since we do not know if their roles have been swapped in the previous sentence.} the upper and lower boundaries of $X\bowtie Y$ are disconnected, we conclude that $Z$ is either a tree or a non-degenerate millefeuille space. In either case, for all $C> 0$, there are points on the $0$-horocycle of $Z$ which cannot be connected to the basepoint $z_0$ by a $C$-coarse path in the $0$-horoball. We will now show that this contradicts Lemma \ref{lem:Connected_outside_horoball}.

Let $D$ be the coarse height preservation constant for $f_{ZX}$. Then for any $z \in Z$ with $h(z) = 0$, its image $f_{ZX}(z)$ has $|h(f_{ZX}(z))|\leq D$ and hence is in the $D$ neighborhood of the $0$-horosphere for $X$. Lemma \ref{lem:Connected_outside_horoball} gives a constant $C$ so that any such point can be connected to the basepoint with a $C$-coarse path. The image of such a $C$-coarse path under $f_{XZ}$ is a $E$-coarse path connecting $z_0$ to $z$, with $E$ depending only on $C$ and the quasi-isometry constants of $f_{XZ}$. If $D'$ is the coarsely height preserving constant for $f_{XZ}$, then all elements of this path have height at most $D'$.

Since $X$ is either a tree or a non-degenerate millefeuille space, there is a point $z \in Z$ which cannot be connected to the basepoint with an $E$-coarse path avoiding the $D'$-horoball. This contradicts the previous paragraph.   
\end{proof}

\subsection{Proof of \cref{thm:boundary_implies_algebra}, Part 2: $W$ has a tree factor}

When only one of $\partial_\ell X, \partial^u Y$ is connected, $X'$ and $Y'$ are not isometric, so \cref{thm:IsomGroupOfHorocyclicProduct} gives
\[\Isom(X\bowtie Y) \subset \Isom(X')_\infty \times_{-h} \Isom(Y').\]
In particular, the finite-index subgroup in \cref{thm:Upgrade_to_Millefeuille} may be taken to be the entire group $G$, that is $G$ acts on $Z\bowtie W$, not just a finite-index subgroup. Since exactly one boundary is connected, exactly one of $Z,W$ is a Heintze group, which we assume without loss of generality to be $Z$. We then have a geometric action of a discrete group $G$ on $Z\bowtie W$, where $Z$ is a Heintze group and $W$ is either a tree or a non-degenerate millefeuille space. By the remark following Lemma \ref{lem:obtain_action_on_Millefeuille}, $G$ actually acts by $\Isom(Z)\times_{-h} \Isom(W)$, not just $\Isom(Z')\times_{-h} \Isom(W')$. Our first task is to compute $\Isom(W)$.

\begin{lemma}[Isometry group of a millefeuille space]
\label{lem:Isom_of_millefeuille}
    Let $M = N\rtimes \R$ be a nontrivial Heintze group and let $k \geq 2$. Then
    \[\Isom(M[k]) = \Isom(M)_\infty \times_h \Aut(T)_\infty.\]
\end{lemma}

\begin{proof}
    Let $f$ be an isometry of $M[k]$. The boundary $\partial M[k]$ has exactly one cut point (the distinguished point at infinity), which is therefore preserved by any isometry. There is a projection $\pi:M[k] \to T$. We will call the preimages of vertices under $\pi$ \textit{branching horocycles}. Isometries must preserve branching horospheres, and since vertices are adjacent if and only if their branching horocycles have hausdorff distance 1, $f$ induces a tree automorphism, $f_T \in \Aut(T_k)_\infty$. This gives the first projection $\Isom(M[k]) \to \Aut(T)_\infty$. From the construction, if $s:T\to M[k]$ is any section of $\pi$, then $f_T = \pi \circ f \circ s$.

    For the other projection, note that every vertical line $\ell \subset T_k$, the preimage $\pi\inv(\ell)$ is an isometric copy of $M$. From the above paragraph, $f(\pi\inv(\ell)) = \pi\inv(f_T(\ell))$. The projection $p:M[k] \to M$ is isometric when restricted to preimages of vertical lines in $T$. Define $f_M\in \Isom(M)_\infty$ as $p\circ f|_{\pi\inv(\ell)}$. This is independent of choice of $\ell$, since the preimages of any two vertical lines agree above some horoball in $M$, and any two isometries of $M$ agreeing on an open set are equal.
    
    The two maps $f_T$ and $f_M$ satisfy $h_T(f_T) = h_M(f_M)$, so $f \mapsto (f_M, f_T)$, is a homomorphism into $\Isom(M)_\infty \times_h\Aut(T)_\infty$. It is clearly surjective, and injective since any two isometries $f,g \in \Isom(M[k])$ with $f_T = g_T$ and $f_M = g_M$ must both send an arbitrary point $(m,v) \in M[k]$ to $(f_M(m), f_T(v)) = (g_M(m), g_T(v))$, so $f = g$.
\end{proof}

We then have a decomposition $$\Isom(Z\bowtie W) \cong \Isom(Z)_\infty \times_{-h} \left( \Isom(M)_\infty \times_h \Aut(T_k)_\infty  \right).$$ In particular, $G$ acts on $T_k$, fixing an end. This is the source of the HNN decomposition.

\begin{lemma}
\label{lem:get_HNN}
    Let $T$ be a regular tree with valence $k\geq 3$. Let $G$ act cocompactly on $T$, fixing an end. Then $T/G$ is a loop, and $G$ splits as an ascending HNN extension.
\end{lemma}

\begin{proof}

    Since $G$ fixes an end, there is an associated height change homomorphism $h: G \to \Z$. Let $t$ be an element with the largest negative height change. By possibly modifying the tree to include only vertices at height $h(t)\Z$, we can obtain a cocompact action of $G$ on $T_{k^{-h(t)}}$. This lets us assume that $h(t) = -1$. Let $v_0 \in \Axis(t)$ be at height $0$ and set $v_k = t^k v_0$. Also set $H = \Stab(v_0)$. For $k\in \N$ and vertex $w \in T$, let $$\up_k(w) = \{z \in T \mid \beta(z) = \beta(w)+k\text{ and }d(w,z) = k\}$$ and $\up(w) = \cup_{k = 0}^\infty \up_k(w)$. This notation is taken from \cite{Caplinger}, which uses the opposite height convention for trees. 

    First, we show that $H$ acts transitively on $\up_1(v_0)$. If not, there is a vertex $w \in \up_1(v_0)$ so that no $h\in H$ satisfies $h\cdot v_1 = w$. But then no hyperbolic element $g \in G$ can have axis intersecting $\up(w)$---if it did, then either $gt^{h(g)}$ or $g\inv t^{-h(g)}$ would be an element of $H$ sending $v_1$ to $w$. But then no two elements of $\up(w)$ at different heights can ever be identified in the quotient $T/G$. This contradicts the cocompactness of the action, proving the claim.

    Conjugating the previous paragraph by $t$ shows that $t^k H t^{-k} \subset H$ acts transitively on $\up_1(v_k)$. Then $H$ also acts transitively on $\up_k(v_0)$ for all $k \in \N$. 

    We now show that the quotient $T/G$ is a loop. Let $e$ be an arbitrary edge in $T$. There is some sufficiently large $k$ so that $t^k (e) \in \up(v_0)$. From the previous paragraph, there is some $h \in H$ so that $h t^k (e) \subset \Axis(t)$. Finally, there is some $m$ so that $t^{-m} h t^{k} (e)$ is the edge connecting $v_0$ and $v_1$. Then all edges in $T$ are identified, and therefore $T/G$ is a loop.

    The structure theorem of Bass-Serre theory (See \cite{Trees}, Theorem 13) now gives that $G = \HNN(H,f)$, where $H = \Stab_G(v_0)$, and $f$ is the inclusion $H \to H$ given by $h \to tht\inv$. Note that the index $[H:f(H)] = k-1$, where $k$ is the valence of $T$.
\end{proof}

To finish the proof of the theorem, we need to show that this vertex stabilizer $H$ is virtually nilpotent. From the explicit form of $\Isom(Z\bowtie W)$, we observe that $\Stab_{\Isom(Z\bowtie W)}(v_0)$ acts on $\pi\inv(v_0) \subset Z\bowtie W$, where $\pi$ is the projection $Z\bowtie W \to T_k$. Since $Z = N_1 \rtimes \R$ is a Heintze group and $W$ is either a tree $T_k$ or a nondegenerate millefeuille space $(N_2 \rtimes \R)[k]$, this preimage $\pi\inv(v_0)$ is either $N_1$ or $N_1 \times N_2$. We consider the induced path metrics on $\pi\inv(v_0)$. When $\pi\inv(v_0) = N_1$, this metric is simply the left invariant Riemannian metric used in $Z = N_1 \rtimes \R$. When $\pi\inv(v_0) = N_1 \times N_2$, the induced path metric is not Riemannian---it is instead the path metric defined by the ambient norm $N$, or in other words $$d_{N_1\times N_2}( (n,m), (\ell, p))=  \min_{\gamma = (\gamma_1, \gamma_2)} \sup_{t_0 < t_1 <\cdots t_n}  \sum_{i = 1}^{n-1} N( d_{N_1}(\gamma_1(t),\gamma_1(t+1)), d_{N_2}\gamma_2(t),\gamma_2(t+1) ),$$ where $d_{N_i}$ are the two left-invariant Riemannian norms and the minimum is taken over all paths connecting $(n,m)$ to $(\ell,p)$. In particular, this metric is also left-invariant.  

Then the vertex stabilizer $H$ acts on a nilpotent Lie group with a left-invariant (possibly non-Riemannian) metric. We will show that this action is in fact geometric, and then invoke Auslander's generalization of Bieberbach's theorem to conclude that $H$ is virtually a lattice in a nilpotent lie group.

\begin{lemma}
\label{lem:Geometric_on_horosphere}
    Let $Z = N_1\rtimes \R$ be a Heintze group, $W = (N_2\rtimes \R)[k]$ be a (possibly degenerate) millefeuille space, and let $G$ act geometrically on $Z\bowtie W$. Let $T$ be the tree of $W$ and $v_0 \in T$. Then $H = \Stab_G(v_0)$ acts geometrically on $\pi\inv(v_0)$ with the induced path metric.
\end{lemma}

\begin{proof}

    Since $\pi\inv(v_0)$ is closed, if $K \subset \pi\inv(v_0)$ is compact in $\pi\inv(v_0)$ then it is also compact in $Z\bowtie W$. Since the action of $G$ on $Z\bowtie W$ is geometric, there are only finitely many elements $g \in G$ such that $gK\cap K \neq \emptyset$, hence only finitely many $g \in H$ with this property, so the action of $G$ on $\pi\inv(v_0)$ is properly discontinuous. It remains to show that the action of $H$ on $\pi\inv(v_0)$ is cocompact.

    Let $t \in G$ have minimal positive height change with $T$-axis intersecting $v_0$. This is possible since $T/G$ is a loop (Lemma \ref{lem:get_HNN}), which implies that $\ker(h)$ acts transitively on each horosphere. Since $t$ acts hyperbolically on each factor, it fixes two points on the boundary of $Z\bowtie Y$. Let $\Axis(t)$ denote the vertical line connecting these two points and let $x_0  = \pi\inv(v_0) \cap \Axis(t)$. Because the action on $Z\bowtie W$ is geometric, there is some $M$ so that $\Nbhd_M(G\cdot x_0) = Z\bowtie W$. 
    
    Let $x \in \pi\inv(v_0)$, and let $y \in Z\bowtie W$ be any point with $h(y) - h(x) = 2M = d_{Z\bowtie W}(x,y)$ (that is, $y$ is directly above $x$ at distance $2M$). Then there is some $g \in G$ with $d_{Z\bowtie W}(g\cdot x_0,y) < M$. Then $d_{Z\bowtie W}(g\cdot x_0, x) < 3M$. Also, since $\pi$ is $2$-Lipschitz, we know that $d_T(\pi(g\cdot x_0), \pi(y))<2M$ while $d_T(\pi(y), v_0)=2M=h_T(\pi_T(y))-h_T(v_0)$. Since the distance from $\pi(y)$ to any point outside $\up(v_0)$ is at least $2M$, we see that $\pi(g\cdot x_0) \in \up(v_0)$. In particular, any vertical geodesic containing $g\cdot x_0$ intersects $\pi\inv(v_0)$ in a single point. Since $x_0 \in \Axis(t)$, this point must be $gt^{-h(g)}\cdot x_0$, and the triangle inequality gives $d_{Z\bowtie W}(gt^{-h(g)} \cdot x_0, x) \leq 6M$. Note that $gt^{-h(g)} \in H$.

    Though we have found an orbit point of $H$ at bounded distance from $x$, we are not done, since this bound is only obtained in the horocyclic product metric---we need to relate this to the path metric on the horosphere component $\pi\inv(v_0)$. Let $z \in \pi\inv(v_0)$ be arbitrary. The $6M$ ball $B_{Z\bowtie W}(z,6M)$ based at $z$ is compact, so the intersection $B_{Z\bowtie W}(z, 6M) \cap \pi\inv(v_0)$ is also compact and therefore bounded in the path metric on $\pi\inv(v_0)$---that is, there is some $R >0$ so that $$B_{Z\bowtie W}(z,6M) \cap \pi\inv(v_0)\subset B_{\pi\inv(v_0)}(z,R).$$ Since $N_1 \times N_2 \subset \Stab_{\Isom(Z\bowtie W)}(v_0)$ acts transitively on $\pi\inv(v_0)$, the choice of $R$ can be made independently of $z$. In particular, this means that $x \in B_{\pi\inv(v_0)}(gt^{-h(g)}\cdot x_0, R)$, so the action of $H$ on $\pi\inv(v_0)$ is cobounded, hence cocompact. 
\end{proof}

We now have a geometric action of $H = \Stab(v_0)$ on a simply connected nilpotent Lie group $\pi\inv(v_0) = N_1\times N_2$. To finish the proof of \cref{thm:boundary_implies_algebra}, Part 2, we need to conclude that $H$ is virtually nilpotent. To do this, we will invoke the following two theorems.

\begin{thm}[Kivioja-Le Donne \cite{Isom_Nilpotent}]
    Let $N$ be a nilpotent lie group equipped with a left-invariant metric that induces the manifold topology. Then $$\Isom(N) \cong N \rtimes \Aut \Isom(N),$$ where $\Aut \Isom(N)$ denotes the group of automorphisms of $N$ that are also isometries.
\end{thm}

We will of course apply the above theorem to the nilpotent Lie group $\pi\inv(v_0) = N_1\times N_2$. Note that the Arzela-Ascoli theorem implies that $\Aut \Isom(\pi\inv(v_0))$ is compact.

\begin{thm}[Auslander \cite{Auslander}]
    Let $N$ be a connected, simply connected nilpotent Lie group and let $C$ be any compact Lie group of automorphisms of $N$. Let $\Gamma \subset N\rtimes C$ be a discrete cocompact subgroup. Then there is a finite group $F$ and a short exact sequence $$1 \to \Gamma \cap N \to \Gamma \to F \to 1.$$  
\end{thm}

We are now prepared to prove \cref{thm:boundary_implies_algebra}, Part 2.

\begin{proof}[Proof of \cref{thm:boundary_implies_algebra}, Part 2]

    Let $G$ act geometrically on a horocyclic product $X\bowtie Y$. By \cref{thm:Upgrade_to_Millefeuille}, we can upgrade this action to an action on $Z\bowtie W$, where $Z,W$ are millefeuille spaces. Since exactly one of the boundaries $\partial^\ell X$ $\partial^u Y$ are connected, exactly one of $Z,W$ is a Heintze group; without loss of generality, say it is $Z$. Moreover, $W$ is a millefeuille space with non-trivial tree factor, meaning (by combining \cref{thm:IsomGroupOfHorocyclicProduct} and Lemma \ref{lem:Isom_of_millefeuille}) $G$ acts cocompactly on a regular tree with valence at least $3$, fixing an end. By Lemma \ref{lem:get_HNN}, $G = \HNN(H, f)$, where $H = \Stab_G(v_0)$ and $f:H \to H$ is an injection with finite-index image at least 2. Lemma \ref{lem:Geometric_on_horosphere} says that $H$ acts geometrically on a nilpotent Lie group with left-invariant metric, which is to say that some quotient $H/K$ by a finite subgroup $K$ is a uniform lattice in $\Isom(N) = N \rtimes \Aut\Isom(N) $ (by the above theorem of Kivioja and Le Donne). Finally, the above Auslander theorem says that $H/K$ is virtually nilpotent, and therefore that $H$ itself is virtually nilpotent by Gromov's polynomial growth theorem \cite{Gromov}.
\end{proof}

We remark that this proof says something even stronger than the statement given in \cref{thm:boundary_implies_algebra}. Namely, that $H$ is nearly a uniform lattice in a nilpotent Lie group, and the map $f$ is nearly an Anosov map. 

\begin{cor}
\label{cor:thmA_part2_STRONG}
    Let $X$ and $Y$ be proper, geodesically-complete $\CAT(-\kappa)$ spaces equipped with height functions based at non-isolated points at infinity. Let $G$ be a finitely-generated group acting geometrically on $X\bowtie Y$ with metric $d_{\bowtie}$ arising from a monotone norm. Assume that exactly one of $\partial_l X$ and $\partial^uY$ is connected. Then there is a subgroup $H\subset G$ and a finite-index injection $f:H\to H$ so that $$G = \HNN(H,f).$$ There is a finite, normal subgroup $K \subset H$ invariant under $f$ so that $H/K$ contains a finite index subgroup $\Gamma$ isomorphic to a lattice in a product of simply connected nilpotent Lie groups $N_1\times N_2$. There are also one parameter families of expanding maps $\alpha_i(t) \in \Aut(N_i)$ so that the restriction of the induced map $\overline{f}:H/K\to H/k$ to $\Gamma$ agrees with the map $(\alpha_1(1)), \alpha_2(-1)$ on $N_1\times N_2$.  

\end{cor}

In this way, all groups subject to the hypotheses of Part 2 of the theorem closely resemble the groups constructed in Subsection \ref{subsec:Construct_millefeuille_example}.

\subsection{Proof of \cref{thm:boundary_implies_algebra}, Part 3: both boundaries connected.}

The proof of Part three of \cref{thm:boundary_implies_algebra} is similar to (and much simpler than) the proof of Part 2. Since both boundaries are assumed to be connected, $Z$ and $W$ are both Heintze groups. As before, Lemma \ref{lem:discrete_heights} gives a homomorphism $h:G\to \Z$, which allows us to write $G = \ker(h)\rtimes \Z$. The main difference from Part 2 is that the horosphere $h\inv_{Z\bowtie W}(0)$ is now connected, so $\ker(h)$ is finitely generated. We will then use the same argument as before to conclude that $\ker(h)$ is virtually nilpotent.

\begin{proof}[Proof of \cref{thm:boundary_implies_algebra}, Part 1]

    Let $G$ act geometrically on a horocyclic product $X\bowtie Y$. By \cref{thm:Upgrade_to_Millefeuille}, we can upgrade this action to an action on $Z\bowtie W$, where $Z,W$ are millefeuille spaces by passing to a subgroup of index at most two. Since both upper and lower boundaries are connected, $Z$ and $W$ must be Heintze groups, and in particular the horospheres of $Z\bowtie W$ are left cosets of the connected, simply connected nilpotent Lie group $h_{Z\bowtie Y}^{-1}(0) = N_1 \times N_2$. We consider $W$ as a degenerate millefeuille space, whose tree factor $T$ is a line. Let $v_0 \in T$ be the vertex at height $0$ and let $h:G \to \Z$ be the height-change homomorphism. Since $T$ is a line, $H = \Stab(v_0) = \ker(h)$ and $\pi\inv(v_0) = \beta\inv(0)$. 

    Lemma \ref{lem:Geometric_on_horosphere} still applies to this case, so $H$ acts geometrically on $\pi\inv(v_0)$, a connected, simply connected nilpotent Lie group with a left-invariant metric. As before, the Kivioja-Le Donne and Auslander theorems imply that $H$ is virtually nilpotent. 
\end{proof}

As before, the proof actually gives a stronger result: 

\begin{cor}
\label{cor:thmA_part3_STRONG}
     Let $X$ and $Y$ be proper, geodesically-complete $\CAT(-\kappa)$ spaces equipped with height functions based at non-isolated points at infinity. Let $G$ be a finitely-generated group acting geometrically on $X\bowtie Y$ with metric $d_{\bowtie}$ arising from a monotone norm. Assume that both $\partial_l X$ and $\partial^uY$ are connected. Then there is a subgroup $H\subset G$ and an automorphism $f \in \Aut(H)$ so that a finite-index subgroup of $G$ is isomorphic to $H\rtimes_f \Z$. There is a finite, normal subgroup $K \subset H$ invariant under $f$ so that $H/K$ contains a finite index subgroup $\Gamma$ isomorphic to a lattice in a product of simply connected nilpotent Lie groups $N_1\times N_2$. There are also one parameter families of expanding maps $\alpha_i(t) \in \Aut(N_i)$ so that the restriction of the induced map $\overline{f}:H/K\to H/k$ to $\Gamma$ agrees with the map $(\alpha_1(1)), \alpha_2(-1)$ on $N_1\times N_2$.  
\end{cor}

\bibliographystyle{alpha}
\bibliography{bib}

\end{document}